\documentclass[oneside]{article}
\usepackage{dsfont,amsmath,amsthm,amssymb,amsbsy,mathrsfs,fancyhdr,marvosym,wasysym,esint,geometry}
\newtheorem{theorem}{Theorem}[section]
\newtheorem{corollary}[theorem]{Corollary}

\newtheorem{lemma}[theorem]{Lemma}
\newtheorem{example}[theorem]{Example}

\newtheorem{remark}[theorem]{Remark}
\newtheorem{proposition}[theorem]{Proposition}

\renewcommand{\Im}{{\rm Im}}
\renewcommand{\Re}{{\rm Re}}
\renewcommand{\t}{\tilde}
\renewcommand{\b}{{\rm b}}

\numberwithin{equation}{section}
\newcommand{\on}[2]{\genfrac{}{}{0pt}{1}{#1}{#2}}

\newcommand{\lb}{\left\{}
\newcommand{\rb}{\right\}}
\newcommand{\lan}{\left\langle}
\newcommand{\ran}{\right\rangle}

\newcommand{\mb}[1]{\boldsymbol{#1}}

\title{Estimates for the Szeg\H{o} Projection on \\ Uniformly Finite-Type Subdomains of $\mathbb{C}^2$}
\author{Aaron Peterson}
\date{}
\begin{document}
\maketitle

\begin{abstract}
We prove precise growth and cancellation estimates for the Szeg\H{o} kernel of an unbounded model domain $\Omega\subset\mathbb{C}^2$ under the assumption that ${\rm b}\Omega$ satisfies a uniform finite-type hypothesis. Such domains have smooth boundaries which are not algebraic varieties, and therefore admit no global homogeneities that allow one to use compactness arguments in order to obtain results. As an application of our estimates, we prove that the Szeg\H{o} projection $\mathbb{S}$ of $\Omega$ is exactly regular on the non-isotropic Sobolev spaces $NL_k^p({\rm b}\Omega)$ for $1<p<{+\infty}$ and $k=0,1,\ldots$, and also that $\mathbb{S}:\Gamma_\alpha (E)\rightarrow \Gamma_\alpha({\rm b}\Omega)$, for $E\Subset {\rm b}\Omega$ and $0<\alpha<+\infty$, with a bound that depends only on ${\rm diam}(E)$, where $\Gamma_\alpha$ are the non-isotropic H\"older spaces.
\end{abstract}

%
%

\section{Introduction}

Let $\Omega\subset \mathbb{C}^2$ be a pseudoconvex domain with smooth boundary ${\rm b}\Omega$, and give ${\rm b}\Omega$ an appropriate measure $dm_{{\rm b}\Omega}$.
The purpose of this paper is to study the Szeg\H{o} projection $\mathbb{S}:L^2({\rm b}\Omega)\to \mathcal{H}^2(\Omega)$ when $\Omega$ belongs to a class of unbounded finite-type model domains for which ${\rm b}\Omega$ is not an algebraic variety. 
Our primary motivation is to discern how the Szeg\H{o} projection behaves in unbounded domains which lack homogeneity.
We prove two types of results.
First we show that the Szeg\H{o} kernel of such domains is smooth off of the diagonal and satisfies scale-invariant differential and cancellation estimates.
Second, we establish the exact regularity of $\mathbb{S}$ on the non-isotropic $L^p$-Sobolev spaces and non-isotropic H\"older spaces associated to ${\rm b}\Omega$.

In this paper, we restrict our attention to domains of the form $$\Omega=\{ \mb{z}=(z,z_2)\in\mathbb{C}^2\ :\ \Im(z_2)>P(z)\},$$ where $P:\mathbb{C}\to\mathbb{R}$ is a smooth, subharmonic, non-harmonic function such that $h:=\Delta P$ satisfies, for some constants $C_1,C_2$: 
\begin{align*}
	{\rm (H1)}  &\quad  \mbox{There exists}\ m\in\mathbb{N}\ \mbox{such that}\ 0<C_1 \leq \sup_{|\nu|=1} \displaystyle\sum_{j=2}^m \Big| \nabla_\nu^{j-2}h(z)\Big|\ \mbox{for all}\ z\in\mathbb{C}.\\
	{\rm (H2)}  &\quad  \|h\|_{C^k(\mathbb{C})} <+\infty\ \mbox{for}\ k=0,1,2,\ldots.\\
	{\rm (H3)}  &\quad   \sup_{z\in \mathbb{C},\ r\in [1,+\infty)} \Big| \int_{1\leq |\eta|\leq r} \frac{h(z+\eta)}{\eta^2} dm(\eta)\Big|\leq C_2.
\end{align*}
In this case, we say that ${\rm b}\Omega$ satisfies a \emph{uniform finite-type} (UFT) hypotheses of order $m$.
Several concrete examples of UFT domains are discussed in Section \ref{sec:examples}.
To avoid degeneracy issues introduced at infinity by the unboundedness of the domain, we take $dm_{{\rm b}\Omega}=dm(z,\Re(z_2))$ to be the Lebesgue measure that $${\rm b}\Omega=\{ (z,t+iP(z))\in \mathbb{C}^2\ :\ (z,t)\in\mathbb{C}\times\mathbb{R}\}$$ receives under its identification with $\mathbb{C}\times\mathbb{R}$.
The precise definition of the Szeg\H{o} projection $\mathbb{S}$---which maps $L^2({\rm b}\Omega)$ onto the closed subspace $\mathcal{H}^2(\Omega)$ of $L^2$-boundary values of holomorphic functions in $\Omega$---is recalled in Appendix \ref{sec:sz}.

Our results are expressed in terms of a non-isotropic geometry on ${\rm b}\Omega$ closely related to the tangential Cauchy-Riemann vector fields $$Z=\partial_{z}+2i P_z(z)\partial_{z_2},\ \bar{Z}=\partial_{\bar{z}}-2i P_{\bar{z}}(z)\partial_{\bar{z}_2},$$ which we now describe.
Writing $X=Z+\bar{Z}$ and $Y=-i(Z-\bar{Z})$, condition (H1) quantitatively expresses that the real vector fields $X$ and $Y$ satisfy H\"ormander's finite-type condition of order $m$ uniformly over ${\rm b}\Omega$, which ensures that any two points $\mb{z},\mb{w}\in{\rm b}\Omega$ can be connected by a piecewise $C^1$ path which is almost everywhere tangent to $X$ or $Y$.
If we equip $\b \Omega$ with a metric such that $\lan X, Y\ran =0$ and $\|X\|=\|Y\|=1$ at each point, then the infimal length $d(\mb{z},\mb{w})$ of such a path is called the Carnot-Carath\'eodory distance between $\mb{z}$ and $\mb{w}$.
One can understand the balls $B_d(\mb{z},\delta)=\{\mb{w}\in \b\Omega\ :\ d(\mb{z},\mb{w})<\delta\}$ as `twisted' ellipsoids of radius $\delta$ in the directions of $X$ and $Y$, and radius $\Lambda(\mb{z},\delta)$ in the direction of $T=\partial_{z_2}+\partial_{\bar{z}_2}$.
For small $\delta$, $\Lambda(\mb{z},\delta)$ is essentially a polynomial in $\delta$ with coefficients that depend on $z$, while for large $\delta$ we have $\Lambda(\mb{z},\delta)\approx \delta^2$ uniformly in $\mb{z}$.
The volume of $B_d(\mb{z},\delta)$ satisfies
\begin{equation}
\label{eq:ballvolume}|B_d(\mb{z},\delta)|\approx \delta^2\Lambda(\mb{z},\delta).
\end{equation}
In Section \ref{sec:geom} we give a more detailed discussion of $d$, but additional background can be found in \cite{NagelSteinWainger1985} for the case $\delta\lesssim 1$ and \cite{Peterson2014} for the case $\delta\gtrsim 1$.

Before stating our results, we give two small pieces of notation.
We denote by $[\mathbb{H}]$ the Schwartz kernel of an operator $\mathbb{H}$ which maps test functions into the space of distributions.
Also, if $\alpha$ is a multi-index then we write $Z_{\mb{\eta}}^\alpha=Z_{1}\cdots Z_{|\alpha|}$, where each $Z_i\in \{Z,\bar{Z}\}$ and acts in the $\mb{\eta}$-variables.

Our first result gives precise size and cancellation estimates on the Szeg\H{o} kernel and its (non-isotropic) derivatives. 

\begin{theorem}[Growth/Cancellation Estimates]\label{thm:szego-decomp}
	The Szeg\H{o} projection $\mathbb{S}$ can be written as the sum of two operators $\mathbb{S}=\mathbb{N}+\mathbb{F}$ such that, for all multi-indices $\alpha,\beta$ and all $N,M\geq 0$, the following hold.
	\begin{enumerate}
		\item[(a)] For $0\leq K+K'<|\alpha|+|\beta|+4$ there exist multi-indices $\gamma,\gamma'$ of length $K,K'$ and an operator $\mathbb{N}_{K,K'}$ such that
		\begin{itemize}
			\item[(i)] $\mathbb{N}=Z^\gamma\mathbb{N}_{K,K'}Z^{\gamma'}$, and
			\item[(ii)] For some $C=C(\alpha,\beta,N,M,K,K',h),$
		\end{itemize}
		\begin{align}\label{eq:szego-near-growth} |T^N&Z_{\mb{z}}^\alpha (Z^{\beta}_{\mb{w}})^{\ast}[\mathbb{N}_{K,K'}](\mb{z},\mb{w})|\\
		 &\leq C \frac{d(\mb{z},\mb{w})^{K+K'-|\alpha|-|\beta|}\Lambda(\mb{z},d(\mb{z},\mb{w}))^{-N}}{|B_d(\mb{z},d(\mb{z},\mb{w}))|}(1+\Lambda(\mb{z},d(\mb{z},\mb{w})))^{-M}.\nonumber\end{align}
		\item[(b)] For $0\leq K+K'<\min(|\alpha|,2)+\min(|\beta|,2)+4$ there exist multi-indices $\gamma,\gamma'$ of length $K,K'$ and an operator $\mathbb{F}_{K,K'}$ such that
		\begin{itemize}
			\item[(i)] $\mathbb{F}=Z^\gamma\mathbb{F}_{K,K'}Z^{\gamma'}$,
			\item[(ii)] For some $C=C(\alpha,\beta,N,K,K',h),$
		\end{itemize}
		\begin{align}\label{eq:szego-far-growth} |T^N &Z_{\mb{z}}^\alpha (Z_{\mb{w}}^{\beta})^\ast[\mathbb{F}_{K,K'}](\mb{z},\mb{w})|\\
		& \leq C \min\Bigg(1,\frac{d(\mb{z},\mb{w})^{K+K'-\min(|\alpha|,2)-\min(|\beta|,2)}\Lambda(\mb{z},d(\mb{z},\mb{w}))^{-N}}{|B_d(\mb{z},d(\mb{z},\mb{w}))|}\Bigg).\nonumber\end{align}
		\item[(c)] The conclusions of (a) and (b) also hold for $\mathbb{N}^\ast$ and $\mathbb{F}^\ast$, respectively, with the same qualifications.
	\end{enumerate}
\end{theorem}

The cancellation properties of the Szeg\H{o} projection are captured by parts (a)-(i) and (b)-(i), which allow us to view $\mathbb{S}$ as a derivative.
Theorem \ref{thm:szego-decomp} implies that $\mathbb{S}$ can be analyzed via the Calder\'on-Zygmund paradigm (see \cite{Street2014}). Indeed, with some additional work we obtain the following classical cancellation and growth estimates for the Szeg\H{o} kernel.

\begin{corollary}[Classical Growth Estimates]\label{cor:szego-growth}
	Fix $\mb{z},\mb{w}\in {\rm b}\Omega$.
	For multi-indices $\alpha$ and $\beta$ with $|\alpha|,|\beta|\leq 2$, and for $N\geq 0$,
	\begin{equation}\label{eq:szego-growth}
		|T^N Z_{\mb{z}}^\alpha (Z_{\mb{w}}^\beta)^\ast [\mathbb{S}](\mb{z},\mb{w})|\leq C\frac{d(\mb{z},\mb{w})^{-|\alpha|-|\beta|} \Lambda(\mb{z},d(\mb{z},\mb{w}))^{-N}}{|B_d(\mb{z},d(\mb{z},\mb{w}))|},\end{equation}
	where $C=C(N,\alpha,\beta,h)$ is independent of $\mb{z}$ and $\mb{w}$. 
	The restriction $|\alpha|,|\beta|\leq 2$  is sharp in the sense that the above estimate may fail to hold if either $|\alpha|\geq 3$ or $|\beta|\geq 3$.
\end{corollary}
\bigskip

\begin{corollary}[Classical Cancellation Estimates]\label{cor:szego-cancellation}
	Let $|\alpha|\leq 2$, and $N\geq 0$. 
	If $\phi$ is a smooth function with support in $B_d(\mb{z},\delta)$, then there is $C=C(N,\alpha,h)$ and $M=M(N,\alpha)$ such that
	\begin{align*}\|T^N &Z^\alpha\mathbb{S}[\phi]\|_{L^\infty (B_d(\mb{z},\delta))} \\
	&\leq C\Lambda(\mb{z},\delta)^{-N}\delta^{-|\alpha|}(\|\phi\|_{L^\infty(B_d(\mb{z},\delta))} + \|(\Lambda(\mb{z},\delta)T)^{M}\phi\|_{L^\infty(B_d(\mb{z},\delta))}).\end{align*}
\end{corollary}

Although the estimates in Corollary \ref{cor:szego-growth} follow immediately from those in Theorem \ref{thm:szego-decomp}, sharpness will follow from examining the situation on tube domains, where an explicit formula for $[\mathbb{S}]$ is available; see Example \ref{ex:examples-tube} and Section \ref{sec:szego-growth} for the details.

We use the estimates from Theorem \ref{thm:szego-decomp} to prove that the Szeg\H{o} projection exactly preserves non-isotropic Sobolev and H\"older regularity in the following sense.

\begin{theorem}\label{thm:szego-mapping}
	$\mathbb{S}$ has the following mapping properties.
	\begin{itemize}
		\item[(a)] $\mathbb{S}:NL_k^p({\rm b}\Omega)\rightarrow NL_k^p({\rm b}\Omega)$ for $1<p<{+\infty}$, $k=0,1,\ldots$, where $NL^p_k({\rm b}\Omega)$ are the non-isotropic Sobolev spaces on ${\rm b}\Omega$ associated to the vector fields $Z$ and $\bar{Z}$.
		\item[(b)] For every Carnot-Carath\'eodory ball $E=B_d(\mb{z}_0,\delta_0)\subset{\rm b}\Omega$, $\mathbb{S}:\Gamma_\alpha(E) \rightarrow \Gamma_\alpha ({\rm b}\Omega)$ for $0<\alpha<+\infty$, where $\Gamma_\alpha(U)$ are the non-isotropic H\"older spaces of functions supported on $U\subset{\rm b}\Omega$ associated to the vector fields $Z$ and $\bar{Z}$. Here, the operator norm depends only on $\alpha$, $\delta_0$, and the constants in (H1), (H2), and (H3).
	\end{itemize}
\end{theorem}

The spaces $NL_k^p({\rm b}\Omega)$ and $\Gamma_\alpha(U)$ are defined in detail in Section \ref{sec:notation}.

The regularity of $\mathbb{S}$ and $[\mathbb{S}]$ on domains with smooth, finite-type boundary has been extensively studied, and is well-understood in situations where the domain is bounded and the geometry is well-behaved.
When $\Omega\Subset \mathbb{C}^2$ is a smoothly bounded weakly pseudoconvex domain of finite-type, Nagel, Rosay, Stein, and Wainger \cite{NagelRosaySteinWainger1989} showed that $[\mathbb{S}](\mb{z},\mb{w})$ is smooth on $$(\bar{\Omega}\times{\rm  b}\Omega)\backslash \lb (\mb{z},\mb{w})\in \b\Omega\times\b\Omega \ :\ \mb{z}=\mb{w}\rb$$ and satisfies estimates similar to those in Corollaries \ref{cor:szego-growth} and \ref{cor:szego-cancellation}, and they obtained results analogous to Theorem \ref{thm:szego-mapping}.
When $\Omega\Subset \mathbb{C}^n$ ($n\geq 3$), similar results were proved by Kor\'anyi and V\'agi \cite{KoranyiVagi1969} on the unit ball, Stein \cite{Stein1972} on strongly pseudoconvex domains, Fefferman, Kohn, and Machedon \cite{FeffermanKohnMachedon1990} on diagonalizable domains, McNeal and Stein \cite{McNealStein1997} on convex domains, and Koenig \cite{Koenig2002} when the Levi form has pointwise-comparable eigenvalues. This culminated in the work of Charpentier and Dupain \cite{CharpentierDupain2014} for geometrically separated domains, which contains all of the previously mentioned cases.

When $\Omega=\{\mb{z}\in \mathbb{C}^n\ :\ \Im(z_n)>P(z_1,\ldots,z_{n-1})\}$ is an unbounded model domain, our knowledge is essentially restricted to cases where we either have an explicit formula for $[\mathbb{S}]$, or where $P$ is a polynomial. 
When $n=2$, explicit formulas for the Szeg\H{o} kernel were obtained by Greiner and Stein \cite{GreinerStein1978} when $P(z)=|z|^{2k}$, by Nagel \cite{Nagel1986} when $P(z)=b(\Re(z))$ is a convex function, and by Haslinger \cite{Haslinger1994} for $P(z)=|z|^a$, $a\geq 2$.
When $n=3$, similar formulas were obtained by Francsics and Hanges \cite{FrancsicsHanges1995}.
Several authors have leveraged these formulas to answer various questions related to the Szeg\H{o} projection and kernel; for examples, see \cite{Kang1989,ChristGeller1992,Haslinger1995,Kamimoto2001,HalfpapNagelWainger2010,GilliamHalfpap2014}.

In the special case where $P$ is a subharmonic, nonharmonic polynomial, full estimates of the type given in Corollary \ref{cor:szego-growth}, Corollary \ref{cor:szego-cancellation}, and Theorem \ref{thm:szego-mapping} were proved for $n=2$ in \cite{NagelRosaySteinWainger1988,NagelRosaySteinWainger1989}, while limited results for special examples in the case $n\geq 3$ are also known \cite{FollandStein1974,Machedon1988,Benguria2017}.
The critical fact in this case is that the class of polynomial model domains in $\mathbb{C}^2$ is highly amenable to study because it is homogeneous, in the sense that there is a large family of affine non-isotropic dilations of $\mathbb{C}^2$ that preserve it.
One is therefore able to effectively `normalize' a polynomial domain by translating any large- or small-scale data to unit scale.
The class of such domains (for a fixed degree $m$) is parametrized by a compact set, which plays a big role in the analysis.

Once one breaks this homogeneity, though, many of the standard techniques fall apart.
As observed by McNeal in \cite{McNeal1994}, except in certain special situations these scaling arguments do not suitably generalize to domains in higher dimensions, even for polynomial domains. For large classes of smooth finite-type pseudoconvex subdomains of $\mathbb{C}^n$ for $n\geq 3$, the standard scaling techniques destroy either the smoothness of the boundary or the finite-type assumption.
This has been a major obstruction to the study of the Szeg\H{o} projection (and the $\bar{\partial}_b$-problem in general) in higher dimensions.

UFT domains (for $m>2$) furnish a situation where the standard scaling arguments fail, and our main task in this work is to develop techniques for studying the Szeg\H{o} kernel in the absence of homogeneity.
We accomplish this by extending an idea of Raich \cite{Raich2006b}, who explored the link between non-isotropic smoothing operators on polynomial model domains in $\mathbb{C}^2$ and one-parameter families of operators on $\mathbb{C}$ which satisfy uniform estimates.
The parameter here comes from taking the partial Fourier transform in the $\Re(z_2)$-variable, and can be thought of as decoupling the operator $\bar{Z}$ into a family of operators on $\mathbb{C}$ which capture a single scale of the $\Re(z_2)$-variable.
In our case, we tie $\mathbb{S}$ to a one-parameter family of weighted $\bar{\partial}$ operators on $\mathbb{C}$. 
We then build off of the work of Christ \cite{Christ1991a} to estimate each operator in the resulting family.
These estimates are then pieced together with an inverse partial Fourier transform.

This work is an extension of my Ph.D. thesis at the University of Wisconsin - Madison, and was partially supported by NSF Grant No. 1147523 - RTG: Analysis and Applications. 
It is a great pleasure to thank my advisor, Alexander Nagel, for his support and guidance throughout this project.
I would also like to thank the referee, whose suggestions greatly improved the paper.

\subsection{Examples}
\label{sec:examples}

Before diving into the argument, let's pause to explore a few basic classes of UFT domains. 

\begin{example}{\rm 
		Perhaps the most basic example of a UFT domain is the upper half-space
		$$U^1 = \{ \mb{z}\in\mathbb{C}^2\ :\ \Im(z_2)>|z|^2\},$$
		the boundary of which is the one-dimensional Heisenberg group $\mathbb{H}^1$. Here $P(z)=|z|^2$ and $h(z)\equiv 4$. In Remark \ref{rem:kernels-heisenberg} we explain how the fact that $h(z)$ is constant allows us to replace estimate (\ref{eq:szego-far-growth}) with
		$$|T^N Z_{\mb{z}}^\alpha (Z_{\mb{w}}^{\beta})^{\ast}[\mathbb{F}_{K,K'}](\mb{z},\mb{w})|\leq C \min\Bigg(1,\frac{d(\mb{z},\mb{w})^{K+K'-|\alpha|-|\beta|}\Lambda(\mb{z},d(\mb{z},\mb{w}))^{-N}}{|B_d(\mb{z},d(\mb{z},\mb{w}))|}\Bigg),$$
		and relax the restriction on $K+K'$ in part (b) of Theorem \ref{thm:szego-decomp} to $0\leq K+K'<|\alpha|+|\beta|+2N+4$.  This in turn replaces inequality (\ref{eq:szego-growth}) with $$|T^N Z_{\mb{z}}^\alpha (Z_{\mb{w}}^\beta)^\ast [\mathbb{S}](\mb{z},\mb{w})|\leq C\frac{d(\mb{z},\mb{w})^{-|\alpha|-|\beta|} \Lambda(\mb{z},d(\mb{z},\mb{w}))^{-N}}{|B_d(\mb{z},d(\mb{z},\mb{w}))|},\qquad N,|\alpha|,|\beta|\geq 0,$$
		which are the known size estimates for $\mathbb{S}$ on $\mathbb{H}^1$; see for example \cite{NagelRosaySteinWainger1988}. Similarly, the result in Corollary \ref{cor:szego-cancellation} holds for all $|\alpha|\geq 0$.}\end{example}

Up to adding a degree 2 harmonic polynomial and scaling, $P(z)=|z|^2$ is the only subharmonic, non-harmonic polynomial that satisfies hypothesis (H2) (and therefore yields a UFT domain). 
To see the richness of the class UFT domains we therefore need to consider general subharmonic functions $P(z)$ for which $h(z)$ satisfies (H1)-(H3).
We reduce our search for such functions $P(z)$ to a search for $h(z)=\Delta P(z)$ by noting that for a given function $h(z)$ satisfying (H2), (H3) is equivalent to the existence of a `nice' class of subharmonic potentials $P(z)$ for $h(z)$.

\begin{proposition}\label{prop:potentials}{ Let $h:\mathbb{C}\to \mathbb{R}$ be a non-negative, smooth function with\newline $\|h\|_{C^k(\mathbb{C})}<+\infty$ for $k=0,1,2,\ldots.$ Then the following are equivalent:
		\begin{itemize}
			\item[(a)] $\displaystyle\sup_{\zeta\in\mathbb{C}}\ \sup_{r\in [1,+\infty)} \Big| \int_{1\leq |\eta|\leq r} \frac{h(\zeta+\eta)}{\eta^2} dm(\eta)\Big| = A_0<+\infty$.
			\item[(b)] There exist constants $A_1,A_2,\ldots$ such that for every fixed $\zeta\in \mathbb{C}$ there exists $P:\mathbb{C}\to \mathbb{R}$, with $P(0)=0$, such that
			\subitem(i) $\Delta P(z)=h(\zeta+z)$ for all $z\in \mathbb{C}$,
			\subitem(ii) $|\nabla P(z)| \leq A_1|z|$,
			\subitem(iii) $\|\nabla^k P\|_\infty \leq A_k$ for $k=2,3,\ldots.$
		\end{itemize}
	}\end{proposition}
	
We prove of Proposition \ref{prop:potentials} in Section \ref{sec:normalization-potentials}.
As an immediate application, we identify two classes of functions $h(z)$ that satisfy (H3).
	
\begin{proposition}\label{prop:potentialsexamples}Let $h:\mathbb{C}\to [0,+\infty)$ satisfy (H2). Then $h$ satisfies (H3) if either
	\begin{itemize}
		\item[(a)] $h(z)=h(\Re(z))$ for all $z\in \mathbb{C}$ (i.e. if $\Omega$ is a tube domain), or
		\item[(b)] There exist constants $A\geq 0$, $B>0$, and $C>0$ so that, uniformly in $z\in \mathbb{C}$ and $r\geq 1$, $$\int_{|\eta|\leq r} |h(\eta+z)-A|dm(\eta)\leq Br^{2-C}.$$
	\end{itemize}
\end{proposition}
\begin{proof}
	To show that (a) implies (H3), we merely note that if $\zeta\in \mathbb{C}$, then 
	$$P(z)=P(x+iy) := \int_0^x \int_0^r h(s+\Re(\zeta))dsdr$$ satisfies (b) of Proposition \ref{prop:potentials}, and therefore $h$ satisfies (H3).
		
	On the other hand, if condition (b) holds, then note that for $K$ such that $2^{K-1}\leq r\leq 2^K$,
	\begin{align*}
		\Big| \int_{1\leq |\eta|\leq r} \frac{h(\eta+z)}{\eta^2}dm(\eta)\Big| & = \Big| \int_{1\leq |\eta|\leq r} \frac{h(\eta+z)-A}{\eta^2}dm(\eta)\Big| \\
		& \leq  \int_{1\leq |\eta|\leq r} \frac{|h(\eta+z)-A|}{|\eta|^2}dm(\eta) \\
		& \leq  \sum_{k=1}^K \int_{2^{k-1}\leq |\eta|\leq 2^k} \frac{|h(\eta+z)-A|}{|\eta|^2}dm(\eta) \\
		& \leq  \sum_{k=1}^K 2^{-2(k-1)}\int_{2^{k-1}\leq |\eta|\leq 2^k} |h(\eta+z)-A|dm(\eta) \\
		& \leq   4\sum_{k=1}^K 2^{-2k}\int_{|\eta|\leq 2^k} |h(\eta+z)-A|dm(\eta) \\
		& \leq  4\sum_{k=1}^K 2^{-2k} B(2^k)^{2-C} \\
		& \leq  4B \sum_{k=1}^K 2^{-kC},
	\end{align*}
	which is bounded by a constant that depends only on $B$ and $C$ (and not on $r$ or $z$). In the first line above we used the fact that for $r\geq 1$,
	$$\int_{1\leq |\eta|\leq r} \frac{1}{\eta^2}dm(\eta) = \int_1^r \int_0^{2\pi} \frac{1}{s}e^{-2i\theta}d\theta ds = 0.$$
	This shows that (H3) holds, and we are done.
\end{proof}
	
\begin{remark}{\rm 
			The interesting case here is $C\leq 2$, because $C>2$ implies that $h$ is constant.
	}\end{remark}
		
\begin{example}{\rm Condition (b) in Proposition \ref{prop:potentialsexamples} is a quantitative way of saying that $h$ is well approximated on large scales by a constant. This holds, for example, when $h(z)=\chi(\Delta Q(z))$, where $Q$ is a subharmonic, non-harmonic polynomial and $\chi$ is a smooth, non-decreasing function with $\chi(t)\equiv t$ for $t\leq 1$ and $\chi(t)\equiv \frac{3}{2}$ for $t\geq 2$. Because subharmonic, non-harmonic polynomials $Q(z)$ satisfy (H1), such functions $\chi(\Delta Q(z))$ give rise to UFT domains.
	}\end{example}

\begin{example}\label{ex:examples-tube}{\rm 
	The Szeg\H{o} kernel for tube domains of the form $$\Omega=\{ \mb{z}\in\mathbb{C}^2\ :\ \Im(z_2)>b(\Re(z)),\quad b:\mathbb{R}\to \mathbb{R}\ \mbox{convex}\}$$ are particularly amenable to study due to the translation invariance of $\Omega$ in the $\Im(z)$-direction.
	This invariance was exploited by Nagel in \cite{Nagel1986}, who showed that for such $\Omega$ the Szeg\H{o} kernel has the form
	\begin{equation}\label{szegoformula}
		[\mathbb{S}](\mb{z},\mb{w}) = \frac{1}{4\pi^2} \displaystyle\int_0^{+\infty} \displaystyle\int_\mathbb{R} \displaystyle\frac{e^{i\tau(z_2-\bar{w}_2)+\eta(z+\bar{w})}}{\displaystyle\int_\mathbb{R} e^{2[\eta \theta - \tau b(\theta)]}d\theta}d\eta d\tau.
	\end{equation}
	The explicit nature of this formula has facilitated the study of the Szeg\H{o} kernel on tube domains.
	For a discussion on the history of this formula, see \cite{HalfpapNagelWainger2010}.
						
	Formula (\ref{szegoformula}) allows us to exhibit the sharpness claims in Corollary \ref{cor:szego-growth} by explicitly studying $[\mathbb{S}]$ for one particular (and rather nicely behaved) convex function $b(x)$ that satisfies (H1) and (H2). Indeed, if $b:\mathbb{R}\to [0,+\infty)$ is chosen so that
	\begin{itemize}
		\item $b(0)=b'(0)=0$, 
		\item $b''(x) = e^{x-n}$ in a neighborhood of $x=n$, for all $n\in \mathbb{Z}$, and
		\item $0<a\leq b''(x)\leq A<+\infty$ for some constants $a,A$ uniformly in $x\in \mathbb{R}$,
	\end{itemize}
	then  for $k\geq 0$ there exists $C=C(k)>0$ so that if $\mb{z}_n=(n,ib(n))$, then
	$$|\bar{Z}^k Z[\mathbb{S}](\mb{z}_n,\mb{z}_{-n})| \geq C\frac{d(\mb{z}_n,\mb{z}_{-n})^{-2}}{|B_d(\mb{z}_n,d(\mb{z}_n,\mb{z}_{-n}))|},\quad n\in \mathbb{Z}.$$
	Details are given in Section \ref{sec:szego-growth}.
	}\end{example}

\begin{remark}{\rm 
		There exist functions $h$ that satisfy (H1) and (H2), but for which (H3) does not hold, and therefore the conclusions of Proposition \ref{prop:potentials} do not hold.
		
		For one such example, consider the smooth function $h(z)$ defined by
		$$h(re^{i\theta})=1+\chi(r)f(\theta),$$ 
		where $f:[-\pi,\pi]\rightarrow [0,1]$ is smooth and supported in $[-\frac{1}{100},\frac{1}{100}]$ with $f(0)=1$, and $\chi$ is a smooth, non-decreasing function with $\chi(r)\equiv 0$ if $r\leq 1$ and $\chi(r)\equiv 1$ for $r\geq 2$.
		For large $|z|$ the arguments in the proof of Proposition \ref{prop:potentials} show that, for a particular subharmonic $\tilde{P}$ with $\Delta \tilde{P}=h$, $|\nabla{\tilde{P}}|\approx |z|\log|z|$. 
		It follows that the estimates in part (b) of that proposition fail to hold for every subharmonic $P$ with $\Delta P=h$.
	}\end{remark}

\subsection{Definitions and Notation}

\label{sec:notation}

As in the introduction, let $\Omega=\{ \mb{z}\in\mathbb{C}^2\ :\ \Im(z_2)>P(z)\}$, where $P:\mathbb{C}\to \mathbb{R}$ is smooth, subharmonic, and non-harmonic.
The space of tangential antiholomorphic vector fields $T^{0,1}({\rm b}\Omega)$ on ${\rm b}\Omega$ is spanned by $\bar{Z}_{{\rm b}\Omega}=\partial_{\bar{z}}-2i P_{\bar{z}}(z)\partial_{\bar{z}_2}$, while the space of tangential holomorphic vector fields $T^{1,0}({\rm b}\Omega)$ on ${\rm b}\Omega$ is spanned by $Z_{{\rm b}\Omega}=\partial_z + 2i P_z(z)\partial_{z_2}.$
When no confusion can arise, we will omit the subscript ${\rm b}\Omega$.

We identify $(z,t+iP(z))\in {\rm b}\Omega$ with $(z,t)\in \mathbb{C}\times \mathbb{R}$ via the diffeomorphism $\Pi:{\rm b}\Omega\to \mathbb{C}\times\mathbb{R}$ given by $\Pi(z,z_2)=(z,\Re(z_2))$.
Under this identification, $\bar{Z}$ and $Z$ become, respectively, $$\bar{Z}=\partial_{\bar{z}}-iP_{\bar{z}}(z)\partial_t\quad \mbox{and}\quad Z=\partial_z + i P_z(z)\partial_t.$$
Give $\Omega$ the standard Lebesgue measure $dm_\Omega=dm(z,z_2)$ that it receives as a subset of $\mathbb{C}^2$, and
${\rm b}\Omega$ the Lebesgue measure $dm_{{\rm b}\Omega}=dm(z,\Re(z_2))=\Pi^\ast dm(z,t)$ that it receives from its identification with $\mathbb{C}\times\mathbb{R}.$ 
As above, we will omit the subscript when no confusion can arise.

Letting $\mathcal{O}(\Omega)$ denote the space of holomorphic functions on $\Omega$, we define the Hardy Space
$$\mathcal{H}^2(\Omega)=\{ F\in \mathcal{O}(\Omega)\ :\ \|F\|^2_{\mathcal{H}^2(\Omega)} = \sup_{\epsilon>0} \int_{\mathbb{C}\times\mathbb{R}} |F_\epsilon (z,t)|^2 dm(z,t) < +\infty\},$$
where $F_\epsilon(z,t):= F(z,t+iP(z)+i\epsilon)$.
We can identify $\mathcal{H}^2(\Omega)$ with the (closed) subspace of $L^2({\rm b}\Omega)$ defined by $$B(\b\Omega)=\{ f\in L^2({\rm b}\Omega)\ :\ \bar{Z}f\equiv 0\ \mbox{as distributions}\},$$
the $L^2(\b\Omega)-$nullspace of $\bar{Z}$, and therefore view $$\mathbb{S}:L^2({\rm b}\Omega)\to B(\b\Omega)\cong \mathcal{H}^2(\Omega)$$ as the orthogonal projection of $L^2({\rm b}\Omega)$ onto the null-space of $\bar{Z}_{{\rm b}\Omega}$; see Appendix \ref{sec:sz} for more details.

For a function $f:\mathbb{C}\to \mathbb{C}$, the symbol $\nabla f$ will denote a generic first-order partial derivative of $f$, while $\nabla^k f$ denotes a generic $k$-th order partial derivative of $f$.
For $|\eta|=1$ we write $\nabla_\eta f$ to denote the derivative of $f$ in the direction of $\eta$.

For $1<p<+\infty$, we say that $f\in NL^p_k({\rm b}\Omega)$ if 
$$\displaystyle \|f\|_{NL^p_k({\rm b}\Omega)}:=\sum_{0\leq|\alpha|\leq k} \|W^\alpha f\|_{L^p({\rm b}\Omega)} <+\infty,$$
where $W^\alpha$ is an $|\alpha|$-order mixed derivative in the vector fields $Z$ and $\bar{Z}$.

For $U\subset {\rm b}\Omega$, the non-isotropic H\"older space $\Gamma_\alpha(U)$ associated to $U$ is defined as follows.
For $0<\alpha<1$ and $k=0,1,2,\ldots$,
\begin{align*}
	\|f\|_{\Gamma_{\alpha+k}(U)}  :=   \inf \{ A\ :\ \mbox{for every}\ \mb{z},\mb{w}\in U\ &\mbox{and}\ |\beta|\leq k,\ \|W^\beta f\|_{L^\infty(U)}\leq A\ \mbox{and}\ \\
	 & |W^\beta f(\mb{z})-W^\beta f(\mb{w})|\leq A d(\mb{z},\mb{w})^\alpha\}.\end{align*}
Now, say that $f\in \Gamma_1(U)$ if 
\begin{equation}\label{eq:gamma1}
	\mbox{we can write}\ \ f=\displaystyle\sum_{k=0}^{+\infty} f_k\quad \mbox{with}\quad \|W^\beta f_k\|_{L^\infty(U)} \leq A 2^{-k}2^{|\beta|k}\quad \mbox{for}\quad |\beta|\leq 2.\end{equation}
We define 
\[\|f\|_{\Gamma_1(U)}:=\inf \{ A\ :\ (\mbox{\ref{eq:gamma1}})\ \mbox{holds}\},\]
and for integer $\alpha>1$ we define $\|f\|_{\Gamma_\alpha(U)}:= \displaystyle\sum_{|\beta|<\alpha} \|W^\beta f\|_{\Gamma_1(U)}$.

Throughout the paper, we will write $A\lesssim B$ to mean that there is a constant $0<C<+\infty$, independent of all relevant parameters, such that $A\leq CB$.
Similarly, write $A\gtrsim B$ when $B\lesssim A$, and $A\approx B$ if $A\lesssim B$ and $B\lesssim A$.

\section{Outline of the Argument}
\label{sec:outline}

In this section we describe the techniques used to prove Theorem \ref{thm:szego-decomp}, and outline the structure of the paper.

The first step in our argument is to exploit the $\Re(z_2)$-translation invariance of $\bar{Z}$ by taking the partial Fourier transform in the $\Re(z_2)$-variable (see for example \cite{Nagel1986}).
For Schwartz functions $f$ on $\mathbb{C}\times \mathbb{R}$, this is defined via
$$\hat{f}(z,\tau)=\mathcal{F}[f](z,\tau)=\int_{\mathbb{R}} e^{-2\pi i \tau t}f(z,t)dt.$$
This allows us to formally\footnote{The Szeg\H{o} projection $\mathbb{S}: L^2({\rm b}\Omega)\to B(\b\Omega)$ can be written as $\mathbb{S}=\displaystyle\lim_{\epsilon\to 0^+} \mathbb{S}^\epsilon$ in the sense of tempered distributions on ${\rm b}\Omega\times {\rm b}\Omega$, where the operators $\mathbb{S}^\epsilon:L^2({\rm b}\Omega)\to B(\b\Omega)$ are defined by $\mathbb{S}^\epsilon [f] = (\mathbb{S}[f])^\epsilon$ as in Appendix \ref{sec:sz}. 
	The Cauchy integral formula and Proposition \ref{prop:kernels-hnw24} in Appendix \ref{sec:sz} imply that the $\mathbb{S}^\epsilon$ have $C^\infty$ Schwartz kernels. 
	Indeed, we have
	$$[\mathbb{S}^\epsilon](\mb{z},\mb{w}) = \int_0^{+\infty} e^{-2\pi \epsilon\tau} e^{2\pi i (\Re(z_2)-\Re(\bar{w}_2))\tau}[\hat{\mathbb{S}}](z,w,\tau)d\tau.$$
	We will prove Theorem \ref{thm:szego-decomp} for $\mathbb{S}^\epsilon$, although all constants that appear in our estimates are independent of $\epsilon>0$. 
	The structure of our argument will allow us to obtain the results for $\mathbb{S}$ by taking $\epsilon\to 0$. For the ease of notation, however, we will omit the $\epsilon$ from all computations.} write 
$$[\mathbb{S}](\mb{z},\mb{w}) = \int_0^{+\infty} e^{2\pi i\tau (\Re(z_2)-\Re(\bar{w}_2))} [\hat{\mathbb{S}}](z,w,\tau)d\tau,$$
where $\hat{\mathbb{S}}= \mathcal{F}\circ(\Pi^{-1})^\ast\circ \mathbb{S}\circ \Pi^\ast \circ \mathcal{F}^{-1}$.
We also have
$$\hat{\bar{Z}}=\partial_{\bar{z}}+2\pi \tau P_{\bar{z}}=:\bar{D}_\tau,\quad \hat{Z}=\partial_{z}-2\pi \tau P_z =:-D_\tau.$$
We think of $\bar{D}_\tau=e^{-2\pi \tau P}\circ \bar{\partial}\circ e^{2\pi \tau P}$ as a weighted $\bar{\partial}$ operator acting on (a dense subspace of) $L^2(\mathbb{C})$, and $D_\tau=\bar{D}_\tau^\ast$ as its adjoint.
As before, we write $W_\tau^\alpha=W_{\tau,1}\cdots W_{\tau,|\alpha|}$, where $W_{\tau,i}\in \{ \bar{D}_\tau,D_\tau\}$.
Writing $\mathbb{S}_\tau:L^2(\mathbb{C})\to L^2(\mathbb{C})$ for the orthogonal projection onto the space of $L^2(\mathbb{C})$ functions annihilated by $\bar{D}_\tau$ in the sense of distributions, we are able to say (Proposition \ref{prop:link}) that
$$[\hat{\mathbb{S}}](z,w,\tau)=[\mathbb{S}_\tau](z,w),\quad \mbox{a.e.}\ (z,w,\tau)\in \mathbb{C}\times\mathbb{C}\times\mathbb{R}.$$

The next step is to analyze the operator $\mathbb{S}_\tau$ for fixed $\tau>0$. 
Here we utilize the work of Christ \cite{Christ1991a}, who studied the operators $\mathbb{G}_\tau=(\bar{D}_\tau D_\tau)^{-1}$, $\mathbb{R}_\tau=D_\tau \circ \mathbb{G}_\tau,$ $\mathbb{R}^\ast_\tau = \mathbb{G}_\tau\circ \bar{D}_\tau$, and $\mathbb{S}_\tau=I-D_\tau\circ \mathbb{G}_\tau \circ \bar{D}_\tau$ on $L^2(\mathbb{C})$ (for fixed $\tau$), and proved pointwise bounds on their Schwartz kernels in terms of a smooth function $\sigma_\tau(z)$ and a metric $\rho_\tau(z,w)$ on $\mathbb{C}$ that are intimately connected to the Carnot-Carath\'eodory metric $d$. To simplify our argument, we will replace $\rho_\tau(z,w)$ with a quasimetric $\t\rho_\tau(z,w)$ that is easier to work with; see Section \ref{sec:kernels} for the details.

In Section \ref{sec:fn} we formally define $\mathbb{F}_{K,K'}$ and $\mathbb{N}_{K,K'}$ via 
$$[\mathbb{F}_{K,K'}](\mb{z},\mb{w}) = \int_0^{+\infty} e^{2\pi i \tau (\Re(z_2)-\Re(\bar{w}_2))}\chi(\tau)[\mathbb{R}_\tau^K \mathbb{S}_\tau (\mathbb{R}_\tau^\ast)^{K'}](z,w)d\tau$$
and
$$[\mathbb{N}_{K,K'}](\mb{z},\mb{w}) = \int_0^{+\infty} e^{2\pi i \tau (\Re(z_2)-\Re(\bar{w}_2))}(1-\chi(\tau))[\mathbb{R}_\tau^K \mathbb{S}_\tau (\mathbb{R}_\tau^\ast)^{K'}](z,w)d\tau,$$
where $\chi:[0,+\infty)\to [0,1]$ is a non-increasing smooth function with $\chi(\tau)\equiv 1$ for $\tau\leq 1$ and $\chi(\tau)\equiv 0$ for $\tau\geq 2$.
These operators are densely defined in $L^2({\rm b}\Omega)$
and satisfy $\mathbb{S}=\mathbb{F}_{0,0}+\mathbb{N}_{0,0}$ and
$$\bar{Z}^K\mathbb{F}_{K,K'} Z^{K'}=\mathbb{F}_{0,0},\quad \bar{Z}^K \mathbb{N}_{K,K'} Z^{K'}=\mathbb{N}_{0,0}.$$

Theorem \ref{thm:szego-decomp} therefore requires us to prove pointwise bounds on the Schwartz kernels of operators of the form 
\begin{equation}\label{eq:genops}
	W_\tau^\alpha \mathbb{R}_\tau^K \mathbb{S}_\tau (\mathbb{R}_\tau^\ast)^{K'} W_\tau^\beta.
\end{equation}
To take advantage of the oscillatory term $e^{2\pi i (\Re(z_2)-\Re(\bar{w}_2))\tau}$ in the integrals defining $\mathbb{N}_{K,K'}$ and $\mathbb{F}_{K,K'}$, we will want to integrate by parts in $\tau$. The heart of our argument, expressed by Theorem \ref{thm:chr-der-szegoderivatives}, shows that
\begin{align*}
|T_\tau^M [W_\tau^\alpha &\mathbb{R}_\tau^K \mathbb{S}_\tau (\mathbb{R}_\tau^\ast)^{K'} W_\tau^\beta](z,w)|\\
&\lesssim \begin{cases} \tau^{-M}\sigma_\tau(w)^{-2-|\alpha|-|\beta|}e^{-\epsilon \t\rho_\tau(z,w)}\quad & \mbox{if}\ \tau\gtrsim 1,\\ 
\tau^{-M}\sigma_\tau(w)^{-2-\min(|\alpha|,2)-\min(|\beta|,2)}e^{-\epsilon \t\rho_\tau(z,w)}\quad & \mbox{if}\ \tau\lesssim 1.\end{cases}\end{align*}
Here $$T_\tau=e^{2\pi i \tau \tilde{T}(z,w)}\circ \partial_\tau \circ e^{-2\pi i \tau \tilde{T}(z,w)}=\partial_\tau -2\pi i \tilde{T}(z,w),$$
where $\displaystyle \tilde{T}(z,w)$ is related to the `twist' $T(z,w)$ in the Carnot-Carath\'eodory geometry described in Section \ref{sec:geom}, and will be chosen based on the size of $\tau$.
Standard integral estimation techniques then allow us to establish the estimates in Theorem \ref{thm:szego-decomp}.

The rest of the paper is structured as follows.
In Section \ref{sec:normalization} we use (H2) and (H3) to prove Proposition \ref{prop:potentials} and construct various biholomorphic changes of variables to simplify our computations.
After recalling and developing the necessary facts about the Carnot-Carath\'eodory metric $d(\mb{z},\mb{w})$ in Section \ref{sec:geom} and defining $\mathbb{F}_{K,K'}$ and $\mathbb{N}_{K,K'}$ in Section \ref{sec:fn}, 
in Section \ref{sec:kernels} we show how Christ's bounds are related to  (H1), (H2), the Carnot-Carath\'eodory metric on ${\rm b}\Omega$, and $\tau$.
We then use an algebraic argument to obtain pointwise bounds on the Schwartz kernels of the operators appearing in (\ref{eq:genops}).
The proof of Theorem \ref{thm:szego-decomp} is given in Section \ref{sec:main}, and the proof of Theorem \ref{thm:szego-mapping} is in Section \ref{sec:mapping}.
In Section \ref{sec:szego-growth} we prove the sharpness claim from Corollary \ref{cor:szego-growth}.
The paper concludes with the proof of  Corollary \ref{cor:szego-cancellation} in Section \ref{sec:szego-cancellation}.
There are two appendices, each containing technical results used in the argument: Appendix \ref{sec:sz} contains a discussion of the technicalities surrounding the definition and properties of $\mathbb{S}$ (building off of the discussion in \cite{HalfpapNagelWainger2010}), and contains the proof of a well-known formula relating the Szeg\H{o} and Bergman kernels for unbounded model domains which, to the author's knowledge, has not yet appeared in the literature.
Appendix \ref{sec:link} is devoted to the proof of several technical results from Section \ref{sec:fn}.

\section{Normalization}
\label{sec:normalization}

In this section we explore (H3) \emph{vis-\`a-vis} its connection to the existence of a class of biholomorphic changes of variables that normalize ${\rm b}\Omega$ near a point $\mb{w}\in {\rm b}\Omega$.
We begin with a proof of Proposition \ref{prop:potentials} in Section \ref{sec:normalization-potentials}. In Section \ref{sec:normalization-biholomorphic} we produce a family of biholomorphisms $\Phi:\mathbb{C}^2\to\mathbb{C}^2$ which `isomorphically' preserves the class of UFT domains in the sense that if $\Omega$ is a UFT domain, then $\tilde{\Omega}=\Phi(\Omega)$ is also a UFT domain that satisfies (H1)-(H3) with the same constants as does $\Omega$, and such that $\Phi$ preserves all of the relevant CR structure and integration measures involved in the problem. 

Now let $\Pi:{\rm b}\Omega \to \mathbb{C}\times \mathbb{R}$ be the diffeomorphism $(z,t)=\Pi(\mb{z})=(z,\Re(z_2))$, and suppose that $\mathbb{H}:L^2({\rm b}\Omega)\to L^2({\rm b}\Omega)$ is a bounded $\Re(z_2)$-translation invariant operator. Then the operator $$(\Pi^{-1})^\ast \circ \mathbb{H}\circ \Pi^\ast$$ is a bounded $t$-translation invariant operator on $L^2(\mathbb{C}\times\mathbb{R})$.

If we define $\mathcal{F}:L^2(\mathbb{C}\times\mathbb{R})\to L^2(\mathbb{C}\times\mathbb{R})$ to be the partial Fourier transform
$$\mathcal{F}[f](z,\tau) = \int_\mathbb{R} e^{-2\pi i \tau t} f(z,t)dt,$$ 
then the operator $\hat{\mathbb{H}}:=\mathcal{F}\circ(\Pi^{-1})^\ast \circ \mathbb{H}\circ \Pi^\ast\circ \mathcal{F}^{-1}:L^2(\mathbb{C}\times\mathbb{R})\to L^2(\mathbb{C}\times\mathbb{R})$ is given on functions by formal integration against a Schwartz kernel $[\hat{\mathbb{H}}](z,w,\tau)$ as $$\hat{\mathbb{H}}[f](z,\tau) = \int_{\mathbb{C}} [\hat{\mathbb{H}}](z,w,\tau)f(w,\tau)dm(w),\quad\ \mbox{for a.e.}\ \tau;$$
see \cite{SteinWeiss1972}.

We construct the biholomorphisms $\Phi$ so that $\mathbb{H}^{{\rm b}\tilde{\Omega}}= (\Phi^{-1})^\ast \circ \mathbb{H}\circ \Phi^\ast$ defines a bounded $\Re(\tilde{z}_2)$-translation invariant operator on $L^2({\rm b}\tilde{\Omega})$, and so we may ask how the kernels $[\hat{\mathbb{H}}^{{\rm b}\tilde\Omega}]$ and $[\hat{\mathbb{H}}]$ are related as functions on $\mathbb{C}\times\mathbb{C}\times\mathbb{R}$. This is explored in Section \ref{sec:normalization-substitution}.

\subsection{Proof of Proposition \ref{prop:potentials}}\label{sec:normalization-potentials}

		We first show that $(a)\Rightarrow (b)$. It suffices to prove (b) for $\zeta=0$, since we can then get the full result by applying (b) to $z\mapsto h(\zeta+z)$.

		To begin, we follow Section 4.6 of \cite{BerensteinGay1991} and define, for $z,\eta\in\mathbb{C}$,
		\begin{align*}
			K_1(z,\eta) & =  \frac{1}{2\pi} \Big[ \log|z-\eta|-\log|\eta|+\Re\Big(\frac{z}{\eta}\Big)\Big], \\
			K_2(z,\eta) & =  \frac{1}{2\pi} \Big[ \log|z-\eta|-\log|\eta|+\Re\Big(\frac{z}{\eta}\Big)+\frac{1}{2}\Re\Big(\Big(\frac{z}{\eta}\Big)^2\Big)\Big].
		\end{align*}
		Because $|K_2(z,\eta)|\leq \displaystyle\frac{2}{3}\Big|\frac{z}{\eta}\Big|^3$ and $h\in L^\infty(\mathbb{C})$, the integrals
		\begin{equation}\label{eq:normalization-ptilde} \tilde{P}(z)=\int_{|\eta|\leq 1} K_1(z,\eta)h(\eta)dm(\eta) + \int_{1\leq |\eta|} K_2(z,\eta)h(\eta)dm(\eta)\end{equation}
		converge for all $z\in \mathbb{C}$, and a localization argument implies that $\Delta \tilde{P}(z)=h(z)$ for all $z$, establishing (i); see \cite{BerensteinGay1991} for the details.
		
		Because $\tilde{P}(z)$ is real-valued, to prove (ii) it suffices to estimate 
		\begin{equation}\label{eq:firstpartial}
			4\pi \frac{\partial \tilde{P}}{\partial z}(z) = \int_{|\eta|\leq 1} \Big(\frac{1}{z-\eta}+\frac{1}{\eta}\Big) h(\eta)dm(\eta) + \int_{|\eta|\geq 1} \Big( \frac{1}{z-\eta}+\frac{1}{\eta} + \frac{z}{\eta^2}\Big) h(\eta)dm(\eta).
		\end{equation}
		We consider two cases.
		\smallskip
		
		\textbf{Case 1: $|z|\leq 4$}.
		
		In this case we write $h_1(\eta)=\chi(|\eta|)h(\eta)$ and $h_2(\eta)=h(\eta)-h_1(\eta)$, where
		$\chi:[0,+\infty)\to [0,1]$ is a smooth, non-increasing cut-off function with $\chi(t)\equiv 1$ for $t\leq 5$ and $\chi(t)\equiv 0$ for $t\geq 6$. Then
		\begin{align*}
			 & 4\pi \frac{\partial \tilde{P}}{\partial z}(z) \\
			 &=\  \int_\mathbb{C} \Big( \frac{1}{z-\eta} + \frac{1}{\eta}\Big) h_1(\eta)dm(\eta) + \int_{1\leq |\eta|} \frac{z}{\eta^2}h_1(\eta)dm(\eta) + \int_\mathbb{C} \frac{z^2}{(z-\eta)\eta^2}h_2(\eta)dm(\eta) \\
			  &=\ \int_\mathbb{C} \frac{h_1(\eta)-h_1(z+\eta)}{\eta} dm(\eta) + \int_{1\leq |\eta|} \frac{z}{\eta^2}h_1(\eta)dm(\eta) + \int_\mathbb{C} \frac{z^2}{(z-\eta)\eta^2}h_2(\eta)dm(\eta) \\
			  &=\ I_1 + I_2 + I_3.
		\end{align*}
		We immediately have
		$$|I_1|\leq |z|\|h_1\|_{C^1} \int_{|\eta|\leq 10} \frac{1}{|\eta|}dm(\eta)\lesssim \|h\|_{C^1}|z|$$
		and $$|I_2|\lesssim \|h\|_{C^0} |z|.$$
		When $|\eta|\geq 5$ and $|z|\leq 4$ we have $|z-\eta|\approx |\eta|$, so that
		$$|I_3|\lesssim |z|^2\|h\|_{C^0} \int_{5\leq |\eta|<+\infty} \frac{1}{|\eta|^3}dm(\eta) \lesssim \|h\|_{C^0}|z|^2\lesssim \|h\|_{C^0}|z|.$$
		This proves (ii) in Case 1.
		\medskip
		
		\textbf{Case 2: $|z|\geq 4$}.
		
		In this case we write
		\begin{align*}
			& 4\pi \frac{\partial \tilde{P}}{\partial z}(z) \\
			& =\  \int_{|\eta|\leq \frac{|z|}{3}} \frac{z}{(z-\eta)\eta} h(\eta) dm(\eta) + \int_{1\leq |\eta|\leq \frac{|z|}{3}} \frac{z}{\eta^2}h(\eta)dm(\eta) \\
			& \ \quad  + \int_{|z-\eta|\leq \frac{|z|}{3}} \frac{z^2}{(z-\eta)\eta^2} h(\eta)dm(\eta) + \int_{\frac{|z|}{3}\leq \min(|\eta|,|z-\eta|)} \frac{z^2}{(z-\eta)\eta^2}h(\eta)dm(\eta) \\
			& =\ I_1 + I_2 + I_3 + I_4.
		\end{align*}
		Condition (a) immediately implies that $|I_2|\leq A_0|z|$. 
		We estimate the other integrals as in Case 1:
		$$|I_1|\lesssim \|h\|_{C^0} \int_{1\leq |\eta|\leq \frac{|z|}{3}} \frac{1}{|\eta|} dm(\eta) \lesssim \|h\|_{C^0}|z|,$$
		$$|I_3|\lesssim \|h\|_{C^0} \int_{|z-\eta|\leq \frac{|z|}{3}} \frac{1}{|z-\eta|}dm(\eta)\lesssim \|h\|_{C^0}|z|,$$
		and
		$$|I_4|\lesssim \|h\|_{C^0} |z|^2 \int_{\frac{|z|}{3}\leq |\eta|} \frac{1}{|\eta|^3}dm(\eta)\lesssim \|h\|_{C^0}|z|,$$
		which completes the proof of (ii).
		
		We turn now to the proof of (iii), which is similar to (but more involved than) that of (ii).
		The assumption that $\|h\|_{C^k}<+\infty$, together with the fact that $\tilde{P}$ is real-valued, implies that we need only show that $\displaystyle 4\pi \frac{\partial^k \tilde{P}}{\partial z^k}(z)$ is bounded for $k\geq 2$.
		\smallskip
		
		\textbf{Case 1: $|z|\leq 4$}. 
		
		Split $h=h_1+h_2$ as in Case 1 of the proof of (ii). Then we first write
		\begin{align*}
		4\pi \frac{\partial \tilde{P}}{\partial z}(z) & =  \int_\mathbb{C} \frac{h_1(\eta)-h_1(z+\eta)}{\eta} dm(\eta) + \int_{1\leq |\eta|} \frac{z}{\eta^2}h_1(\eta)dm(\eta) \\
		& \qquad+ \int_\mathbb{C} \Big(\frac{1}{z-\eta} + \frac{1}{\eta} + \frac{z}{\eta^2}\Big)h_2(\eta)dm(\eta).\end{align*}
		When $k=2$ we have
		\begin{align*}
		4\pi \frac{\partial^2 \tilde{P}}{\partial z^2}(z)& = \int_\mathbb{C} \frac{-\frac{\partial h_1}{\partial \eta}(z+\eta)}{\eta} dm(\eta) + \int_{1\leq |\eta|} \frac{h_1(\eta)}{\eta^2}dm(\eta) \\
		&  \qquad + \int_\mathbb{C} \Big(\frac{1}{\eta^2}-\frac{1}{(z-\eta)^2}\Big)h_2(\eta)dm(\eta)\\
			& = I_1+I_2+I_3,
		\end{align*}
		which is estimated (using the fact that $|z|\leq 4$ and $h_1(\eta)\equiv 0$ for $|\eta|\geq 6$) as
		$$|I_1|\lesssim \|h\|_{C^1},\quad |I_2|\lesssim \|h\|_{C^0},$$
		and
		$$|I_3|\lesssim |z|\|h\|_{C^0} \int_{5\leq |\eta|} \frac{1}{|\eta|^3}dm(\eta) \lesssim |z|\|h\|_{C^0}\lesssim \|h\|_{C^0}.$$
		For $k\geq 3$, we have
		\begin{align*}4\pi \frac{\partial^k \tilde{P}}{\partial z^k}(z)& = \int_\mathbb{C} \frac{-\frac{\partial^{k-1} h_1}{\partial \eta^{k-1}}(z+\eta)}{\eta} dm(\eta) + \int_\mathbb{C} \frac{(-1)^{k-1}(k-1)!}{(z-\eta)^k}h_2(\eta)dm(\eta)\\
			& = I_1+I_2,
		\end{align*}
		which is similarly estimated as
		$$|I_1|\lesssim \|h\|_{C^{k-1}}\quad \mbox{and}\quad |I_2|\lesssim \|h\|_{C^0}.$$
		
		\textbf{Case 2: $|z|\geq 4$}. 
		
		In this case, fix a point $z_0$ with $|z-z_0|\leq \frac{1}{4}$, and let $h=h_1+h_2$, where $h_1(\eta)=\chi(|z_0-\eta|)h(\eta)$ and $h_2=h-h_1$, where now $\chi:[0,{+\infty})\to [0,1]$ is a smooth, non-increasing cutoff function with $\chi(t)=1$ if $t\leq \frac{1}{3}$ and $\chi(t)=0$ if $t\geq \frac{1}{2}$. 
		We write
		\begin{align*}
			& 4\pi \frac{\partial \tilde{P}}{\partial z}(z)\\
			 & =\ \int_{|\eta|\leq 1} \Big(\frac{1}{z-\eta}+\frac{1}{\eta}\Big) h(\eta)dm(\eta) + \int_{\mathbb{C}} \Big(\frac{1}{z-\eta} + \frac{1}{\eta} + \frac{z}{\eta^2}\Big) h_1(\eta)dm(\eta) \\
			&  \qquad  +\int_{1\leq |\eta|} \Big( \frac{1}{z-\eta} + \frac{1}{\eta} + \frac{z}{\eta^2}\Big) h_2(\eta)dm(\eta)\\
			& =\ \int_\mathbb{C} \frac{-h_1(\eta+z)}{\eta}dm(\eta) + \int_\mathbb{C} \Big(\frac{1}{\eta}+\frac{z}{\eta^2}\Big)h_1(\eta)dm(\eta) \\
			& \qquad + \int_{|\eta|\leq 1} \Big(\frac{1}{z-\eta}+\frac{1}{\eta}\Big)h(\eta)dm(\eta)+ \int_{1\leq|\eta|} \Big(\frac{1}{z-\eta}+\frac{1}{\eta}+\frac{z}{\eta^2}\Big)h_2(\eta)dm(\eta),
		\end{align*}
		which yields
		\begin{align}
			4\pi \frac{\partial^2 \tilde{P}}{\partial z^2}(z) & = \int_\mathbb{C} \frac{-\frac{\partial h_1}{\partial \eta}(z+\eta)}{\eta}dm(\eta)+ \int_\mathbb{C} \frac{h_1(\eta)}{\eta^2}dm(\eta) \nonumber  \\
			& \quad + \int_{|\eta|\leq 1} \frac{-h(\eta)}{(z-\eta)^2}dm(\eta) + \int_{1\leq |\eta|} \Big( \frac{1}{\eta^2}-\frac{1}{(z-\eta)^2}\Big)h_2(\eta)dm(\eta).\label{eq:seconddircase2}
		\end{align}
		Noting that $h_1(\eta)=h(\eta)$ and $h_2(\eta)=0$ for $|z-\eta|\leq \frac{1}{12}$,  and $h_2(\eta)=h(\eta)$ and $h_1(\eta)=0$ for $|z-\eta|\geq \frac{3}{4}$, we have
		\begin{align*}
			  \int_{1\leq |\eta|} \Big(\frac{1}{\eta^2}- &\frac{1}{(z-\eta)^2}\Big) h_2(\eta)dm(\eta) \\
			 & = \int_{1\leq |\eta|\leq \frac{|z|}{3}} \frac{h(\eta)}{\eta^2}dm(\eta) + \int_{1\leq |\eta|\leq \frac{|z|}{3}} \frac{-h(\eta)}{(z-\eta)^2}dm(\eta)  \\
			&   \qquad + \int_{\frac{1}{12}\leq |z-\eta|\leq 1} \Big(\frac{1}{\eta^2}-\frac{1}{(z-\eta)^2}\Big)h_2(\eta)dm(\eta) \\
			&   \qquad +\int_{1\leq |z-\eta|\leq \frac{|z|}{3}} \frac{h(\eta)}{\eta^2}dm(\eta) + \int_{1\leq |z-\eta|\leq \frac{|z|}{3}} \frac{-h(\eta)}{(z-\eta)^2}dm(\eta)\\
			&   \qquad +\int_{\frac{|z|}{3}\leq \min(|\eta|,|z-\eta|)} \Big(\frac{1}{\eta^2}-\frac{1}{(z-\eta)^2}\Big)h(\eta)dm(\eta)
		\end{align*}
		and
		$$\int_{|\eta|\leq 1} \frac{-h(\eta)}{(z-\eta)^2}dm(\eta) = \int_{|\eta|\leq \frac{|z|}{3}} \frac{-h(\eta)}{(z-\eta)^2}dm(\eta) + \int_{1\leq|\eta|\leq \frac{|z|}{3}} \frac{h(\eta)}{(z-\eta)^2}dm(\eta).$$
		We may therefore rewrite (\ref{eq:seconddircase2}) as
		\begin{align*}
			 & 4\pi \frac{\partial^2 \tilde{P}}{\partial z^2}(z) \\
			  & =\  
			\int_\mathbb{C} \frac{-\frac{\partial h_1}{\partial \eta}(z+\eta)}{\eta}dm(\eta) 
			+ \int_\mathbb{C} \frac{h_1(\eta)}{\eta^2}dm(\eta)  
			+\int_{|\eta|\leq \frac{|z|}{3}}\frac{-h(\eta)}{(z-\eta)^2}dm(\eta) \\
			& \quad \qquad+ \int_{1\leq|\eta|\leq \frac{|z|}{3}} \frac{h(\eta)}{\eta^2}dm(\eta) 
			+ \int_{\frac{1}{12}\leq |z-\eta|\leq 1} \Big(\frac{1}{\eta^2} - \frac{1}{(z-\eta)^2}\Big)h_2(\eta)dm(\eta)
			\\
			& \quad \qquad + \int_{1\leq |\eta-z|\leq \frac{|z|}{3}} \frac{h(\eta)}{\eta^2} dm(\eta)+ \int_{1\leq |z-\eta|\leq \frac{|z|}{3}} \frac{-h(\eta)}{(z-\eta)^2}dm(\eta)
			\\
			& \quad \qquad + \int_{\frac{|z|}{3}\leq\min(|\eta|,|z-\eta|)} \Big(\frac{1}{\eta^2}-\frac{1}{(z-\eta)^2}\Big)h(\eta)dm(\eta) \\
			& =\ I_1 + I_2 + I_3 + I_4 + I_5 + I_6+ I_7 +I_8.
		\end{align*}
		As before, we have
		$$|I_1|\lesssim \frac{\|h\|_{C^1}}{|z|}\lesssim \|h\|_{C^1},\quad |I_2|+|I_3|\lesssim \frac{\|h\|_{C^0}}{|z|^2}\lesssim \|h\|_{C^0},\quad |I_5|+|I_6|\lesssim \|h\|_{C^0},$$
		while 
		$$|I_8|\lesssim |z|\|h\|_{C^0} \int_{\frac{|z|}{3}\leq |\eta|} \frac{1}{|\eta|^3}dm(\eta) \lesssim \|h\|_{C^0}$$
		and, by (a),
		$$|I_4|+|I_7|\leq 2A_0.$$
		This completes the proof of Case 2 when $k=2$.
		For Case 2 when $k\geq 3$, we differentiate (\ref{eq:seconddircase2}) to obtain
		\begin{align*}
			4\pi \frac{\partial^k \tilde{P}}{\partial z^k}(z) & =\ \int_\mathbb{C} \frac{-\frac{\partial^{k-1} h_1}{\partial \eta^{k-1}}(z+\eta)}{\eta}dm(\eta) + \int_{|\eta|\leq 1} \frac{(-1)^{k-1}(k-1)!h(\eta)}{(z-\eta)^k}dm(\eta)\nonumber\\
			& \quad\qquad\qquad + \int_{1\leq |\eta|} \frac{(-1)^{k-1}(k-1)!h_2(\eta)}{(z-\eta)^k}dm(\eta)\nonumber\\
			& =\ \int_\mathbb{C} \frac{-\frac{\partial^{k-1} h_1}{\partial \eta^{k-1}}(z+\eta)}{\eta}dm(\eta)+\int_{\frac{1}{12}\leq |z-\eta|} \frac{(-1)^{k-1}(k-1)!h_2(\eta)}{(z-\eta)^k}dm(\eta),\nonumber
		\end{align*}
		which is estimated to be
		$$\Big| 4\pi \frac{\partial^k \tilde{P}}{\partial z^k}(z)\Big| \lesssim \|h\|_{C^{k-1}}.$$
		This completes the proof that $(a)\Rightarrow (b)$.
		
		It remains to show that if (b) holds, then the bound in (a) holds.
		By translation, it suffices to show that the existence of such a function $P$ for $\zeta=0$ implies that
		\begin{equation}\label{eq:reducedbounda}
			\sup_{r\in [1,+\infty)}\Big| \int_{1\leq |\eta|\leq r} \frac{h(\eta)}{\eta^2} dm(\eta)\Big| \leq A_0<+\infty
		\end{equation}
		for some constant $A_0$ that only depends on the constants from (b).
		Let $\tilde{P}$ be the function constructed in (\ref{eq:normalization-ptilde}).
		\medskip
		
		\textbf{Claim 1: If $\tilde{P}$ satisfies the conclusion (ii) of (b), then (a) holds.}
		\medskip
		
		To see this, note that for $|z|\geq 3$ we have
		\begin{align*}
			4\pi \frac{\partial \tilde{P}}{\partial z}(z) & = \int_{|\eta|\leq \frac{|z|}{3}} \frac{z}{(z-\eta)\eta} h(\eta)dm(\eta) + \int_{1\leq |\eta|\leq \frac{|z|}{3}} \frac{z}{\eta^2} h(\eta)dm(\eta) \\
			 & \qquad  + \int_{|z-\eta|\leq \frac{|z|}{3}} \frac{z^2}{(z-\eta)\eta^2}h(\eta)dm(\eta) \\
			 & \qquad + \int_{\frac{|z|}{3}\leq \min(|\eta|,|z-\eta|)} \frac{z^2}{(z-\eta)\eta^2}h(\eta)dm(\eta) \\
			& = I_1 + I_2 + I_3+ I_4.
		\end{align*}
		Because $$\Big| 4\pi \frac{\partial \tilde{P}}{\partial z}(z)\Big| \leq A_1|z|$$
		by assumption, and we may estimate as above to see that
		$$|I_1|\lesssim \|h\|_{C^0}|z|,\quad |I_3|\lesssim \|h\|_{C^0}|z|,\quad \mbox{and}\quad |I_4|\lesssim \|h\|_{C^0}|z|,$$
		we must have
		$$|I_2|=|z|\Bigg| \int_{1\leq |\eta|\leq \frac{|z|}{3}} \frac{h(\eta)}{\eta^2}dm(\eta)\Bigg| \leq C(A_1,\|h\|_{C^0})|z|,$$
		so that
		$$\Bigg| \int_{1\leq |\eta|\leq \frac{|z|}{3}} \frac{h(\eta)}{\eta^2}dm(\eta)\Bigg| \leq C(A_1,\|h\|_{C^0}),\ \ |z|\geq 3.$$
		Because $r=\frac{|z|}{3}\in [1,+\infty)$ is arbitrary, the claim is proved.
		\bigskip
		
		\textbf{Claim 2: $\tilde{P}$ satisfies conclusion (ii) of (b).} 
		\medskip
		
		Let $P$ be a function satisfying (b) (for $\zeta=0$).
		By estimating directly as in the proof that $(a)\Rightarrow (b)$, one shows that
		$$|\nabla \tilde{P}(z)| \leq C(\|h\|_{C^1})|z|\log(|z|+2),$$
		while we have assumed \emph{a priori} that $|\nabla P(z)|\leq A_1|z|$.
		Now, $P(z)-\tilde{P}(z)=Q(z)$ is harmonic on $\mathbb{C}$ and satisfies $|\nabla Q(z)|\leq C(A_1,\|h\|_{C^1})|z|\log(|z|+2)$, and therefore $Q(z)$ is a degree 2 harmonic polynomial\footnote{Let $V(z)$ be a harmonic conjugate of $Q(z)$, and consider the entire function $f(z)=Q(z)+iV(z)$. The Cauchy-Riemann equations imply that $|\nabla V(z)|\leq C(h)|z|\log(|z|+2)$ as well, so that if $w\in \mathbb{C}$ is any fixed complex number and $R>2|w|+10$, then
			$$|f'''(w)| = |(f')''(w)| \leq 2R^{-2}\displaystyle \sup_{|\eta-w|=R} |f'(\eta)| \lesssim \frac{(R+|w|)\log(2+R+|w|)}{R^2}\to 0\ \mbox{as}\ R\to +\infty,$$
			proving that $f'''\equiv 0$, and therefore $f(z)$ is a degree 2 polynomial.}. Because $Q(0)=0$ and $|\nabla Q(0)|=0$, we may write $$Q(z)=2\Re(cz^2)\ \mbox{for some}\ c\in\mathbb{C}.$$
		It follows that
		\begin{equation}\label{eq:qderivform}
			2cz=\frac{\partial Q}{\partial z}(z)= \frac{\partial P}{\partial z}(z)-\frac{\partial \tilde{P}}{\partial z}(z).
		\end{equation}
		For $|z|\leq 3$, the inequalities
		$$\Big| \frac{\partial P}{\partial z} (z)\Big| \leq A_1|z|\quad \mbox{and}\quad \Big| \frac{\partial \tilde{P}}{\partial z}(z)\Big|\leq C(\|h\|_{C^1})|z|$$
		imply that 
		$$|c|\leq \frac{C(\|h\|_{C^1})+A_1}{2}.$$
		By writing $\tilde{P}(z)=P(z)-Q(z)$, we conclude that
		$$|\nabla \tilde{P}(z)| \leq C(\|h\|_{C^1},A_1)|z|.$$
		Because this constant depends only on $A_1$ and $\|h\|_{C^1}$ (and not on $\zeta$), the claim (and therefore the proposition) is proved.

\subsection{Biholomorphic Changes of Variables}
\label{sec:normalization-biholomorphic}

We start by defining a family of biholomorphisms $\Phi:\mathbb{C}^2\to \mathbb{C}^2$ that preserves the class of UFT domains.
Fix $\mb{\sigma}=(\sigma,\sigma_2)\in\mathbb{C}^2$ and an entire function $H:\mathbb{C}\to\mathbb{C}$ with $H(0)=0$, and define the map $\Phi:\mathbb{C}^2\to\mathbb{C}^2$ via
$$(\tilde{z},\tilde{z}_2)=\Phi(\mb{z})=(z-\sigma,z_2-\sigma_2-2i H(z-\sigma)).$$
It is immediate to check that $\Phi$ is a biholomorphism.

Define $\tilde{P}(\tilde{z})=P(\tilde{z}+\sigma)-\Im(\sigma_2)+2\Re(H(\tilde{z}))$ and $\tilde{\Omega}=\{(\t{z},\t{z}_2)\in\mathbb{C}^2\ :\ \Im(\t{z}_2)>\t{P}(\t{z})\}$, and give $\t{\Omega}$ and ${\rm b}\t{\Omega}$ the Lebesgue measures $dm_{\t{\Omega}}$ and $dm_{{\rm b}\t{\Omega}}$ as in the introduction. Then the following elementary result holds.

\begin{proposition}\label{prop:normalization-biholgeneral}{ The domain $\t\Omega$ and the biholomorphism $\Phi$ have the following properties.
		\begin{itemize}
			\item[(a)] $\t\Omega$ is a UFT domain with constants in (H1)-(H3) identical to those of $\Omega$.
			\item[(b)] $\Phi(\Omega)=\t{\Omega}$, $\Phi(\b \Omega)=\b\t\Omega$, and $\Phi(\mb{\sigma})=\mb{0}$.
			\item[(c)] $\Phi^\ast (dm_{\t\Omega})=dm_{\Omega}$ and $\Phi^\ast (dm_{\b\t\Omega})=dm_{\b\Omega}$.
			\item[(d)] As differential operators, $\Phi^\ast (V_{\b\t\Omega})=V_{\b\Omega}$ for $V\in \{Z,\bar{Z},T\}$.
			\item[(e)] $\mathbb{S}^{\b\t\Omega}=(\Phi^{-1})^\ast\circ \mathbb{S}\circ \Phi^\ast$.
		\end{itemize}
}\end{proposition}
\begin{proof}
	For part (a), we need only note that $\Delta \t P(\t z)=\Delta P(\t z+\sigma)$. 
	Parts (b), (c), and (d) follow from direct computations. 
	Part (e) is proved by noting that $\Phi^\ast:L^2(\b\t\Omega)\to L^2(\b\Omega)$ is an isomorphism with $\bar{Z}_{\b\t\Omega}f=0$ if and only if $\bar{Z}_{\b\Omega} \Phi^\ast f=0$ in the sense of distributions.
\end{proof}

By restricting our attention to $\mb{\sigma}\in\b\Omega$ and carefully choosing the entire function $H$, we can ensure that the function $\t P(\t z)$ behaves nicely near $\t z=0$. In particular, for $\mb{\sigma}=(\sigma,\sigma_2)\in\b\Omega$ we let $P^{\mb{\sigma}}(z)$ be the subharmonic potential function for $h(z + \sigma)$ given by part (b) of Proposition \ref{prop:potentials}.  For $\kappa\geq 2$, we then define
\begin{equation}\label{eq:normalization-pwk}
	P^{\mb{\sigma},\kappa}(z) = P^{\mb{\sigma}}(z)-2\Re\Big( \sum_{j=2}^\kappa \frac{1}{j!} \frac{\partial^j P^{\mb{\sigma}}}{\partial z^j}(0)z^j\Big).
	\end{equation}

Then our main result of this section is as follows. 

\begin{lemma}\label{lem:normalization-bihol}{ Let $\Omega=\{\mb{z}\in\mathbb{C}^2\ :\ \Im(z_2)>P(z)\}$ be a UFT domain, and fix $\mb{\sigma}\in \b\Omega$ and $\kappa\geq 2$. 
		\begin{itemize} 
			\item[(a)] There exists an entire function $H:\mathbb{C}\to\mathbb{C}$ with $H(0)=0$ so that the biholomorphism $\Phi:\mathbb{C}^2\to\mathbb{C}^2$ from Proposition \ref{prop:normalization-biholgeneral} sends $\Omega$ to $\t\Omega$, where $\t P(\t z)= P^{\mb{\sigma},\kappa}(\t z)$.
			\item[(b)] If we further assume that $P=P^{\mb{0},2}$, then $\displaystyle H(z)=\sum_{j=1}^\kappa \frac{1}{j!} \frac{\partial^j P}{\partial z^j}(\sigma)z^j$.
		\end{itemize}
	}\end{lemma}
\begin{proof}
Let $H$ be the unique entire function with $H(0)=0$ and
$$2\Re(H(z))=P^{\mb{\sigma},\kappa}(z)-P(z+\sigma)+P(\sigma).$$
Then 
\begin{align*}
	\t P(\t z) & = P(\t z+\sigma)-\Im(\sigma_2)+2\Re(H(\t z)) \\
	& =  P(\t z+\sigma)-P(\sigma)+(P^{\mb{\sigma},\kappa}(\t z)-P(\t z+ \sigma)+P(\sigma)) \\
	& =  P^{\mb{\sigma},\kappa}(\t z),
\end{align*}
proving (a).

Under the additional assumption that $P=P^{\mb{0},2}$ and $\kappa\geq 3$, we let $H_2$ denote the unique entire function with $H_2(0)=0$ and $$2\Re(H_2(z))=P^{\mb{\sigma},2}(z)-P(z+ \sigma)+P(\sigma).$$
Then $H(0)=H_2(0)=0$ and 
\begin{align*}
P^{\mb{\sigma},\kappa}(z)& = P^{\mb{\sigma},2}(z)+2\Re\Bigg( \sum_{j=3}^\kappa \frac{1}{j!} \frac{\partial^j P^{\mb{\sigma}}}{\partial z^j} (0)z^j\Bigg)\\
& = P^{\mb{\sigma},2}(z)+2\Re\Bigg( \sum_{j=3}^\kappa \frac{1}{j!} \frac{\partial^j P^{\mb{\sigma,2}}}{\partial z^j} (0)z^j\Bigg),\end{align*}
so uniqueness implies that
$$H(z)=H_2(z)+\sum_{j=3}^\kappa \frac{1}{j!} \frac{\partial^j P^{\mb{\sigma},2}}{\partial z^j}(0)z^j.$$
Moreover, writing
$$P^{\mb{\sigma},2}(z)=P(z+\sigma)-P(\sigma)+H_2(z)+\overline{H_2(z)}$$ yields
$$\frac{\partial^j P^{\mb{\sigma},2}}{\partial z^j}(z)= \frac{\partial^j P}{\partial z^j}(z+\sigma)+ \frac{\partial^j H_2}{\partial z^j}(z),$$
and hence
\begin{equation}\label{eq:normalization-polystep}
	H(z)=H_2(z) + \sum_{j=3}^\kappa \frac{1}{j!} \Big( \frac{\partial^j P}{\partial z^j}(\sigma)+\frac{\partial^j H_2}{\partial z^j}(0)\Big) z^j.
\end{equation}
The result (b) immediately follows from (\ref{eq:normalization-polystep}) once we show that $H_2(z)=P_z(\sigma)z + \frac{1}{2}P_{z,z}(\sigma)z^2$.

To see this, note first that $H_2'(z)=P^{\mb{\sigma},2}_z(z)-P_z(z+\sigma)$, so that
\begin{align*}
|P^{\mb{\sigma},2}(z)-P(z+\sigma)+P(\sigma)| & = |2\Re(H_2(z))|\\
& = 2\Big|\Re\Big(z\int_0^1 P^{\mb{\sigma},2}_z(tz)-P_z(\sigma+tz)dt\Big)\Big| \\
& \leq 2|z|\int_0^1 |P^{\mb{\sigma},2}_z(tz)-P_z(\sigma+tz)|dt.
\end{align*}
The Maximum Modulus Theorem applied to $H'_2(z)$ on the disc $|\zeta|\leq |z|$, together with part (b) of Proposition \ref{prop:potentials} and part (a) of Proposition \ref{prop:normalization-biholgeneral} then implies that
\begin{align*}
	|P^{\mb{\sigma},2}(z)-P(z+\sigma)+P(\sigma)| & \leq 2|z|\max_{|\zeta|=|z|} |P^{\mb{\sigma},2}_z(\zeta)-P_z(\sigma+\zeta)| \\
	& \leq 2 |z|\Bigg( \max_{|\zeta|=|z|} |P^{\mb{\sigma},2}_z(\zeta)| + \max_{|\zeta|=|z|} |P_z(\sigma+\zeta)|\Bigg)\\
	& \leq 2A_1|z|(|z|+|\sigma|+|z|)\\
	& \lesssim  |z|(1+|z|),
\end{align*}
and therefore $P^{\mb{\sigma},2}(z)-P(z+\sigma)+P(\sigma)$ is a harmonic polynomial of degree no more than 2.

We may therefore write $H_2(z)=a+bz+cz^2$ for some complex constants $a,b,c$. The condition $H_2(0)=0$ immediately yields $a=0$, while
$$b+2cz=H_2'(z)=P_z(\sigma+z)-P^{\mb{\sigma},2}_z(z),$$
so that $$b=H_2'(0)=P_z(\sigma)-\underbrace{P^{\mb{\sigma},2}_z(0)}_{=0}=P_z(\sigma).$$
Finally, 
$$2c=H''_2(0)=P_{z,z}(\sigma)-\underbrace{P^{\mb{\sigma},2}_{z,z}(0)}_{=0}=P_{z,z}(\sigma).$$
This concludes the proof.
\end{proof}

%

\subsection{Substitution of Schwartz Kernels}
\label{sec:normalization-substitution}

We return now to the question posed in the beginning of Section \ref{sec:normalization} about Schwartz kernels. 
In particular, let $\b\Omega=\{\mb{z}\in \mathbb{C}^2\ :\ \Im(z_2)>P(z)\}$ be a UFT domain and let $\mathbb{H}:L^2({\rm b}\Omega)\to L^2(\b \Omega)$ be a bounded $\Re(z_2)$-translation invariant operator.
We further assume that $P=P^{\mb{0},2}$. Fix $\mb{\sigma}=(\sigma,\sigma_2)\in \b\Omega$ and $\kappa\geq 2$, and let $\Phi:\mathbb{C}^2\to\mathbb{C}^2$ be the biholomorphism constructed in Section \ref{sec:normalization-biholomorphic} corresponding to $\mb{\sigma}$ and the entire function $H(z)$ in part (b) of Lemma \ref{lem:normalization-bihol}, so that $\Phi(\Omega)=\t\Omega=\{(\t z,\t z_2)\in\mathbb{C}^2\ :\ \Im(\t z_2)> P^{\mb{\sigma},\kappa}(\t z)\}.$
Then $\mathbb{H}^{\b\t\Omega}=(\Phi^{-1})^\ast \circ \mathbb{H}\circ \Phi^\ast$ is a bounded $\Re(\t z_2)$-translation invariant operator on $L^2(\b\t\Omega)$.

Denote by $\hat{\mathbb{H}}$ and $\hat{\mathbb{H}}^{\b\t\Omega}$ the bounded (on $L^2(\mathbb{C}\times\mathbb{R})$) operators
$$\hat{\mathbb{H}}=\mathcal{F}\circ(\Pi^{-1})^\ast \circ \mathbb{H}\circ \Pi^\ast\circ \mathcal{F}^{-1}$$
and
$$\hat{\mathbb{H}}^{\b\t\Omega}=\mathcal{F}\circ(\Pi^{-1})^\ast \circ \mathbb{H}^{\b\t\Omega}\circ \Pi^\ast\circ \mathcal{F}^{-1}.$$

\begin{lemma}\label{lem:normalization-schwartz}
	As functions on $\mathbb{C}\times\mathbb{C}\times\mathbb{R}$, the Schwartz kernels $[\hat{\mathbb{H}}](z,w,\tau)$ and $[\hat{\mathbb{H}}^{\b\t\Omega}](\t z,\t w,\t\tau)$ satisfy
	$$[\hat{\mathbb{H}}](z,w,\tau)= e^{-2\pi i \tau(T_\kappa(z,\sigma)-T_\kappa(w,\sigma))}[\hat{\mathbb{H}}^{\b\t\Omega}](z-\sigma,w-\sigma,\tau),$$
	where $$T_\kappa(\zeta,\sigma)=-2\Im\Big( \sum_{j=1}^\kappa \frac{1}{j!} \frac{\partial^j P}{\partial z^j}(\sigma)(\zeta-\sigma)^j\Big).$$
\end{lemma}

\begin{proof}
	The proof is similar in spirit to the derivation of Equation (\ref{szegoformula}) in \cite{Nagel1986}. We begin by noting that for $f\in L^2(\mathbb{C}\times\mathbb{R})$ we have
	$$\mathcal{F}\circ(\Pi^{-1})^\ast \circ \Phi^\ast \circ\Pi^\ast\circ \mathcal{F}^{-1}\circ\hat{\mathbb{H}}^{\b\t\Omega}[f](z,\tau)= \hat{\mathbb{H}}\circ \mathcal{F}\circ(\Pi^{-1})^\ast  \circ \Phi^\ast\circ \Pi^\ast\circ \mathcal{F}^{-1}[f](z,\tau)$$

	For $(w,s)\in \mathbb{C}\times\mathbb{R}$,
	\begin{align*}
	\Pi(\Phi(\Pi^{-1}(w,s))) & = \Pi(\Phi(w,s+i P(w))) \\
	& = \Pi(w-\sigma,s+iP(w)-\sigma_2-2iH(w-\sigma))\\
	& = (w-\sigma,s-\Re(\sigma_2)+2\Im(H(w-\sigma))),
	\end{align*}
	and therefore for $( w,\tau)\in \mathbb{C}\times\mathbb{R}$ we have
	\begin{align*}
		\mathcal{F}\circ(\Pi^{-1})^\ast  \circ \Phi^\ast\circ &\Pi^\ast\circ \mathcal{F}^{-1}[f]( w, \tau) \\
		& =  \int_\mathbb{R} e^{-2\pi i \tau s} \mathcal{F}^{-1}[f](\Pi(\Phi(\Pi^{-1}( w,s))))ds \\
		& =  \int_\mathbb{R} e^{-2\pi i \tau s} \mathcal{F}^{-1}[f]( w-\sigma,s-\Re(\sigma_2)+2\Im(H( w-\sigma)))ds \\
		& =  e^{-2\pi i  \tau( \Re(\sigma_2)-2\Im(H( w-\sigma)))}\int_\mathbb{R} e^{-2\pi i  \tau s} \mathcal{F}^{-1}[f](w-\sigma,s)ds\\
		& =  e^{-2\pi i \tau( \Re(\sigma_2)-2\Im(H( w-\sigma)))}f( w-\sigma,\tau)
	\end{align*}
	and
	$$\mathcal{F}\circ(\Pi^{-1})^\ast \circ \Phi^\ast \circ\Pi^\ast\circ \mathcal{F}^{-1}\circ\hat{\mathbb{H}}^{\b\t\Omega}[f]( z, \tau) = e^{-2\pi i  \tau (\Re(\sigma_2)-2\Im(H(z-\sigma)))}\hat{\mathbb{H}}^{\b\t\Omega}[f]( z-\sigma, \tau).$$
	We therefore have
	\begin{align*}
		\hat{\mathbb{H}}^{\b\t\Omega}[f]& ( z, \tau)\\
		 & =  e^{2\pi i  \tau (\Re(\sigma_2)-2\Im(H(z)))}\mathcal{F}\circ(\Pi^{-1})^\ast \circ \Phi^\ast \circ\Pi^\ast\circ \mathcal{F}^{-1}\circ\hat{\mathbb{H}}^{\b\t\Omega}[f]( z+\sigma, \tau)\\
		& =   e^{2\pi i  \tau (\Re(\sigma_2)-2\Im(H(z)))}\hat{\mathbb{H}}\circ \mathcal{F}\circ(\Pi^{-1})^\ast  \circ \Phi^\ast\circ \Pi^\ast\circ \mathcal{F}^{-1}[f](z+\sigma,\tau)\\
	 & =  e^{2\pi i  \tau (\Re(\sigma_2)-2\Im(H(z)))}\\
	 & \qquad\qquad\times\int_{\mathbb{C}} [\hat{\mathbb{H}}](z+\sigma,w,\tau)\mathcal{F}\circ(\Pi^{-1})^\ast  \circ \Phi^\ast\circ \Pi^\ast\circ \mathcal{F}^{-1}[f](w,\tau)dm(w)\\
	 & =  e^{2\pi i  \tau (\Re(\sigma_2)-2\Im(H(z)))} \\
	 & \qquad\qquad\times  \int_{\mathbb{C}} [\hat{\mathbb{H}}](z+\sigma,w,\tau)e^{-2\pi i \tau( \Re(\sigma_2)-2\Im(H( w-\sigma)))}f( w-\sigma,\tau)dm(w)\\
	 & =   \int_{\mathbb{C}} [\hat{\mathbb{H}}](z+\sigma,w,\tau) e^{-2\pi i \tau(2\Im(H(z))-2\Im(H( w-\sigma)))}f( w-\sigma,\tau)dm(w)\\
	 & =  \int_{\mathbb{C}} [\hat{\mathbb{H}}](z+\sigma,w+\sigma,\tau) e^{-2\pi i \tau(2\Im(H(z))-2\Im(H(w)))}f( w,\tau)dm(w).
	\end{align*}
	Because this holds for every $f\in L^2(\mathbb{C}\times\mathbb{R})$,
	$$[\hat{\mathbb{H}}^{\b\t\Omega}](\t z,\t w,\t\tau)=[\hat{\mathbb{H}}](\t z+\sigma,\t w+\sigma,\t \tau) e^{-2\pi i \tau(2\Im(H(\t z))-2\Im(H(\t w)))},$$
	or equivalently
	$$[\hat{\mathbb{H}}](z,w,\tau)=e^{2\pi i \tau(2\Im(H(z-\sigma))-2\Im(H(w-\sigma)))}[\hat{\mathbb{H}}^{\b\t\Omega}](z-\sigma,w-\sigma,\tau).$$
	Part (b) of Lemma \ref{lem:normalization-bihol} implies that $-2\Im(H(\zeta-\sigma))=T_\kappa(\zeta,\sigma)$. This completes the proof.	
\end{proof}

\section{Metrics}
\label{sec:geom}

In this section we study the properties of the Carnot-Carath\'eodory control metric $d(\mb{z},\mb{w})$ on ${\rm b}\Omega$, obtain approximate formulas for $d(\mb{z},\mb{w})$ to be used in later estimates, and construct a smooth version of $d$ that allows us to construct bump functions on ${\rm b}\Omega$ adapted to the control geometry.

On ${\rm b}\Omega$ we decompose $\bar{Z}=\frac{1}{2}(X+iY)$ and $Z=\frac{1}{2}(X-iY)$, where $X=Z+\bar{Z}$ and $Y=\frac{1}{i}(Z-\bar{Z})$ are real vector fields.
We begin by defining the control metric associated to the vector fields $X$ and $Y$, and recalling a few of its properties.

The Carnot-Carath\'eodory distance between $\mb{z}$ and $\mb{w}$ on ${\rm b}\Omega$ is defined to be
\begin{align*}
	d(\mb{z},\mb{w}) & = {\rm inf}\ \lb \delta\ :\ \exists \gamma:[0,1]\rightarrow {\rm b}\Omega,\ \gamma(0)=\mb{z},\ \gamma(1)=\mb{w},\right. \\
	& \qquad\qquad\quad \gamma'(t)=\alpha(t)\delta X(\gamma(t))+\beta(t)\delta Y(\gamma(t))\ a.e., \\
	& \qquad\qquad\quad \left. \alpha,\beta\in FPWS[0,1]\ \mbox{and}\  \||\alpha(\cdot)|^2+|\beta(\cdot)|^2\|_\infty \leq 1\rb.
\end{align*}

Here, the function space $FPWS[0,1]$ consists of all functions $f:[0,1]\rightarrow \mathbb{R}$ for which there exist $0=a_0<a_1<\cdots<a_N<a_{N+1}=1$ such that, for all $i=0,\ldots,N,$ $f$ is smooth on $(a_i,a_{i+1})$ and $f|_{(a_i,a_{i+1})}$ extends continuously to $[a_i,a_{i+1}]$.

By the results in \cite{Peterson2014}, the balls with respect to this metric are given by 
\begin{align}
	  B_d(\mb{z},&\delta) \nonumber \\
	:= & \lb \mb{w}\in\b\Omega\ :\ d(\mb{z},\mb{w})<\delta\rb \nonumber\\
	 \approx & \lb \mb{w}\in\b\Omega\ :\ |z-w|<\delta\ \mbox{and}\ |\Re(z_2)-\Re(w_2)-T(z,w)|<\Lambda(\mb{z},\delta)\rb,\label{eq:geom-ballapprox}
\end{align}
where $$T(z,w)=-2\Im\Big(\int_0^1 (z-w)P_z(w+(z-w)r)dr\Big)$$ and $\delta\mapsto\Lambda(\mb{z},\delta)$ is defined as 
\[ \Lambda(\mb{z},\delta):= \sup\{ |t|\ :\ t\in\mathbb{R}\ \mbox{and}\ d(\mb{z},(z,t+z_2))<\delta\}.\]
Indeed, defining
$$Cyl_d(\mb{z},\delta)=\{ \mb{w}\in\b\Omega\ :\ |z-w|<\delta\ \mbox{and}\ |\Re(z_2)-\Re(w_2)-T(z,w)|<\Lambda(\mb{z},\delta)\},$$
we may express (\ref{eq:geom-ballapprox}) quantitatively as
$$Cyl_d\Big(\mb{z},\frac{1}{4}\delta\Big)\subset B_d(\mb{z},\delta)\subset Cyl_d(\mb{z},3\delta).$$
To see this, note that if $|z-w|\leq \delta$ then $\Lambda(\mb{z},\delta)\geq \Lambda(\mb{w},\frac{1}{3}\delta)$, and therefore
\begin{align*}
	& Cyl_d(\mb{z},\delta) \\
	& =\ \{ \mb{w}\ :\ |w-z|\leq \delta\ \mbox{and}\ |\Re(z_2)-\Re(w_2)-T(z,w)|< \Lambda(\mb{z},\delta)\}\\
	& \subset\ \{ \mb{w}\ :\ |w-z|\leq \delta\ \mbox{and}\ |\Re(z_2)-\Re(w_2)-T(z,w)|< \max_{|z^\ast-z|\leq \delta}\Lambda((z^\ast,z^\ast_2),\delta)\}\\
	& \subset\ B_d(\mb{z},4\delta)
\end{align*}
and
\begin{align*}
	& B_d(\mb{z},\delta) \\
	& \subset\ \{ \mb{w}\ :\ |w-z|\leq \delta\ \mbox{and}\ |\Re(z_2)-\Re(w_2)-T(z,w)|< \max_{|z^\ast-z|\leq \delta}\Lambda((z^\ast,z^\ast_2),\delta)\}\\
	& \subset\ \{ \mb{w}\ :\ |w-z|\leq \delta\ \mbox{and}\ |\Re(z_2)-\Re(w_2)-T(z,w)|< \Lambda(\mb{z},3\delta)\}\\
	& \subset\ Cyl_d(\mb{z},3\delta).
\end{align*}
Thus, $B_d(\mb{z},\delta)$ is a `twisted ellipsoid', with minor radii $\delta$ in the $z$-direction and $\Lambda(\mb{z},\delta)$ in the $\Re(z_2)$-direction.

\begin{remark}\label{rem:metric-invar}{\rm By Proposition \ref{prop:normalization-biholgeneral} and part (a) of Lemma \ref{lem:normalization-bihol}, $Z$, $\bar{Z}$, $T$, $\mathbb{S}$, and the metric $d(\mb{z},\mb{w})$ are preserved under the biholomorphisms produced in Section \ref{sec:normalization}. In other words, the estimates appearing in Theorem \ref{thm:szego-decomp}, Corollary \ref{cor:szego-growth}, Corollary \ref{cor:szego-cancellation}, and Theorem \ref{thm:szego-mapping} are also invariant under these biholomorphisms.  We will therefore assume that $P=P^{\mb{0},2}$ throughout the rest of the paper.}\end{remark}

Our first major result in this section, proved in Section \ref{sec:metric-uftisaq}, describes $\Lambda(\mb{z},\delta)$ in terms of $h$.

\begin{lemma}\label{lem:geom-uftisaq}
	Suppose that $h(z)$ satisfies (H1) and (H2). Then, uniformly in $0<\delta<{+\infty}$ and $\mb{z}\in {\rm b}\Omega$,
	\begin{equation}\label{eq:geom-uftisaq}
		\Lambda(\mb{z},\delta)\approx\int_{|\eta-z|<\delta} h(\eta)dm(\eta).
	\end{equation}
	Moreover, there exists $\delta_0>0$ so that $$\Lambda(\mb{z},\delta)\approx \delta^2\qquad \mbox{for}\ \delta_0\leq \delta<+\infty\ \mbox{and}\ \mb{z}\in {\rm b}\Omega.$$
\end{lemma}

One consequence of Lemma \ref{lem:geom-uftisaq} is that (H1) and (H2) imply that every UFT domain is approximately quadratic in the sense of \cite{Peterson2014}.
The conclusion of Lemma \ref{lem:geom-uftisaq} for $\delta\geq \delta_0$ will follow from the following technical result, whose statement is slightly altered from (but admits the same proof as) the original.

\begin{lemma}[{\cite{Peterson2014}, Theorem 4.2}]\label{lem:geom-ugschar}
	For bounded and continuous $h(z)$, the following are equivalent:
	\begin{itemize}
		\item[(a)] There exists $0<\delta_0<+\infty$ with $\Lambda(\mb{z},\delta)\approx \delta^2$ uniformly in $\delta \geq \delta_0$ and $\mb{z}\in {\rm b}\Omega$,
		\item[(b)] For some $\delta_0> 0$, $\displaystyle \int_{|\eta-z|<\delta_0} h(\eta)dm(\eta) \approx 1$ uniformly in $z\in\mathbb{C}$, and
		\item[(c)] There exists $0<\delta_0<+\infty$ with $\displaystyle \int_{|\eta-z|<\delta} h(\eta)dm(\eta) \approx \delta^2$, uniformly in $z\in\mathbb{C}$ and $\delta\geq \delta_0.$	
	\end{itemize}
\end{lemma}

By Lemma \ref{lem:geom-uftisaq}, we may take $\delta\mapsto \Lambda(\mb{z},\delta)$  strictly increasing, and can therefore compute its inverse $\delta\mapsto \mu(\mb{z},\delta)$.
That is, 
\begin{equation}\label{eq:metrics-lambdamu}
	\Lambda(\mb{z},\mu(\mb{z},\delta))=\delta=\mu(\mb{z},\Lambda(\mb{z},\delta)).\end{equation}

The results of \cite{Peterson2014} and Lemma \ref{lem:geom-uftisaq} imply that
\begin{equation}\label{eq:geom-metric}
	d(\mb{z},\mb{w})\approx |z-w|+\mu(\mb{w},|\Re(z_2)-\Re(w_2)-T(z,w)|).
\end{equation}
We establish (\ref{eq:geom-metric}) by taking $\delta=d(\mb{z},\mb{w})$ in (\ref{eq:geom-ballapprox}), and noting that if $A=|z-w|$ and $B=|\Re(z_2)-\Re(w_2)-T(z,w)|$, then $B\lesssim \Lambda(\mb{z},\delta)$, so that $A+\mu(\mb{z},B)\lesssim \delta$.
Similarly, $A+\mu(\mb{z},B)\gtrsim \delta$, establishing (\ref{eq:geom-metric}).

By using our biholomorphisms from Section \ref{sec:normalization-biholomorphic}, we can obtain a simpler version of formula (\ref{eq:geom-metric}) for $d(\mb{z},\mb{w})$ depending on the size of $|z-w|$. 

\begin{lemma}\label{lem:metric-approx} 
	Let $d(\mb{z},\mb{w})$ be as above.
	\begin{itemize}
		\item[(a)] For $|z-w|\gtrsim 1$, \begin{align*}
			d(\mb{z},\mb{w}) & \approx |z-w|+\mu(\mb{w},|\Re(z_2)-\Re(w_2)-T_2(z,w)|)\\
			& \approx |z-w|+\mu(\mb{w},|\Re(z_2)-\Re(w_2)-T_1(z,w)|).\end{align*}
		\item[(b)] For $|z-w|\lesssim 1$ and $\kappa\geq m$, $$d(\mb{z},\mb{w})\approx |z-w|+\mu(\mb{w},|\Re(z_2)-\Re(w_2)-T_\kappa(z,w)|).$$
	\end{itemize}
\end{lemma}

We prove Lemma \ref{lem:metric-approx} in Section \ref{sec:metric-approx}.

Finally, the proof of Theorem \ref{thm:szego-mapping} requires, for fixed $\mb{w}\in{\rm b}\Omega$, a smooth version of the function $\mb{z}\mapsto d(\mb{z},\mb{w})$ for constructing smooth bump functions; cf. \cite{NagelStein2001b,Street2014}.
This is accomplished by the following result, proved in Section \ref{sec:metric-smooth}.

\begin{lemma}\label{lem:metric-smooth} 
	For each $\mb{w}\in\b\Omega$ there is a function $d^\ast(\bullet,\mb{w}):\b\Omega\to[0,+\infty)$ with 
	\begin{itemize}
		\item[(a)] $d^\ast(\mb{z},\mb{w})\approx d(\mb{z},\mb{w})$, and
		\item[(b)] $|Z_{\mb{z}}^\alpha d^\ast(\mb{z},\mb{w})|\lesssim d^\ast(\mb{z},\mb{w})^{1-|\alpha|} \approx d(\mb{z},\mb{w})^{1-|\alpha|}$ for $|\alpha|\leq 2$,
	\end{itemize}
	where $Z_{\mb{z}}^\alpha$ denotes an arbitrary $|\alpha|$-order derivative in the vector fields $Z$ or $\bar{Z}$ acting in the $\mb{z}$ variables. Moreover, the constants in (a) and (b) can be chosen independently of $\mb{z}$ and $\mb{w}$.
	\end{lemma}

\subsection{Proof of Lemma \ref{lem:geom-uftisaq}}\label{sec:metric-uftisaq}

The first step in the proof of Lemma \ref{lem:geom-uftisaq} is a technical result establishing several quantities that are comparable to $\Lambda(\mb{z},\delta)$ for $\delta$ sufficiently small and $z\in\mathbb{C}$.

\begin{lemma}\label{lem:metric-approxsmall} Suppose that $h(z)$ satisfies (H1) and (H2). Then there exists $\delta_0>0$ and $\nu_0,\ldots,\nu_m\in S^1\subset\mathbb{C}$ so that if
$$\Lambda_1(z,\delta)=\int_{|\eta-z|<\delta} h(\eta)dm(\eta),\qquad \Lambda_2(z,\delta) = \sup_{|\nu|=1} \sum_{j=2}^m |\nabla_\nu^{j-2}h(z)|\delta^j,$$
$$\Lambda_3(z,\delta)=\sum_{j=2}^m \Bigg(\sum_{k=0}^m |\nabla_{\nu_k}^{j-2}h(z)|\Bigg)\delta^j,\quad \Lambda_4(z,\delta) = \sum_{j=2}^m \Bigg(\sum_{k=0}^{j-2} \Big| \frac{\partial^{j-2}h}{\partial z^k \partial \bar{z}^{j-2-k}}(z)\Big|\Bigg)\delta^{j},$$
then, uniformly in $0<\delta\leq \delta_0$ and $\mb{z}=(z,z_2)\in{\rm b}\Omega$,
$$\Lambda(\mb{z},\delta)\approx\Lambda_1(z,\delta)\approx\Lambda_2(z,\delta)\approx\Lambda_3(z,\delta)\approx\Lambda_4(z,\delta).$$\end{lemma}

The proof of Lemma \ref{lem:metric-approxsmall} uses the following two elementary facts. The first allows us to interchange mixed partial derivatives with linear combinations of directional derivatives.

\begin{proposition}\label{prop:metric-vandermonde} Fix $j\geq 0$ and $z\in\mathbb{C}$, and suppose that $f$ is smooth in a neighborhood of $z.$
	\begin{itemize}
		\item[(a)] If $\nu\in S^1\subset \mathbb{C}$, then $$\nabla^j_\nu f(z) = \sum_{k=0}^{j} \binom{j}{k} \nu^k \bar{\nu}^{j-k} \frac{\partial^j f}{\partial z^k \partial \bar{z}^{j-k}}(z).$$
		\item[(b)] If $\nu_0,\ldots,\nu_j\in S^1\subset \mathbb{C}$ are chosen so that the $\nu_n^2$ are distinct, then there exist constants $a(n,k)$ for $0\leq n,k\leq j$, with $$\frac{\partial^j f}{\partial z^k \partial \bar{z}^{j-k}}(z) = \sum_{n=0}^j a(n,k)\nabla_{\nu_n}^j f(z),$$
		where $|a(n,k)|\leq\displaystyle (j!)^2 \Big(\min_{\alpha\neq\beta} |\nu_\alpha^2-\nu_\beta^2|\Big)^{-\frac{j(j+1)}{2}}.$
	\end{itemize}
\end{proposition}
\begin{proof}
	Part (a) follows immediately from the formula $\nabla_\nu f(z)=\nu f_z(z)+\bar{\nu}f_{\bar{z}}(z)$ and induction.
	
	For (b), define vectors $\mb{D},\mb{H},\mb{V}^{n}\in\mathbb{C}^{j+1}$ for $n=0,1,\ldots,j$ by $$\displaystyle \mb{D}_{k+1} = \nabla_{\nu_k}^jf(z),\ \mb{H}_{k+1}=\binom{j}{k}\frac{\partial^j f}{\partial z^k \partial \bar{z}^{j-k}}(z),$$
	and $$\mb{V}^n_{k+1} = \nu_n^k\bar{\nu}_n^{j-k}=\nu_n^{2k-j}\quad\mbox{for}\quad 0\leq n,k\leq j.$$
	Then if $A_\nu$ is the $(j+1)\times(j+1)$ matrix with rows $\mb{V}^0,\mb{V}^1,\ldots,\mb{V}^{j}$, then $A_\nu \mb{H}=\mb{D}$ and 
	\begin{align*}
	{\rm det}A_\nu & =  (\nu_0\nu_1\cdots\nu_j)^{-j}{\rm det}\begin{bmatrix} 1 & \nu_0^2 & \nu_0^4 & \cdots & \nu_0^{2j}\\ 1 & \nu_1^2 & \nu_1^4 & \cdots & \nu_1^{2j}\\ 1 & \nu_2^2 & \nu_2^4 & \cdots & \nu_2^{2j} \\
	\vdots & \vdots & \vdots & \ddots & \vdots \\ 1 & \nu_j^2 & \nu_j^4 & \cdots & \nu_j^{2j}\end{bmatrix} \\
	& =  (\nu_0\nu_1\cdots\nu_j)^{-j}\prod_{0\leq \beta<\alpha\leq j} (\nu_\alpha^2-\nu_\beta^2)\neq 0\end{align*}
	because the numbers $\nu_k^2$ are distinct, where we have used the formula for the determinant of a Vandermonde matrix.
	
	Therefore $A_\nu$ is invertible, and $\mb{H} = A_{\nu}^{-1}\mb{D}$. We now estimate the entries of $A_{\nu}^{-1}$. Note that if $A_{\nu}^{p,q}$ denotes the $pq$-th minor of $A_\nu$, then $|A_{\nu}^{p,q}|\leq j!$ because all of the entries of $A_\nu$ have unit modulus. By the well-known formula for the classical adjoint and the above explicit formula for ${\rm det}A_\nu$, 
	$$|(A_{\nu}^{-1})_{pq}|\leq j!\Big|\prod_{0\leq \beta<\alpha\leq j} (\nu_\alpha^2-\nu_\beta^2)\Big|^{-1}\leq j!\Big(\min_{\alpha\neq \beta} |\nu_\alpha^2-\nu_\beta^2|\Big)^{-\frac{j(j+1)}{2}},$$ and therefore the constants $a(n,k)$ in $$\frac{\partial^j f}{\partial z^k \partial \bar{z}^{j-k}}(z) = \sum_{n=0}^j a(n,k)\nabla_{\nu_n}^j f(z)$$  are bounded by 
	\begin{align*}
		|a(n,k)|& \leq \binom{j}{k}^{-1} j!\Big(\min_{\alpha\neq \beta} |\nu_\alpha^2-\nu_\beta^2|\Big)^{-\frac{j(j+1)}{2}}\\
		& = k!(j-k)!\Big(\min_{\alpha\neq \beta} |\nu_\alpha^2-\nu_\beta^2|\Big)^{-\frac{j(j+1)}{2}}\leq (j!)^2 \Big(\min_{\alpha\neq\beta} |\nu_\alpha^2-\nu_\beta^2|\Big)^{-\frac{j(j+1)}{2}}
		\end{align*}
	as desired.
	\end{proof}

The second elementary fact allows us to choose, for a fixed $C^{J}(\mathbb{C})$ function $h$ and $z\in\mathbb{C}$, a direction $\nu_\ast\in S^1\subset \mathbb{C}$ so that $|\nabla_\nu^j h(z)|$ is essentially maximal for $j=0,1,\ldots,J$.

\begin{proposition}\label{prop:metric-maxdirection}
	Fix $J\in\mathbb{N}$ and let $h:\mathbb{C}\to\mathbb{C}$ be $C^J$. Then there exists a constant $C(J)>0$, independent of $h$, and for each $z\in\mathbb{C}$ there exists a direction $\nu_\ast\in S^1\subset\mathbb{C}$ so that
	$$|\nabla_{\nu_\ast}^j h(z)|\geq C(J)\sum_{k=0}^{j} \Big| \frac{\partial^j h}{\partial z^k \partial \bar{z}^{j-k}}(z)\Big|,\quad j=0,\ldots,J.$$
	In particular, taking $J=m-2$ and $C=\min\limits_{0\leq J\leq m-2} C(J)$ we have
	$$\sum_{j=2}^{m} |\nabla_{\nu_\ast}^{j-2}h(z)|\delta^j \geq C\sum_{j=2}^{m} \Bigg(\sum_{k=0}^{j-2} \Big| \frac{\partial^{j-2} h}{\partial z^k \partial \bar{z}^{j-2-k}}(z)\Big|\Bigg)\delta^j,\quad 0<\delta<+\infty.$$
	\end{proposition}
\begin{proof}
	
	We claim that for each $j\geq 0$ and $c\in (0,2\pi)$ there exists $C=C(j,c)>0$ so that, if $\sigma$ denotes arc-length measure on $S^1\subset \mathbb{C}$, 
	\begin{equation}\label{eq:metric-maxdir} 
		\sigma\Big( \Big\{ \nu\in S^1\ :\ |\nabla_{\nu}^j h(z)|\geq C(j,c)\sum_{k=0}^{j} \Big| \frac{\partial^j h}{\partial z^k \partial \bar{z}^{j-k}}(z)\Big|\Big\}\Big) > c.\end{equation}
	Granting this for the moment, for $j\geq 0$ we set $c=2\pi - 2^{-j-1}$ and choose $C(j,2\pi - 2^{-j-1})$ accordingly. If $$A(j)=\Big\{\nu\in S^1\ :\ |\nabla_\nu^j h(z)|\geq C(j,2\pi-2^{-j-1})\sum_{k=0}^{j} \Big| \frac{\partial^j h}{\partial z^k \partial \bar{z}^{j-k}}(z)\Big|\Big\},$$ then
	\begin{align*}
	\sigma\Bigg(\bigcap_{j=0}^{m-2} A(j)\Bigg) & = 2\pi - \sigma\Bigg(\bigcup_{j=0}^{m-2} A(j)^c\Bigg) \geq 2\pi - \sum_{j=0}^{m-2} \sigma(A(j)^c) \\
	& \geq 2\pi - \sum_{j=0}^{m-2} 2^{-j-1} \geq 2\pi -1>0,\end{align*}
	so we can choose $\displaystyle \nu_\ast\in \bigcap _{j=0}^{m-2} A(j)$. Setting $C=\min\limits_{0\leq j\leq m-2} C(j,2\pi -2^{-j-1})$, for this particular $\nu_\ast$, we have
	$$|\nabla_{\nu_\ast}^{j-2} h(z)|\geq C\sum_{k=0}^{j-2} \Big| \frac{\partial^{j-2} h}{\partial z^k \partial\bar{z}^{j-2-k}}(z)\Big|,\quad j=2,3,\ldots,m,$$
	and therefore
	$$\sum_{j=2}^m |\nabla_{\nu_\ast}^{j-2} h(z)|\delta^j\geq C\sum_{j=2}^m \Bigg( \sum_{k=0}^{j-2} \Big| \frac{\partial^{j-2} h}{\partial z^k \partial\bar{z}^{j-2-k}}(z)\Big|\Bigg) \delta^j,\quad 0<\delta<+\infty.$$
	
	It therefore suffices to establish (\ref{eq:metric-maxdir}).
	Fix $z\in \mathbb{C}$. For $\nu\in S^1$, part (a) of Proposition \ref{prop:metric-vandermonde} implies that
	$$\nu^j \nabla_\nu^{j} h(z) = \sum_{k=0}^j \binom{j}{k} \frac{\partial^j h}{\partial z^k \partial \bar{z}^{j-k}}(z)\nu^{2k}.$$
	Write $\displaystyle b_k = \binom{j}{k} \frac{\partial^j h}{\partial z^k \partial \bar{z}^{j-k}}(z),$
	so that 
	$$\nu^j \nabla_\nu^{j} h(z) = \sum_{k=0}^j b_k\nu^{2k}.$$
	
	If $b_k=0$ for $k=0,1,\ldots,j$, then for any choice of $c$ (\ref{eq:metric-maxdir}) holds for any $C(j,c)>0$. We may therefore assume without loss of generality that $b_k\neq 0$ for some $k$. 
	Define $B=\sum_{k=0}^j |b_k|$, and $a_k=B^{-1}b_k$. Defining $S_{\ell^1(\mathbb{C}^{N})}(0,1)=\{\mb{a}\in \mathbb{C}^{N}\ :\ \|\mb{a}\|_{\ell^1}=1\}$ to be the unit sphere in $\mathbb{C}^{N}$ in the $\ell^1$-norm, we have
	$$h_{\mb{a}}(\nu) := B^{-1}\nu^j \nabla_\nu^{j} h(z) = \sum_{k=0}^j a_k\nu^{2k}\in \mathcal{P}(2j,1),$$
	where $\mathcal{P}(m,1)\displaystyle :=\{ p_{\mb{A}}(\nu)=\sum_{k=0}^{m} A_k \nu^k\ :\ \mb{A}\in S_{\ell^1(\mathbb{C}^{m+1})}(0,1)\}.$
	
	We now show that if $c\in (0,2\pi)$, then there exists $C>0$ so that for $p_{\mb{A}}\in \mathcal{P}(2j,1)$, 
	\begin{equation}\label{eq:metric-maxdirhyp} 
		\sigma(\{\ \nu\ :\ |p_{\mb{A}}(\nu)|\geq C\}) > c.
		\end{equation}
	To see this, suppose that on the contrary that for all $C>0$ there exists $\mb{A}(C)\in S_{\ell^1(\mathbb{C}^{2j+1})}(0,1)$  with 
	$$\sigma( \{ \nu\in S^1\ :\ |p_{\mb{A}(c)}(\nu)|\geq C\})\leq c.$$
	Define a smooth, nondecreasing function $\chi:\mathbb{R}\to [0,1]$ with $$\chi(t) = \begin{cases} 1 & \mbox{if}\ t\geq 2,\\ 0 & \mbox{if}\ t\leq 1,\end{cases}$$
	and for $\mathbf{A}\in \mathbb{C}^{2j+1}$ define
	$$f_C(\mb{A}) = \int_{S^1} \chi(C^{-1} |p_{\mb{A}}(\nu)|) d\sigma(\nu).$$
	Note that
	\begin{itemize}
		\item[(a)] $C\mapsto f_C(\mb{A})$ is non-increasing.
		\item[(b)] $\mb{A}\mapsto f_C(\mb{A})$ is continuous in $\mb{A}$.
		\item[(c)] $\sigma(\{ \nu\ :\ |p_{\mb{A}}(\nu)|\geq 2C\})\leq f_C(\mb{A})\leq \sigma(\{ \nu\ :\ |p_{\mb{A}}(\nu)|\geq C\}).$
	\end{itemize}
	Let $$H(C)= \{\mb{A}\ :\ \mb{A}\in S_{\ell^1(\mathbb{C}^{2j+1})}(0,1)\ \mbox{and}\ f_C(\mb{A})\leq c\}.$$
	By (b), $H(C)$ is a closed subset of $S_{\ell^1(\mathbb{C}^{2j+1})}(0,1)$, and is therefore compact. Moreover, (a) implies that $H(C')\subseteq H(C)$ for $C'\leq C$.
	By choosing $\mb{A}(C)$ for a sequence $C\to 0$ and passing to a convergent subsequence, we see that there exists $\displaystyle \mb{A}\in \bigcap_{C>0} H(C)$. But then
	$$\sigma(\{\nu\ :\ |p_{\mb{A}}(\nu)|>0\}) = \lim_{C\to 0^+} \sigma(\{\nu\ :\ |p_{\mb{A}}(\nu)|\geq 2C\}) \leq c<2\pi,$$
	so that
	$$\sigma(\{\nu\in S^1\ :\ |p_{\mb{A}}(\nu)|=0\}>0.$$
	By interpolation we have $\mb{A}=\mb{0}$, contradicting the fact that $\mb{A}\in S_{\ell^1(\mathbb{C}^{2j+1})}(0,1)$. 
	This establishes (\ref{eq:metric-maxdirhyp}).
	
	To see how (\ref{eq:metric-maxdir}) follows, we need only note that
		\begin{equation*}
			\sigma\Big( \Big\{ \nu\in S^1\ :\ |\nabla_{\nu}^j h(z)|\geq C\sum_{k=0}^{j} \Big| \frac{\partial^j h}{\partial z^k \partial \bar{z}^{j-k}}(z)\Big|\Big\}\Big) = \sigma\Big( \{ \nu\in S^1\ :\ |h_{\mb{a}}(\nu)|\geq C\}\Big).
			\end{equation*}
\end{proof}

We also recall the following result from \cite{BrunaNagelWainger1988} for convex functions of one variable.

\begin{proposition}[{\cite{BrunaNagelWainger1988}, Lemmas 3.2 and 3.3}]\label{prop:geom-obj-convex}
	Suppose that $Q(t)=\displaystyle\sum_{j=0}^m a_j t^j + R(t)$ is convex on $[0,T]$, such that $Q(0)=a_0=0$, $Q'(0)=a_1=0$. We assume that $|R^{(k)}(t)|\leq C_k t^{m+1-k}$ for $0\leq k\leq m+1$, and that $\displaystyle\sum_{j=2}^m |a_j| \approx 1$. Then there is a positive constant $C=C(m,C_k)$ such that, for $0\leq t\leq \min(C,T)$,
	\begin{align}
		Q(t)& \approx\ \displaystyle\sum_{j=2}^m |a_j|t^j,\label{eq:geom-obj-polybound}\\
		Q'(t)& \approx\ \displaystyle\sum_{j=2}^m |a_j|t^{j-1}.\label{eq:geom-obj-polyderivbound}
	\end{align}
\end{proposition}

\begin{remark}{\rm The result in \cite{BrunaNagelWainger1988} actually shows that, for example, $$Q(t)\gtrsim \displaystyle\sum_{j=2}^m |a_j|t^j + At^{m+1}.$$
		Because the sum of the $|a_j|$ is comparable to $1$, the `junk' term $At^{m+1}$ is negligible if $t$ is sufficiently small, which yields the result above.}
\end{remark}

Before we give the proof of Lemma \ref{lem:metric-approxsmall}, we need to recall some terminology from \cite{Peterson2014}.
We say that a set $A\subset \mathbb{C}$ is a \emph{pen} if it is open, connected, simply connected, and if ${\rm b}A$ is FPWS (i.e. it is locally parametrized by a continuous function with $FPWS$ velocity).
Let $L({\rm b}A)$ denote the perimeter of $A$.
Then for $z\in \mathbb{C}$ and $\delta>0$, a finite collection of pens $R=\lb R_1,\ldots,R_N\rb$ is called a $(z,\delta)-stockyard$ if
$$z\in \bigcup_{i=1}^N {\rm b}R_i,\quad \displaystyle\sum_{i=1}^N L({\rm b}R_i)\leq \delta,\quad\mbox{and}\quad \bigcup_{i=1}^N {\rm b}R_i\ \mbox{is connected}.$$
One of the main results of \cite{Peterson2014} characterizes $\Lambda(\mb{z},\delta)$ in terms of $(z,\delta)-$stockyards.

\begin{theorem}[{\cite{Peterson2014}, Theorem 1.1}]\label{thm:geom-obj-stock} $$\displaystyle \Lambda(\mb{z},\delta)=\displaystyle\sup_{(z,\delta)-stockyards\, R} \displaystyle\sum_{R_i\in R} \int_{R_i} h(\eta)dm(\eta).$$\end{theorem}
We are now ready to prove Lemma \ref{lem:metric-approxsmall}.

\begin{proof}[Proof of Lemma \ref{lem:metric-approxsmall}]
	
Without loss of generality, we may assume that $\mb{z}=\mb{0}$. Let $\nu_n=\exp\Big(\frac{\pi i n}{m+1}\Big)$ for $n=0,1,2,\ldots,m$. Then the numbers $\nu_n^2$ are distinct, and therefore satisfy the hypotheses of Proposition \ref{prop:metric-vandermonde}.

We first claim that there is some $\delta_0>0$ such that 
\begin{equation}\label{eq:geom-obj-small}\Lambda_1(0,\delta)\approx \Lambda_2(0,\delta)\approx \Lambda_3(0,\delta)\approx \Lambda_4(0,\delta)\end{equation}
holds uniformly for all $0<\delta\leq \delta_0$, with constants that depend only on $m$ and the constants from (H1) and (H2).

To see this, first expand $h(\eta)$ in its Taylor series as
\begin{equation}\label{eq:metric-taylor}
h(\eta)=\displaystyle\sum_{j=2}^{m}\displaystyle\sum_{k=0}^{j-2} \frac{1}{k!(j-2-k)!} \frac{\partial^{j-2}h}{\partial \eta^k \partial \bar{\eta}^{j-2-k}}(0) \eta^k \bar{\eta}^{j-2-k} + R_{m-2}(\eta),\end{equation}
and choose $\nu_\ast=e^{i\theta_0}$ as in Proposition \ref{prop:metric-maxdirection}.
A simple size estimate and Proposition \ref{prop:metric-vandermonde} yield
\begin{align}
	\int_{|\eta|<\delta} h(\eta)dm(\eta) & \lesssim \displaystyle\sum_{j=2}^{m}\sum_{k=0}^{j-2} \Big| \frac{\partial^{j-2}h}{\partial \eta^k \partial \bar{\eta}^{j-2-k}}(0)\Big| \delta^{j} + \mathcal{O}(\|h\|_{C^{m-1}}\delta^{m+1}) \nonumber\\
	& \lesssim \sum_{n=0}^m \Big(\sum_{j=2}^m |\nabla_{\nu_n}^{j-2}h(0)|\Big)\delta^j + \mathcal{O}(\|h\|_{C^{m-1}}\delta^{m+1})\nonumber \\
	& \lesssim \sup_{|\nu|=1} \displaystyle\sum_{j=2}^m |\nabla_{\nu}^{j-2}h(0)| \delta^j + \mathcal{O}(\|h\|_{C^{m-1}}\delta^{m+1})\nonumber\\
	& \approx  \sup_{|\nu|=1} \displaystyle\sum_{j=2}^m  |\nabla_{\nu}^{j-2}h(0)| \delta^j \nonumber\\
	& \approx  \displaystyle\sum_{j=2}^m  |\nabla_{\nu_\ast}^{j-2}h(0)| \delta^j\label{eq:metric-upperbound}
\end{align}
if $0<\delta\leq \delta_0$, provided that $\delta_0$ is taken small enough (depending only on $C_1$ from (H1), $\|h\|_{C^{m-1}}$, and $m$).

We next show that $\Lambda_2(0,\delta)\lesssim \Lambda_1(0,\delta)$, which (together with (\ref{eq:metric-upperbound})) immediately implies (\ref{eq:geom-obj-small}) as long as $\delta_0>0$ is chosen sufficiently small.
Fix $c>0$ (to be chosen later),  for integers $\alpha,\beta$ define $\displaystyle C(\alpha,\beta,c)=\frac{2\sin(c|\alpha-\beta|)}{|\alpha-\beta|}$, and note that for $|\alpha|,|\beta|\leq m$ we have $|C(\alpha,\beta,c)-2c|\lesssim c^3$.
We compute
\begin{align}
	\Lambda_1(0,\delta) & =  \int_{|\eta|<\delta} h(\eta)dm(\eta) \nonumber \\
	& \geq \displaystyle\int_0^\delta \int_{|\theta-\theta_0|\leq c} rh(re^{i\theta})d\theta dr \nonumber \\
	& =  \displaystyle\int_0^\delta \int_{|\theta-\theta_0|\leq c} rh(re^{i\theta_0})d\theta dr + \displaystyle\int_0^\delta \int_{|\theta-\theta_0|\leq c} r(h(re^{i\theta})-h(re^{i\theta_0}))d\theta dr\nonumber \\
	& =  2c \Bigg[ \displaystyle\sum_{j=2}^{m} \frac{\delta^{j}}{j} \displaystyle\sum_{k=0}^{j-2} \frac{1}{k!(j-2-k)!} \frac{\partial^{j-2}h}{\partial \eta^k \partial \bar{\eta}^{j-2-k}}(0)\nu_\ast^k \bar{\nu}_\ast^{j-2-k} \\
	&  \qquad\qquad\qquad+ \int_0^\delta rR_{m-2}(re^{i\theta_0})dr\Bigg] \nonumber \\
	&  \qquad +\displaystyle\sum_{j=2}^{m}\frac{\delta^{j}}{j} \displaystyle\sum_{k=0}^{j-2} \frac{[C(k,j-2-k,c)-2c]}{k!(j-2-k)!} \frac{\partial^{j-2}h}{\partial \eta^k \partial \bar{\eta}^{j-2-k}}(0)\nu_\ast^k \bar{\nu}_\ast^{j-2-k} \nonumber \\
	&  \qquad +\int_0^\delta \int_{|\theta-\theta_0|\leq c} r[R_{m-2}(re^{i\theta})-R_{m-2}(re^{i\theta_0})]d\theta dr \nonumber \\
	& = 2c \Bigg[ \displaystyle\sum_{j=2}^{m} \frac{\delta^{j}}{(j-2)!j} \nabla_{\nu_\ast}^{j-2}h(0) + \int_0^\delta rR_{m-2}(re^{i\theta_0})dr\Bigg] \label{eq:geom-obj-inteform} \\
	&  \qquad +\displaystyle\sum_{j=2}^{m}\frac{\delta^{j}}{j} \displaystyle\sum_{k=0}^{j-2} \frac{[C(k,j-2-k,c)-2c]}{k!(j-2-k)!} \frac{\partial^{j-2}h}{\partial \eta^k \partial \bar{\eta}^{j-2-k}}(0)\nu_\ast^k \bar{\nu}_\ast^{j-2-k} \nonumber \\
	&  \qquad +\int_0^\delta \int_{|\theta-\theta_0|\leq c} r[R_{m-2}(re^{i\theta})-R_{m-2}(re^{i\theta_0})]d\theta dr, \nonumber
\end{align}
where we used (\ref{eq:metric-taylor}) in the third step and part (a) of Proposition \ref{prop:metric-vandermonde} in the fourth step.

Note that because
\begin{align*}\Big|\displaystyle\sum_{j=2}^{m}\frac{\delta^{j}}{j} \displaystyle\sum_{k=0}^{j-2} \frac{[C(k,j-2-k,c)-2c]}{k!(j-2-k)!}& \frac{\partial^{j-2}h}{\partial \eta^k \partial \bar{\eta}^{j-2-k}}(0)\nu_\ast^k \bar{\nu}_\ast^{j-2-k}\Big| \\
&\lesssim c^3\displaystyle\sum_{j=2}^{m}\sum_{k=0}^{j-2} \Big| \frac{\partial^{j-2}h}{\partial \eta^k \partial \bar{\eta}^{j-2-k}}(0)\Big| \delta^{j}= c^3 \Lambda_2(0,\delta)\end{align*}
and
$$\Big|\int_0^\delta \int_{|\theta-\theta_0|\leq c} r[R_{m-2}(re^{i\theta})-R_{m-2}(re^{i\theta_0})]d\theta dr\Big|\lesssim c\|h\|_{C^{m-1}}\delta^{m+1},$$
by (\ref{eq:metric-upperbound}) and by taking $\delta_0$ and $c$ sufficiently small it suffices to show that 
$$\Bigg[ \displaystyle\sum_{j=2}^{m} \frac{\delta^{j}}{(j-2)!j} \nabla_{\nu_\ast}^{j-2}h(0) + \int_0^\delta rR_{m-2}(re^{i\theta_0})dr\Bigg]\gtrsim \displaystyle\sum_{j=2}^m |\nabla_{\nu_\ast}^{j-2}h(0)|\delta^j.$$

To this end, for $t\geq 0$ define $H(t)$ as
$$H(t)=\int_0^t \Bigg[ \displaystyle\sum_{j=2}^{m} \frac{r^{j}}{(j-2)!j} \nabla_{\nu_\ast}^{j-2} h(0) + \int_0^r s R_{m-2}(se^{i\theta_0})ds\Bigg]dr.$$
Because $$H(0)=H'(0)=0,\quad H''(t)=t h(t e^{i\theta_0})\geq 0,$$
and $$\Big|\frac{d^k}{dt^k} \int_0^t \int_0^r s R_{m-2}(se^{i\theta_0})dsdr\Big|\leq C_2 t^{m+2-k},\quad k=0,\ldots,m+2,$$
we can apply conclusion (\ref{eq:geom-obj-polyderivbound}) of Proposition \ref{prop:geom-obj-convex} to obtain
\begin{align*}
	\displaystyle\sum_{j=2}^{m} \frac{t^{j}}{(j-2)!j} \nabla_{\nu_\ast}^{j-2}h(0) + \int_0^t rR_{m-2}(re^{i\theta_0})dr & \gtrsim  \displaystyle\sum_{j=2}^m |\nabla_{\nu_\ast}^{j-2}h(0)|t^j\end{align*}
for all $0<t\leq \delta_0$, if $\delta_0$ is taken to be sufficiently small (depending only on $C_1$, $\|h\|_{C^{m-1}}$, and $m$).

In particular, choosing $t=\delta$ we see that the right hand side of (\ref{eq:geom-obj-inteform}) is therefore bounded below by a constant multiple of
\begin{align*}
2c\displaystyle\sum_{j=2}^m |\nabla_{\nu_\ast}^{j-2}h(0)|\delta^j +\mathcal{O}\Big(c^2\displaystyle\sum_{j=2}^m &  |\nabla_{\nu_\ast}^{j-2}h(0)|\delta^j+c\delta^{m+1}\Big)\\
& \gtrsim c\displaystyle\sum_{j=2}^m  |\nabla_{\nu_\ast}^{j-2}h(0)|\delta^j\gtrsim c\Lambda_2(0,\delta),\end{align*}
provided that $c$ and $\delta$ are taken to be sufficiently small (depending only on $C_1$, $\|h\|_{C^{m-1}}$, $m$, and $\delta_0$).
This concludes the proof of the (\ref{eq:geom-obj-small}). 

We now show that for $0<\delta\leq \delta_0$ (for $\delta_0$ as above),
$$\Lambda(\mb{0},\delta)\approx \int_{|\eta|<\delta} h(\eta)dm(\eta).$$
We apply (\ref{eq:geom-obj-small}) to obtain
\begin{align*}
	\int_{|\eta|<\delta} h(\eta)dm(\eta) & \leq  \sup_{(0,10\delta)-stockyards\ R} \displaystyle\sum_{R_i\in R} \int_{R_i} h(\eta)dm(\eta) \\
	& \lesssim \Bigg( \sup_{|\eta|<10\delta} h(\eta) \Bigg)\delta^2 \\
	& \lesssim \sup_{|\eta|<10\delta}\displaystyle\sum_{j=2}^{m} \frac{|\eta|^{j-2}\delta^2}{(j-2)!}\nabla_{\frac{\eta}{|\eta|}}^{j-2}h(0)+ \mathcal{O}(\|h\|_{C^{m-1}}|\eta|^{m-1})\delta^2 \\
	& \lesssim \displaystyle\sup_{|\nu|=1}\sum_{j=2}^{m} |\nabla_{\nu}^{j-2}h(0)|\delta^j+ \mathcal{O}(\|h\|_{C^{m-1}}\delta^{m+1}) \\
	& \approx  \displaystyle\sup_{|\nu|=1}\sum_{j=2}^{m} |\nabla_{\nu}^{j-2}h(0)|\delta^j \\
	& \approx  \sum_{j=2}^{m} |\nabla_{\nu_\ast}^{j-2}h(0)|\delta^j\\
	& \approx  \int_{|\eta|<\delta} h(\eta)dm(\eta),
\end{align*}
which, after applying Theorem \ref{thm:geom-obj-stock}, concludes the proof of Lemma \ref{lem:metric-approxsmall}.
\end{proof}

\begin{remark}\label{rem:metric-approxhighorder}{\rm In the sequel it will be helpful to note that, for fixed $\delta_0>0$ and $\kappa\geq m$, we have
$$\sum_{j=2}^\kappa \Bigg(\sum_{k=0}^{j-2} \Big| \frac{\partial^{j-2}h}{\partial z^k \partial \bar{z}^{j-2-k}}(z)\Big|\Bigg)\delta^{j}\approx \Lambda(\mb{z},\delta)\quad \mbox{for}\ 0<\delta\leq\delta_0,$$
with constants that depend only on $\delta_0$, $\kappa$, and the constants appearing in (H1) and (H2).
}\end{remark}

\begin{proof}[Proof of Lemma \ref{lem:geom-uftisaq}]
The conclusion for $0<\delta\leq \delta_0$ follows immediately from Lemma \ref{lem:metric-approxsmall}. The proof is complete once one observes that (H1) and Lemma \ref{lem:metric-approxsmall} imply that $h(z)$ satisfies Lemma \ref{lem:geom-ugschar}(b), and therefore also satisfies parts (a) and (c).
\end{proof}

As an immediate corollary of Lemma \ref{lem:geom-uftisaq} and Lemma \ref{lem:metric-approxsmall}, we gain the ability (to be used in the proof of Theorem \ref{thm:szego-decomp}) to approximate the function $\tau\mapsto \mu(\mb{z},\tau)$ with one that has a specified bound on its growth rate.
\begin{corollary}\label{cor:metric-mutilde}
	There exist constants $0<c<C<+\infty$, and for fixed $\mb{z}\in \b\Omega$ there exists a non-decreasing function $\tau\mapsto \mu^\ast(\mb{z},\tau)$, such that $c\mu^\ast(\mb{z},\tau)\leq \mu(\mb{z},\tau)\leq C\mu^\ast(\mb{z},\tau)$ and $$\mu^\ast(\mb{z},2\tau)\leq 2^{\frac{1}{2}}\mu^\ast(\mb{z},\tau).$$
\end{corollary}
\begin{proof}
	By Lemmas \ref{lem:geom-uftisaq} and \ref{lem:metric-approxsmall} and (H1), for fixed $\mb{z}\in \b\Omega$ there exists a continuous function $$\delta\to \Lambda^\ast(\mb{z},\delta)=\begin{cases} \displaystyle\sum_{j=2}^m a_j \delta^j & \mbox{if}\ \delta\leq 1,\\ \delta^2 & \mbox{if}\ \delta\geq 1,\end{cases}$$
	with the $a_i\geq 0$, $a_2+\cdots+a_m=1$,  and $\Lambda^\ast(\mb{z},\delta)\approx \Lambda(\mb{z},\delta)$ for $0<\delta<+\infty$ with constants in the approximation depend only on the constants from (H1) and (H2).  If $\mu^\ast(\mb{z},\tau)$ satisfies
	$$\mu^\ast(\mb{z},\Lambda^\ast(\mb{z},\delta))=\delta=\Lambda^\ast(\mb{z},\mu^\ast(\mb{z},\delta)),$$
	then we have $\mu^\ast(\mb{z},\tau)\approx \mu(\mb{z},\tau)$, as desired.
	The inequality
	$$\Lambda^\ast(\mb{z},\mu^\ast(\mb{z},2\tau))=2\tau=2\Lambda^\ast(\mb{z},\mu^\ast(\mb{z},\tau))\leq \Lambda^\ast(\mb{z},2^{\frac{1}{2}}\mu^\ast(\mb{z},\tau)),$$
	then yields
	$$\mu^\ast(\mb{z},2\tau)\leq 2^{\frac{1}{2}}\mu^\ast(\mb{z},\tau).$$
\end{proof}

\subsection{Proof of Lemma \ref{lem:metric-approx}}\label{sec:metric-approx}

	Let $\t\Omega=\{\tilde{\mb{z}}\in\mathbb{C}^2\ :\ \Im(\t z_2)>P^{\mb{w},2}(\t z)\}$, and let $\Phi:\mathbb{C}^2\to\mathbb{C}^2$ be the associated biholomorphism in Section \ref{sec:normalization-biholomorphic}. If $\tilde{d}$ denotes the Carnot-Carath\'eodory metric on $\b\t\Omega$, then
	$$d(\mb{z},\mb{w})=\tilde{d}(\Phi(\mb{z}),\Phi(\mb{w}))=\t d(\Phi(\mb{z}),\mb{0}).$$
	Denote by $\tilde{\Lambda}$ and $\tilde{\mu}$ the analogues of $\Lambda$ and $\mu$ on $\b\t\Omega$, and note that because $\tilde{\Lambda}(\mb{0},\delta)=\Lambda(\mb{w},\delta)$ for $\delta>0$ by Remark \ref{rem:metric-invar}, we have $\tilde{\mu}(\mb{0},\delta)=\mu(\mb{w},\delta)$ as well.
	Let $$\tilde{\mb{z}}=(\t z,\t z_2)=\Phi(\mb{z})=(z-w,z_2-w_2-2i\sum_{j=1}^2 \frac{1}{j!} \frac{\partial^j P}{\partial z^j}(w)(z-w)^j).$$	
	Setting $\displaystyle \t T(\t z,0)=-2\Im\Big(\int_0^1 \t z P^{\mb{w},2}_{z}(\t z s)ds\Big),$ equation (\ref{eq:geom-metric}) yields
	$$\t{d}(\tilde{\mb{z}},\mb{0})\approx |\t z| + \t \mu(\mb{0},|\Re(\t z_2)-\t T(\t z,0)|).$$
	
	Note that when $|\t z|\gtrsim 1$ property (b)-(ii) of Proposition \ref{prop:potentials} and Lemma \ref{lem:geom-uftisaq} yield
	$$|\t T(\t z,0)|=\Big|2\Im\Big(\int_0^1 \t z \frac{\partial P^{\mb{w},2}}{\partial\t z}(\t zs)ds\Big)\Big|\lesssim |\t z|^2\approx \t\Lambda(\mb{0},|\t z|).$$
	Hence, either
	\begin{align*}
		|\Re(\t z_2)-\t T(\t z,0)| & \approx  |\Re(\t z_2)|\quad\mbox{or} \\
		|\Re(\t z_2)-\t T(\t z,0)| & \lesssim \tilde{\Lambda}(\mb{0},|\t z|),
	\end{align*}
	so that $\t d(\tilde{\mb{z}},\mb{0})\approx |\t z|+\t\mu(\mb{0},|\Re(\t z_2)|)$. For $|z-w|\gtrsim 1$ we therefore have
	\begin{align*}
	d(\mb{z},\mb{w}) & =\t d(\Phi(\mb{z}),\mb{0})\approx |\t z|+\t\mu(\mb{0},|\Re(\t z_2)|)\\
	& =|z-w|+\mu(\mb{w},|\Re(z_2)-\Re(w_2)-T_2(z,w)|).\end{align*}
	Because $$\displaystyle \Big|\frac{1}{2}\frac{\partial^2 P}{\partial z^2}(w)(z-w)^2\Big|\lesssim |z-w|^2\approx \Lambda(\mb{w},|z-w|)\quad\mbox{when}\ |z-w|\gtrsim 1,$$ we also have
	$$d(\mb{z},\mb{w})\approx |z-w|+\mu(\mb{w},|\Re(z_2)-\Re(w_2)-T_1(z,w)|)\quad\mbox{for}\ |z-w|\gtrsim 1,$$
	which proves (a).
	
	On the other hand, if $|\t z|\lesssim 1$ and $\kappa\geq m$ then 
	\begin{align*}
		& P^{\mb{w},2}(\t z)\\
		& =  2\Re\Bigg(\sum_{j=1}^m \frac{1}{j!} \frac{\partial^j P^{\mb{w},2}}{\partial \t z^j}(0)\t z^j\Bigg)+ \displaystyle\sum_{k=2}^m \frac{1}{k!} \displaystyle\sum_{j=1}^{k-1} \binom{k}{j} \frac{\partial^k P^{\mb{w},2}}{\partial\t z^j \partial \bar{\t z}^{k-j}}(0) \t z^j \bar{\t z}^{k-j} + R_m(\t z)\\
		& =  2\Re\Bigg(\sum_{j=1}^\kappa \frac{1}{j!} \frac{\partial^j P^{\mb{w},2}}{\partial \t z^j}(0)\t z^j\Bigg)+ \displaystyle\sum_{k=2}^m \frac{1}{k!} \displaystyle\sum_{j=1}^{k-1} \binom{k}{j} \frac{\partial^k P^{\mb{w},2}}{\partial\t z^j \partial \bar{\t z}^{k-j}}(0) \t z^j \bar{\t z}^{k-j}+\tilde{R}_{m,\kappa}(\tilde{z}),
	\end{align*}
	where
	$$\tilde{R}_{m,\kappa}(\tilde{z})=R_m(\t z)-2\Re\Bigg(\sum_{j=m+1}^\kappa \frac{1}{j!} \frac{\partial^j P^{\mb{w},2}}{\partial \t z^j}(0)\t z^j\Bigg).$$
	Then
	\begin{align*}
		\frac{\partial P^{\mb{w},2}}{\partial \t z}(\t z)& =  \sum_{j=1}^\kappa \frac{1}{(j-1)!} \frac{\partial^j P^{\mb{w},2}}{\partial \t z^j}(0)\t z^{j-1} + \frac{\partial \tilde{R}_{m,\kappa}}{\partial \t z}(\t z)\\
		& \qquad + \displaystyle\sum_{k=2}^m \frac{1}{(k-1)!} \displaystyle\sum_{j=1}^{k-1} \binom{k-1}{j-1} \frac{\partial^k P^{\mb{w},2}}{\partial\t z^j \partial \bar{\t z}^{k-j}}(0)\t z^{j-1} \bar{\t z}^{k-j},
	\end{align*}
	so that 
	\begin{align*}
	 -2\Im\Bigg( \int_0^1 &\t z \frac{\partial P^{\mb{w},2}}{\partial\t z}(\t zs)ds\Bigg)\\
		  = & -2\Im\Bigg(\sum_{j=1}^\kappa \frac{1}{j!} \frac{\partial^j P^{\mb{w},2}}{\partial \t z^j}(0)\t z^j\Bigg) -2\Im\Bigg(\int_0^1 \t z \frac{\partial \tilde{R}_{m,\kappa}}{\partial \t z}(\t z s)ds\Bigg)\\
		& \qquad -2\Im\Bigg( \displaystyle\sum_{k=2}^m \frac{1}{k!} \sum_{j=1}^{k-1} \binom{k-1}{j-1} \frac{\partial^k P^{\mb{w},2}}{\partial\t z^j \partial \bar{\t z}^{k-j}}(0)\t z^j \bar{\t z}^{k-j}\Bigg).
	\end{align*}
	Because $$-2\Im\Bigg(\sum_{j=1}^\kappa \frac{1}{j!} \frac{\partial^j P^{\mb{w},2}}{\partial \t z^j}(0)\t z^j\Bigg)=\t T_\kappa(\t z,0),$$
	$$\int_0^1 \t z \frac{\partial \tilde{R}_{m,\kappa}}{\partial \t z}(\t z s)ds=\mathcal{O}(\|h\|_{C^{\kappa-1}}|\t z|^{m+1}),$$
	and
	$$\Bigg|-2\Im\Bigg( \displaystyle\sum_{k=2}^m \frac{1}{k!} \sum_{j=1}^{k-1} \binom{k-1}{j-1} \frac{\partial^k P^{\mb{w},2}}{\partial\t z^j \partial \bar{\t z}^{k-j}}(0)\t z^j \bar{\t z}^{k-j}\Bigg)\Bigg|\lesssim \tilde{\Lambda}(\mb{0},|\t z|)$$
	by Lemma \ref{lem:metric-approxsmall}, we have
	$$\t T(\t z,0)-\t T_\kappa(\t z,0) = \mathcal{O}(\tilde{\Lambda}(\mb{0},|\t z|)).$$
	Hence, either
	\begin{align*}
		|\Re(\t z_2)-\t T(\t z,0)| & \approx  |\Re(\t z_2)-\t T_\kappa(\t z,0)|\quad\mbox{or} \\
		|\Re(\t z_2)-\t T(\t z,0)| & \lesssim \t\Lambda(\mb{0},|\t z|),
	\end{align*}
		so that $\t d(\tilde{\mb{z}},\mb{0})\approx |\t z|+\t\mu(\mb{0},|\t z-\t T_\kappa(\t z,0)|)$. This yields
		$$d(\mb{z},\mb{w})=\t d(\Phi(\mb{z}),\mb{0})\approx |\t z|+\t\mu(\mb{0},|\t z-\t T_\kappa(\t z,0)|)=|z-w|+\mu(\mb{w},|z-w-T_\kappa(z,w)|)$$
		for $|z-w|\lesssim 1,$ which proves (b).

\subsection{Proof of Lemma \ref{lem:metric-smooth}}\label{sec:metric-smooth}

By arguing as in the proof of Lemma \ref{lem:metric-approx}, it suffices to prove the result when $\mb{w}=\mb{0}$ and $P=P^{\mb{0},2}$. 
Fix $\mb{z}=(z,z_2)\in {\rm b}\Omega$, and define
$$g(\mb{z})=\Re(z_2)-T(z,0).$$
Choosing $\nu_0,\ldots,\nu_m$ to be $\nu_n=\exp(\frac{\pi i n}{m+1})$, we apply Lemma \ref{lem:metric-approxsmall} and Lemma \ref{lem:geom-uftisaq} to define a smooth, increasing function 
$$\tilde{\Lambda}(\delta)\approx\begin{cases} \displaystyle\sum_{j=2}^m \Bigg(\sum_{n=0}^m |\nabla_{\nu_n}^{j-2}h(0)|\Bigg)\delta^j\quad &\mbox{if}\ \delta \leq 1,\\ \delta^2\quad &\mbox{if}\ \delta \gg 1,\end{cases}$$
with $\delta^{-1}\tilde{\Lambda}(\delta)$ and $\delta^{-2}\tilde{\Lambda}(\delta)$ non-decreasing, and let $\tilde{\mu}(\delta)$ be the inverse function to $\delta\mapsto \tilde{\Lambda}(\delta)$.
We define
$$d^\ast_{small}(\mb{z},\mb{0}):=\tilde{\mu}\Big(\big(\tilde{\Lambda}(|z|^2)+g(\mb{z})^2\big)^\frac{1}{2}\Big)$$
and 
$$d^\ast_{large}(\mb{z},\mb{0}):=(|z|^4 + |\Re(z_2)|^2)^{\frac{1}{4}}.$$
By the explicit formula in Lemma \ref{lem:geom-uftisaq}, 
$$d(\mb{z},\mb{0})\approx \begin{cases} d^\ast_{small}(\mb{z},\mb{0}) \quad & \mbox{if}\ d(\mb{z},\mb{0})\lesssim 1,\\ d^\ast_{large}(\mb{z},\mb{0})\quad &\mbox{if}\ d(\mb{z},\mb{0})\gtrsim 1.\end{cases}$$
Let $\chi:[0,{+\infty})\rightarrow [0,1]$ be a smooth, non-increasing function such that $\chi(t)\equiv 1$ for $t\leq 1$ and $\chi(t)\equiv 0$ for $t\geq 2$.
We then define the $\tilde{d}$ on ${\rm b}\Omega$ via the formula
$$d^\ast:= \chi(d^\ast_{small})d^\ast_{small}+(1-\chi(\epsilon d^\ast_{large}))d^\ast_{large}.$$
By choosing $\epsilon$ sufficiently small, we can guarantee that $d\approx d^\ast$. 

It remains to show that the formulas in (b) hold. To this end, it suffices to estimate the derivatives of $d^\ast_{small}$ when $d^\ast\lesssim 1$ (so that $|z|\lesssim 1$), and the derivatives of $d^\ast_{large}$ when $d^\ast\gtrsim 1$. Because the derivatives of $d^\ast_{large}$ are much simpler than those of $d^\ast_{small}$, we only show the details for $d^\ast_{small}$.

We start with $d^\ast_{small}$.  Note first that
\begin{align*}
   -iZg(\mb{z})	& = P_{z}(z)-\int_0^1 P_{z}(zr)dr - \int_0^1 zrP_{z,z}(zr)dr +\int_0^1 \bar{z}rP_{z,\bar{z}}(zr)dr\\
	& = 2\int_0^1 \bar{z}rP_{z,\bar{z}}(zr)dr
\end{align*}
and
$$i\bar{Z}g(\mb{z})=2\int_0^1 zrP_{z,\bar{z}}(zr)dr.$$
Arguing as in the proof of Lemma \ref{lem:metric-approxsmall} and writing $P_{z,\bar{z}}=h$, for $|\alpha|=1$ we have
\begin{align*}
|W^\alpha g(\mb{z})| & =2|z|\int_0^1 rh(zr)dr \leq 2|z|\Big(\sup_{|\eta|\leq |z|} h(\eta) \Big) \\
& \lesssim \frac{1}{|z|}\int_{|\eta|<|z|}h(\eta)dm(\eta)\approx \frac{1}{|z|}\tilde\Lambda(|z|)\end{align*}
by Lemma \ref{lem:metric-approxsmall}.

For second-order derivatives when $|z|\lesssim 1$, we apply (\ref{eq:metric-taylor}) to see that
\begin{align*}
	& -iZZ g(\mb{z}) \\
	& =  2\bar{z}\partial_z \int_0^1 rh(zr)dr \\
	& =  2\bar{z}\partial_z \int_0^1 \sum_{j=2}^m \sum_{k=0}^{j-2} \frac{r^{j-1}}{k!(j-2-k)!} \frac{\partial^{j-2}h}{\partial z^k \partial \bar{z}^{j-2-k}}(0)z^k \bar{z}^{j-2-k} + rR_{m-2}h(zr)dr\\
	& = 2\bar{z} \int_0^1 \Bigg\{\sum_{j=2}^m \sum_{k=1}^{j-2} \frac{r^{j-1}}{(k-1)!(j-2-k)!} \frac{\partial^{j-2}h}{\partial z^k \partial \bar{z}^{j-2-k}}(0)z^{k-1} \bar{z}^{j-2-k} \\
	& \qquad\qquad\qquad\qquad\qquad\qquad+ r^2\frac{\partial R_{m-2}h}{\partial z}(zr)\Bigg\}dr,
\end{align*}
so that because $|z|\lesssim 1$ we have
$$|ZZg(\mb{z})|\lesssim \sum_{j=2}^{m} \Bigg(\sum_{k=0}^{j-2} \Big| \frac{\partial^{j-2}h}{\partial z^k \partial \bar{z}^{j-2-k}}(0)\Big|\Bigg) |z|^{j-2} + \mathcal{O}(\|h\|_{C^{m-1}}|z|^{m-1})\approx \frac{1}{|z|^2}\tilde{\Lambda}(|z|).$$
Similar computations show that
$$|\bar{Z}\bar{Z}g(\mb{z})|+|Z\bar{Z}g(\mb{z})|+|\bar{Z}Zg(\mb{z})|\lesssim\frac{1}{|z|^2}\tilde{\Lambda}(|z|)\quad\mbox{for}\ |z|\lesssim 1$$
as well. Note also that 
$$\displaystyle \Big|\frac{\partial^k}{\partial \delta^k} \tilde{\mu}(\sqrt{\delta})\Big|\approx \tilde{\mu}(\sqrt{\delta})\delta^{-k}\quad \mbox{for}\ k\geq 0.$$

Writing $g:=g(\mb{z}),$ and $d^\ast:=d^\ast(\mb{z},\mb{0}),$ we compute that for $d^\ast\lesssim 1$ (in which case $|z|\lesssim 1$ as well),
	\begin{align*}
		|Wd^\ast| & = \Big|W\t\mu \Big(\big(\t\Lambda(|z|^2)+g^2\big)^{\frac{1}{2}}\Big)\Big|\\
	    & \lesssim \frac{d^\ast}{\tilde{\Lambda}(|z|^2)+g^2}\Big(\frac{\tilde{\Lambda}(|z|^2)}{|z|}+|g|\frac{\tilde{\Lambda}(|z|)}{|z|}\Big) \\
		& \lesssim \frac{d^\ast}{\tilde{\Lambda}(|z|)+|g|}\frac{\tilde{\Lambda}(|z|)}{|z|}\\
		& \approx \frac{d^\ast}{\tilde{\Lambda}(d^\ast)}\frac{\tilde{\Lambda}(|z|)}{|z|} \lesssim 1,
	\end{align*}
	where in the third line we used the facts that $\tilde{\Lambda}(\delta^2)\approx (\tilde{\Lambda}(\delta))^2$ and $a^2+b^2\approx (a+b)^2$ for $a,b\geq 0$, and in the last inequality we used the fact that $\delta\mapsto \delta^{-1}\tilde{\Lambda}(\delta)$ is increasing.
	Similarly, we have
	\begin{align*}
		|W^2d^\ast|&\lesssim \frac{d^\ast}{\tilde{\Lambda}(d^\ast)^4}\Big( \frac{\tilde{\Lambda}(|z|)^2}{|z|}+|g|\frac{\tilde{\Lambda}(|z|)}{|z|}\Big)^2 + \frac{d^\ast}{\tilde{\Lambda}(d^\ast)^2}\Big( \frac{\tilde{\Lambda}(|z|)^2}{|z|^2}+|g|\frac{\tilde{\Lambda}(|z|)}{|z|^2}\Big)\\
		& \lesssim \frac{d^\ast}{\tilde{\Lambda}(d^\ast)^4} \tilde{\Lambda}(d^\ast)^2 \Big(\frac{\tilde{\Lambda}(d^\ast)}{d^\ast}\Big)^2 +\frac{d^\ast}{\tilde{\Lambda}(d^\ast)^2} \tilde{\Lambda}(d^\ast) \frac{\tilde{\Lambda}(d^\ast)}{d^\ast} \\
		& \lesssim (d^\ast)^{-1},
	\end{align*}
	which gives the desired estimates for $d^\ast\lesssim 1$.
	
	The proof for $d^\ast_{large}$ when $d^\ast\gtrsim 1$ is similar, except that one must estimate the derivative of $d^\ast_{large}$ and use the fact that $|P_{z}(z)|\lesssim |z|$.
	This concludes the proof of Lemma \ref{lem:metric-smooth}.

\begin{remark}{\rm Lemma \ref{lem:metric-smooth} is very similar to a result of Nagel and Stein \cite{NagelStein2001b}, who constructed such a metric $d^\ast$ on the boundary of compact domains and polynomial model domains.
In that setting, compactness allowed them to take higher order derivatives of $d^\ast$.
We do not pursue this here, because our results only necessitate control over the first and second order derivatives of $d^\ast$.}\end{remark}

\section{Definition of $\mathbb{F}_{K,K'}$ and $\mathbb{N}_{K,K'}$}\label{sec:fn}

Throughout this section, fix a UFT domain $\Omega=\{\mb{z}\in\mathbb{C}^2\ :\ \Im(z_2)>P(z)\}$ and assume that $P=P^{\mb{\sigma},\kappa}$ for some $\mb{\sigma},\kappa$.  The proof of Theorem \ref{thm:szego-decomp} is fundamentally reduced to the study of the weighted $\bar{\partial}$ equation on $\mathbb{C}$, carried out by Christ\footnote{Christ only considers the case $\tau=1$ in his paper, so the following discussion and the results in Section \ref{sec:kernels} should be interpreted (in the notation of \cite{Christ1991a}) as an application of Christ's results to the function $\phi(z)=\tau P(z)$ for $\tau>0$.} in \cite{Christ1991a}.
That is, for subharmonic $P:\mathbb{C}\rightarrow \mathbb{R}$ such that $4\pi \tau\Delta P dm$ is a doubling measure, he defines the (closed, densely defined) operators
\begin{equation}\label{eq:fn-zzbar}\hat{\bar{Z}}=\bar{D}_\tau=\bar{\partial}+2\pi \tau P_{\bar{z}},\qquad -\hat{Z}=D_\tau=-\partial + 2\pi \tau P_{z}\end{equation}
on $L^2(\mathbb{C})$, which are equivalent to $\bar{\partial}$ and $-\partial$ acting on $L^2(\mathbb{C}; e^{-4\pi \tau P(z)}dm(z))$.
He then carefully studies the operators 
\begin{equation}\label{eq:dbarops}
	\begin{aligned}
	& \mathbb{G}=(\bar{D}_\tau D_\tau)^{-1},\quad  \mathbb{R}_\tau=D_\tau\circ(\bar{D}_\tau D_\tau)^{-1},\\
	&\mathbb{R}_\tau^\ast =(\bar{D}_\tau D_\tau)^{-1}\circ \bar{D}_\tau, \quad \mbox{and}\quad  \mathbb{S}=I-D_\tau\circ(\bar{D}_\tau D_\tau)^{-1}\bar{D}_\tau,
\end{aligned}
\end{equation}
giving pointwise bounds on their Schwartz kernels in terms of a metric $\rho_\tau:\mathbb{C}\times\mathbb{C}\to[0,+\infty)$ and a closely related smooth function $\sigma_\tau:\mathbb{C}\to [0,+\infty)$ defined by 
$$\int_{|\eta-z|<\sigma_\tau(z)} 4\pi \tau\Delta P(\eta)dm(\eta)\approx 1$$
uniformly in $z\in\mathbb{C}$, and which measures the local average degeneracy of $\tau\Delta P$.

\begin{remark}{\rm For the UFT domain $\Omega$, the fact that $\tau\Delta P$ is a doubling measure on $\mathbb{C}$ (with doubling constant independent of $\tau$) and the existence of the function $\sigma_\tau(z)$ follow immediately from Lemma \ref{lem:geom-uftisaq}.}\end{remark}
 
 In this section we use Christ's weighted operators to define the operators $\mathbb{F}_{K,K'}$ and $\mathbb{N}_{K,K'}$ appearing in the statement of Theorem \ref{thm:szego-decomp}.
 More detailed information about $\sigma_\tau$, $\rho_\tau$, and the pointwise bounds for $[\mathbb{S}_\tau]$, $[\mathbb{G}_\tau]$, $[\mathbb{R}_\tau]$, and $[\mathbb{R}^\ast_\tau]$ can be found in Section \ref{sec:kernels}, where we give pointwise bounds on the Schwartz kernels of $(Z^\alpha \mathbb{F}_{K,K'}Z^\beta)^\wedge$ and $(Z^\alpha \mathbb{N}_{K,K'}Z^\beta)^\wedge$. 
 
 We begin by making three basic observations.

	\begin{lemma}\label{lem:unifL2bound} Let $\displaystyle A(\tau)=\sup_{z\in\mathbb{C}} \sigma_\tau (z)$. Then $A(\tau)<+\infty$ and is non-increasing for $\tau\in (0,+\infty)$. Moreover, there exists a constant $C>0$, depending only on the constants appearing in (H1)-(H3) (but independent of $\tau$), such that as operators from $L^2(\mathbb{C})\to L^2(\mathbb{C})$,
	$$\|\mathbb{G}_\tau\|\leq C A(\tau)^2,\quad \|\mathbb{R}_\tau\|+\|\mathbb{R}^\ast_\tau\|\leq CA(\tau),\quad \mbox{and}\ \|\mathbb{S}_\tau\|\leq 1.$$
	\end{lemma}

	\begin{proposition}\label{prop:continuity} As operators on $L^2(\mathbb{C})$,  $\mathbb{G}_\tau,$ $\mathbb{R}_\tau$, $\mathbb{R}_\tau^\ast$, and $\mathbb{S}_\tau$ are strongly continuous in $\tau\in (0,+\infty)$.  That is, for fixed $f\in L^2(\mathbb{C})$ and any operator $\mathbb{H}_\tau$ listed above, $$\lim_{h\to 0} \mathbb{H}_{\tau+h}f=\mathbb{H}_\tau f\quad\mbox{in}\ L^2(\mathbb{C}).$$
\end{proposition}

\begin{proposition}\label{prop:link}For almost every $(z,w,\tau)\in \mathbb{C}\times\mathbb{C}\times\mathbb{R}$, $$[\hat{\mathbb{S}}](z,w,\tau)=\begin{cases} [\mathbb{S}_\tau](z,w) & \mbox{if}\ \tau>0,\\ 0 & \mbox{if}\ \tau \leq 0.\end{cases}$$\end{proposition}

The proofs of Lemma \ref{lem:unifL2bound}, Proposition \ref{prop:continuity}, and Proposition \ref{prop:link} are provided in Appendix \ref{sec:link}.

	Now fix a smooth, non-increasing function $\chi:\mathbb{R}\to [0,1]$ that satisfies $\chi(\tau)=1$ for $\tau\leq 1$ and $\chi(\tau)=0$ for $\tau\geq 2$, and formally define the operators $\hat{\mathbb{F}}_{K,K'}$ and $\hat{\mathbb{N}}_{K,K'}$ on $L^2(\mathbb{C}\times\mathbb{R})$ via
	$$\hat{\mathbb{F}}_{K,K'}[f](z,\tau) = \int_{\mathbb{C}}\chi(\tau)[\mathbb{R}_\tau^K \mathbb{S}_\tau (\mathbb{R}_\tau^\ast)^{K'}](z,w)f(w,\tau)dm(w)$$
	and
	$$\hat{\mathbb{N}}_{K,K'}[f](z,\tau) = \int_{\mathbb{C}}(1-\chi(\tau))[\mathbb{R}_\tau^K \mathbb{S}_\tau (\mathbb{R}_\tau^\ast)^{K'}](z,w)f(w,\tau)dm(w)$$
	for $\tau> 0$, and $$\hat{\mathbb{F}}_{K,K'}[f](z,\tau)\equiv \hat{\mathbb{N}}_{K,K'}[f](z,\tau)\equiv 0\quad \mbox{for}\ \tau\leq 0.$$
	By Lemma \ref{lem:unifL2bound}, 
	$$\|\mathbb{R}_\tau^K\mathbb{S}_\tau(\mathbb{R}_\tau^\ast)^{K'}\|_{L^2(\mathbb{C})\to L^2(\mathbb{C})}\lesssim A(\tau)^{K+K'},\quad \tau>0,$$
	we therefore immediately have that $\hat{\mathbb{F}}_{K,K'}$ and $\hat{\mathbb{N}}_{K,K'}$ are closed and densely defined operators on $L^2(\mathbb{C}\times\mathbb{R})$. 
	
	Finally, define $$\mathbb{F}_{K,K'}=\Pi^\ast\circ \mathcal{F}^{-1}\circ\hat{\mathbb{F}}_{K,K'}\circ\mathcal{F}\circ (\Pi^{-1})^\ast$$
	and
	$$\mathbb{N}_{K,K'}=\Pi^\ast\circ \mathcal{F}^{-1}\circ\hat{\mathbb{N}}_{K,K'}\circ\mathcal{F}\circ (\Pi^{-1})^\ast.$$
	By (\ref{eq:fn-zzbar}) we immediately have
	$$(Z^\alpha \mathbb{F}_{K,K'}Z^\beta)^{\wedge}=W_\tau^\alpha \hat{\mathbb{F}}_{K,K'}W_\tau^\beta\quad\mbox{and}\quad
	(Z^\alpha \mathbb{N}_{K,K'}Z^\beta)^{\wedge}=W_\tau^\alpha \hat{\mathbb{N}}_{K,K'}W_\tau^\beta,$$
	so that
	\begin{equation}\label{eq:fn-fkernel}
		[(Z^\alpha \mathbb{F}_{K,K'}Z^\beta)^{\wedge}](z,w,\tau)=\chi(\tau)[W_\tau^\alpha \mathbb{R}_\tau^K\mathbb{S}_\tau(\mathbb{R}_\tau^\ast)^{K'}W_\tau^\beta](z,w)
		\end{equation}
	and
	\begin{equation}\label{eq:fn-nkernel}
		[(Z^\alpha \mathbb{N}_{K,K'}Z^\beta)^{\wedge}](z,w,\tau)=(1-\chi(\tau))[W_\tau^\alpha \mathbb{R}_\tau^K\mathbb{S}_\tau(\mathbb{R}_\tau^\ast)^{K'}W_\tau^\beta](z,w),
		\end{equation}
	where for a multi-index $\alpha$ we have $Z^\alpha=Z_{\alpha_1}\cdots Z_{\alpha_\ell}$ and $D_\tau^\alpha=D_{\tau,\alpha_1}\cdots D_{\tau,\alpha_\ell}$, where $Z_{\alpha_i}\in\{\bar{Z},Z\}$ and $D_{\tau,\alpha_i}=\hat{Z}_{\alpha_i}$.
	
	In particular,
	\begin{align}\label{eq:fn-fkkkernel}
		& T^N Z_{\mb{z}}^\alpha (Z^\beta_{\mb{w}})^\ast [\mathbb{F}_{K,K'}](\mb{z},\mb{w})\\
		& =(2\pi i)^N\int_0^{{+\infty}} e^{2\pi i\tau (\Re(z_2)-\Re(\bar{w}_2))}\chi(\tau)\tau^N[W_\tau^\alpha \mathbb{R}_\tau^K\mathbb{S}_\tau(\mathbb{R}_\tau^\ast)^{K'}W_\tau^\beta](z,w)d\tau\nonumber
	\end{align}
	and
	\begin{align}\label{eq:fn-nkkkernel}
		& T^N Z_{\mb{z}}^{\alpha} (Z^\beta_{\mb{w}})^{\ast}[\mathbb{N}_{K,K'}](\mb{z},\mb{w})\\
		& =(2\pi i)^N\int_0^{{+\infty}} e^{2\pi i\tau (\Re(z_2)-\Re(\bar{w}_2))}(1-\chi(\tau))\tau^N[W_\tau^\alpha \mathbb{R}_\tau^K\mathbb{S}_\tau(\mathbb{R}_\tau^\ast)^{K'}W_\tau^\beta](z,w)d\tau.\nonumber
	\end{align}	
	
	\begin{remark}{\rm By defining $\mathbb{F}=\mathbb{F}_{0,0}$ and $\mathbb{N}=\mathbb{N}_{0,0}$, we have $\mathbb{S}=\mathbb{N}+\mathbb{F}$ as well as parts (a)-(i) and (b)-(i) of Theorem \ref{thm:szego-decomp}.}\end{remark}
		
	Let $\Phi:\mathbb{C}^2\to\mathbb{C}^2$, $\Phi(\Omega)=\tilde{\Omega}$ be the biholomorphism constructed in Section \ref{sec:normalization-biholomorphic} associated to $\mb{\sigma}\in \b\Omega$ and $\kappa\geq 2$. Then Remark \ref{rem:metric-invar}, Lemma \ref{lem:normalization-schwartz}, and the above discussion immediately imply that
\begin{lemma}\label{lem:fn-substitution}
	If $$\mathbb{F}^{\alpha,\beta,N}_{K,K'}=T^N Z^\alpha\mathbb{F}_{K,K'}Z^\beta\ ,\qquad \mathbb{N}^{\alpha,\beta,N}_{K,K'}=T^N Z^\alpha\mathbb{N}_{K,K'}Z^\beta,$$ and 
	$$\tilde{\mathbb{F}}^{\alpha,\beta,N}_{K,K'}=T^N_{\b\t\Omega}Z_{\b\t\Omega}^\alpha\mathbb{F}^{\b\t\Omega}_{K,K'}Z_{\b\t\Omega}^\beta\ ,\qquad \tilde{\mathbb{N}}^{\alpha,\beta,N}_{K,K'}=T^N_{\b\t\Omega}Z_{\b\t\Omega}^\alpha\mathbb{N}^{\b\t\Omega}_{K,K'}Z_{\b\t\Omega}^\beta,$$
	Then
	\begin{equation}\label{eq:fn-subn}[\hat{\mathbb{N}}^{\alpha,\beta,N}_{K,K'}](z,w,\tau)=e^{-2\pi i \tau(T_\kappa(z,\sigma)-T_\kappa(w,\sigma))}[\hat{\tilde{\mathbb{N}}}^{\alpha,\beta,N}_{K,K'}](z-\sigma,w-\sigma,\tau).\end{equation}
	and
	\begin{equation}\label{eq:fn-subf}[\hat{\mathbb{F}}^{\alpha,\beta,N}_{K,K'}](z,w,\tau)=e^{-2\pi i \tau(T_\kappa(z,\sigma)-T_\kappa(w,\sigma))}[\hat{\tilde{\mathbb{F}}}^{\alpha,\beta,N}_{K,K'}](z-\sigma,w-\sigma,\tau)\end{equation}
\end{lemma}
\begin{proof}
	Lemma \ref{lem:unifL2bound} implies that $\mathbb{N}_{K,K'}$ is bounded on $L^2(\b\Omega)$, and so Lemma \ref{lem:normalization-schwartz} and standard Schwartz kernel arguments imply (\ref{eq:fn-subn}).
	
	For $\delta>0$ we define the operators $\mathbb{F}^\delta_{K,K'}$ via
	$$[\mathbb{F}^\delta_{K,K'}](\mb{z},\mb{w}) = \int_{\delta}^{+\infty} e^{2\pi i \tau (\Re(z_2)-\Re(\bar{w}_2))}\chi(\tau)[\mathbb{R}_\tau^K \mathbb{S}_\tau (\mathbb{R}_\tau^\ast)^{K'}](z,w,\tau)d\tau.$$
	Then Lemma \ref{lem:unifL2bound} implies that $\mathbb{F}^\delta_{K,K'}$ is bounded on $L^2(\b\Omega)$ for $\delta>0$ and therefore
	$$[\hat{\mathbb{F}}^\delta_{K,K'}](z,w,\tau) = \begin{cases} \chi(\tau)[\mathbb{R}_\tau^K \mathbb{S}_\tau (\mathbb{R}_\tau^\ast)^{K'}](z,w,\tau) & \mbox{if}\ \tau\geq \delta, \\ 0 & \mbox{otherwise}.\end{cases}$$
	By arguing as for (\ref{eq:fn-subn}) and then taking $\delta\to 0$, we obtain (\ref{eq:fn-subf}).	
	\end{proof}
	
\section{Kernel Estimates}\label{sec:kernels}

We now come to the heart of the argument.
Throughout this section we work with a UFT domain $\Omega=\{\mb{z}\in\mathbb{C}^2\ :\ \Im(z_2)>P(z)\}$, where $ P=P^{\mb{\sigma},\kappa}$ for some $\mb{\sigma}\in\b\Omega$ and $\kappa\geq 2$.

\subsection{Past Work and Translation}

In this subsection we recall the results of Christ \cite{Christ1991a} on the weighted $\bar{\partial}$ equation on $\mathbb{C}$.
As briefly mentioned in Section \ref{sec:fn}, for subharmonic $P:\mathbb{C}\rightarrow \mathbb{R}$ such that $4\pi \tau\Delta P dm$ is a doubling measure, we define the (closed, densely defined) operators $\bar{D}_\tau=\bar{\partial}+2\pi \tau P_{\bar{z}}$ and $D_\tau=-\partial + 2\pi \tau P_{z}$
on $L^2(\mathbb{C})$.

When $\tau=1$, \cite{Christ1991a} Christ carefully studies the operators
\begin{align*}
	\mathbb{G}=(\bar{D}_\tau D_\tau)^{-1},\quad & \quad  \mathbb{R}_\tau=D_\tau\circ(\bar{D}_\tau D_\tau)^{-1},\\  \mathbb{R}_\tau^\ast =(\bar{D}_\tau D_\tau)^{-1}\circ \bar{D}_\tau, \quad & \quad  \mathbb{S}=I-D_\tau\circ(\bar{D}_\tau D_\tau)^{-1}\bar{D}_\tau
\end{align*}
in terms of a metric $\rho_\tau(z,w)$ on $\mathbb{C}$ given essentially as $d\rho_\tau^2=\sigma_\tau(z)^{-2}ds^2$, where $ds^2$ is the standard Euclidean metric on $\mathbb{C}$ and $\sigma_\tau(z)$ is a smooth function satisfying
$$\int_{|\eta-z|<\sigma_\tau(z)} 4\pi \tau \Delta P(\eta)dm(\eta)\approx 1,$$
uniformly in $z\in\mathbb{C}$. Because the constants here can be taken to universal, we can extend $\sigma_\tau(z)$ to $\tau\in (0,+\infty)$.

The metric $\rho_\tau(z,w)$ and function $\sigma_\tau(z)$ satisfy the following estimates.
\begin{lemma}[\cite{Christ1991a}]\label{lem:chr-rec-metest} If $z,w\in\mathbb{C}$ satisfy $|z-w|\geq \sigma_\tau(w)$, then 
	\begin{equation}\label{eq:chr-rec-rholarge}
		\rho_\tau(z,w)\geq C\Big( \frac{|z-w|}{\sigma_\tau(w)}\Big)^\delta,
	\end{equation}
	where $C,\delta>0$ depend only on the doubling constant $2^Q$ of $4\pi \tau\Delta P dm$.
	
	Moreover, one can find $C,M>0$ also depending only on $2^Q$ such that
	\begin{equation}\label{eq:chr-rec-sigmaratios}
		\frac{\sigma_\tau(z)}{\sigma_\tau(w)}+\frac{\sigma_\tau(w)}{\sigma_\tau(z)}\leq C\Big( \frac{|z-w|}{\sigma_\tau(w)}\Big)^{M}.
	\end{equation}
\end{lemma}

The main theorem of that work provides the following estimates on the Schwartz kernels $\mathbb{G}_\tau,\ \mathbb{R}_\tau,$ and $\mathbb{S}_\tau$, respectively.

\begin{theorem}[\cite{Christ1991a}]\label{thm:chr-rec-Christest}
	There exist constants $C,\epsilon>0$, depending only on the doubling constant of $4\pi \tau\Delta P dm$, such that
	\begin{align}
		|[\mathbb{G}_\tau](z,w)| & \leq  C\begin{cases} \log\Big( \frac{2\sigma_\tau(w)}{|z-w|}\Big),\quad & |z-w|\leq \sigma_\tau(w), \\ e^{-\epsilon\rho_\tau(z,w)},\quad & |z-w|\geq \sigma_\tau(w),\end{cases} \label{eq:chr-rec-bbinvest}\\
		|[\mathbb{R}_\tau](z,w)| & \leq  C\begin{cases} |z-w|^{-1},\quad & |z-w|\leq \sigma_\tau(w), \\ \sigma_\tau(w)^{-1}e^{-\epsilon\rho_\tau(z,w)},\quad & |z-w|\geq \sigma_\tau(w),\end{cases} \label{eq:chr-rec-dbinvest}\\
		|[\mathbb{S}_\tau](z,w)| & \leq  C\sigma_\tau(w)^{-2}e^{-\epsilon\rho_\tau(z,w)}. \label{eq:chr-rec-szegoest}
	\end{align}
\end{theorem}

\begin{remark}{\rm Because $[\mathbb{R}_\tau^\ast](z,w)=\overline{[\mathbb{R}_\tau](w,z)}$, $[\mathbb{R}_\tau^\ast]$ satisfies the same estimates as does $[\mathbb{R}_\tau]$.}\end{remark}

\begin{remark}{\rm For some large fixed $M$, at the expense of perhaps a larger $C=C(M)$, we may replace the estimates for $[\mathbb{G}_\tau](z,w)$ by
		\begin{equation}\label{eq:chr-rec-bbaltered}
			|[\mathbb{G}_\tau](z,w)|  \leq  C\begin{cases} \log\Big( \frac{2M\sigma_\tau(w)}{|z-w|}\Big),\quad & |z-w|\leq M\sigma_\tau(w), \\ e^{-\epsilon\rho_\tau(z,w)},\quad & |z-w|\geq M\sigma_\tau(w),\end{cases}
		\end{equation}
		and similarly for $|[\mathbb{R}_\tau](z,w)|$ and $|[\mathbb{R}_\tau^\ast](z,w)|$.
	}\end{remark}
	
	In order to utilize Christ's results in our setting, we first make two crucial observations.
	First, note that the value of $\tau$ does not affect whether or not $4\pi \tau Pdm$ is a doubling measure, nor its doubling constant. In particular,
	\begin{proposition}{\rm The estimates in (\ref{eq:chr-rec-rholarge}), (\ref{eq:chr-rec-sigmaratios}), and Theorem \ref{thm:chr-rec-Christest} only depend on the constants in (H1)-(H3), and not on $\tau$.}\end{proposition}
	
	
	
	
	Second, we must understand the quantities $\sigma_\tau(z)$ and $\rho_\tau(z,w)$ appearing in Christ's estimates in terms of the geometric quantities studied in Section \ref{sec:geom}. 
	
	\begin{proposition}\label{prop:chr-rec-specforms}Let $\mb{w}\in \b\Omega$. Then for some $C,C',C'',\nu,\delta>0$ which depend only on the constants in (H1) and (H2),
		\begin{itemize}
			\item[(a)] $\sigma_\tau(w)\approx \mu(\mb{w},\tau^{-1})$,
			\item[(b)] $\displaystyle \rho_\tau(z,w)\geq C\Big(\frac{|z-w|}{\sigma_\tau(w)}\Big)^{\delta}\geq C'\Big(\tau\Lambda(\mb{w},|z-w|)\Big)^\nu\geq C''\Big(\frac{|z-w|}{\sigma_\tau(w)}\Big)^{2\nu} $ for $|z-w|\geq \sigma_\tau(w)$.
		\end{itemize}
	\end{proposition}
	\begin{proof}
		For (a), we merely apply Lemma \ref{lem:geom-uftisaq} to the definition of $\sigma_\tau(w)$.
		
		To prove (b), let $\delta$ be as in Lemma \ref{lem:chr-rec-metest} and define $\nu=\frac{\delta}{m}$. Choose $k\geq 0$ so that $|z-w|\approx 2^k\sigma_\tau(w)$. Then Lemma \ref{lem:geom-uftisaq} yields
		\begin{align*}
		\Big(\frac{|z-w|}{\sigma_\tau(w)}\Big)^2 & \approx 2^{2k}\tau\Lambda(\mb{w},\sigma_\tau(w))\lesssim\tau\Lambda(\mb{w},|z-w|)\\
		& \lesssim \tau 2^{mk}\Lambda(\mb{w},\sigma_\tau(w))\approx 2^{mk}\approx \Big(\frac{|z-w|}{\sigma_\tau(w)}\Big)^m.\end{align*}
		Raising each term to the power $\nu$ and applying (\ref{eq:chr-rec-rholarge}) gives the result.
	\end{proof}
	
		\begin{remark}{\rm As before, the above estimates can be taken to be symmetric in $z$ and $w$, and one can interchange the roles of $z$ and $w$ at the expense of perhaps larger constants.}\end{remark}
	
	To simplify the computations in the rest of the paper, it will be convenient to replace $\rho_\tau(z,w)$ with a (simpler) quasimetric. To this end, define
	$$\t\rho_\tau (z,w)=(\tau \Lambda(\mb{w},|z-w|)+\tau \Lambda(\mb{z},|z-w|))^\nu,$$
	where $\mb{z}=(z,z_2)\in \b\Omega$ and $\nu$ is as in Proposition \ref{prop:chr-rec-specforms}. Then Proposition \ref{prop:chr-rec-specforms} and the results in Section \ref{sec:geom} immediately imply that
	\begin{proposition}\label{prop:kernels-rhotilde} Uniformly in $z,w,\eta\in\mathbb{C}$ and $\tau\in (0,+\infty)$,
		\begin{itemize}
			\item[(i)] $\t\rho_\tau(z,w)\approx (\tau \Lambda(\mb{w},|z-w|))^\nu\approx (\tau \Lambda(\mb{z},|z-w|))^\nu$,
			\item[(ii)] $\t\rho_\tau(z,w)\lesssim \rho_\tau(z,w)$ when $|z-w|\geq \sigma_\tau(w)$,
			\item[(iii)] $\t\rho_\tau(z,w)= \t\rho_\tau(w,z)$,
			\item[(iv)] $\t\rho_\tau(z,w)\lesssim \t\rho_\tau(z,\eta)+\t\rho_\tau(\eta,w)$,
			\item[(v)] If $|z-w|\approx |z-\eta|$, then $\t\rho_\tau(z,w)\approx \t\rho_\tau(z,\eta)$.
			\item[(iv)] $\displaystyle\frac{|z-w|}{\sigma_\tau(w)}\lesssim \t\rho_\tau(z,w)^N$ when $|z-w|\geq \sigma_\tau(w)$, for some $N>0$.
		\end{itemize}
	\end{proposition}

	\begin{remark}\label{rem:kernels-estrestate}{\rm The formulas from Theorem \ref{thm:chr-rec-Christest} can be recast in the following slightly weaker form.  Defining $$K_{\tau,0}(z,w)=1,\quad K_{\tau,1}(z,w)=\Big(1+\frac{\sigma_\tau(w)}{|z-w|}\Big),\quad K_{\tau,2}(z,w)=\Big(1+\log\Big(\frac{2\sigma_\tau(w)}{|z-w|}\Big)\Big),$$
			and taking $\t\rho_\tau(z,w)$ as above, there is a constant $C>0$ so that
			$$|[\mathbb{G}_\tau](z,w)|\leq C\sigma_\tau(w)^{-2+2}K_{\tau,2}(z,w)e^{-\epsilon \t\rho_\tau(z,w)},$$
			$$|[\mathbb{R}_\tau](z,w)|+|[\mathbb{R}^\ast_\tau](z,w)|\leq C \sigma_\tau(w)^{-2+1}K_{\tau,1}(z,w)e^{-\epsilon \t\rho_\tau(z,w)},$$
			and
			$$|[\mathbb{S}_\tau](z,w)|\leq C\sigma_\tau(w)^{-2+0}K_{\tau,0}(z,w)e^{-\epsilon \t\rho_\tau(z,w)}.$$ These estimates are symmetric in $z$ and $w$, perhaps at the cost of slightly enlarging $C$.}\end{remark}

	In light of equations (\ref{eq:fn-fkkkernel}) and (\ref{eq:fn-nkkkernel}), in order to prove Theorem \ref{thm:szego-decomp} it is sufficient to obtain precise size estimates for $[W_\tau^\alpha \mathbb{R}^K_\tau \mathbb{S}_\tau (\mathbb{R}_\tau^\ast)^{K'}W_\tau^\beta](z,w).$
	
	\begin{remark}{\rm To take advantage of the oscillatory term in the integrals (\ref{eq:fn-fkkkernel}) and (\ref{eq:fn-nkkkernel}), we will also need to make sense of, and prove size estimates for, the Schwartz kernels of operators of the form 
	\begin{equation}\label{eq:ops}
		\partial_\tau^M (W_\tau^\alpha \mathbb{R}_\tau^K \mathbb{S}_\tau (\mathbb{R}_\tau^\ast)^{K'}W_\tau^\beta).
	\end{equation}
	Note that by the arguments used to prove Lemma \ref{lem:normalization-schwartz} the operators $\partial_\tau$ associated to $\Omega$ and $\tilde{\partial}_{\tau}$ associated to $\tilde\Omega$ are related by $\partial_\tau=e^{-2\pi i \tau T_\kappa(z,\sigma)}\circ \tilde{\partial}_\tau \circ e^{2\pi i \tau T_\kappa(z,\sigma)} = \tilde\partial_\tau+2\pi i T_\kappa(z,\sigma)$, and therefore it is natural (see \cite{Raich2006b}) to replace $\partial_\tau$ with the `twisted' derivative $$e^{2\pi i \tau T_\kappa(z,\sigma)}\circ \partial_\tau \circ e^{-2\pi i \tau T_\kappa(z,\sigma)}=\partial_\tau -2\pi i T_\kappa(z,\sigma)$$ when studying the kernel $[\mathbb{F}^{\alpha,\beta}_{K,K'}](z,w,\tau)$ described in Lemma \ref{lem:fn-substitution}. 
	The substitution described in Lemma \ref{lem:fn-substitution} has the effect of `un-twisting' the operator $$W_\tau^\alpha \mathbb{R}_\tau^K \mathbb{S}_\tau (\mathbb{R}_\tau^\ast)^{K'}W_\tau^\beta,$$ and therefore the $\tau$-derivatives considered in (\ref{eq:ops}), which will be computed after applying Lemma \ref{lem:fn-substitution}, are not twisted in this sense.
}\end{remark}

	The size estimates that we obtain for the operators (\ref{eq:ops}) are most easily expressed as follows:
	For $k\in\mathbb{Z}$ and $w_0\in \mathbb{C}$, and an operator $\mathbb{H}_\tau$ on $\mathbb{C}$, we say that $\mathbb{H}_\tau= {\rm Op}^{w_0}_\tau(k)$ if, for $\phi$ supported in $\{ w\in \mathbb{C}\ :\ |w-w_0|<\sigma_\tau(w_0)\}$, $\mathbb{H}_\tau[\phi](z)$ is given by integrating $\phi$ against a locally integrable Schwartz kernel $[\mathbb{H}_\tau]$ which satisfies
	\begin{equation}
		|[\mathbb{H}_\tau](z,w)|  \lesssim \epsilon^{-1}\sigma_\tau(w)^{k-2} e^{-\epsilon \t\rho_\tau(z,w)},\quad z\in \mathbb{C},\ |w-w_0|\leq \sigma_\tau(w_0)
	\end{equation}
	for some $\epsilon>0$.
	
	If $I\subset (0,{+\infty})$ and $\mathbb{H}_I:=\lb\mathbb{H}_\tau\rb _{\tau\in I}$ is a one-parameter family of operators on $\mathbb{C}$, then say that $\mathbb{H}_I\in {\rm Op}^{w_0}_{I}(k)$ if $\mathbb{H}_\tau\in {\rm Op}^{w_0}_\tau(k)$ uniformly for $\tau\in I$.
	
	\begin{remark}{\rm 
			As is customary, we will often use the notation ${\rm Op}^{w_0}_I(k)$ to refer to an arbitrary sum of operators in ${\rm Op}^{w_0}_I(k)$.
		}\end{remark}
		
		\begin{remark}{\rm The operators ${\rm Op}^{w_0}_I$ are similar in spirit to the one-parameter families of Raich \cite{Raich2006b}, which were designed for the situation when $P$ is a subharmonic, non-harmonic polynomial. Our families are, in a sense, an adaptation of his to the non-polynomial setting, although we have no need for the type of cancellation conditions he imposes on his operators of order $\leq 0$ because our operators all have locally integrable Schwartz kernels.
			}\end{remark}
			
		
		One simple but crucial observation is the relationship between ${\rm Op}_I^{w_0}(k)$ and ${\rm Op}_I^{w_0}(\ell)$ for various intervals $I$.
		
		\begin{proposition}\label{prop:kernels-inclusion} For fixed $0<\Theta<{+\infty}$,
			$${\rm Op}^{w_0}_{(0,\Theta]}(k)\subset {\rm Op}^{w_0}_{(0,\Theta]}(\ell)\quad \mbox{for}\ k\leq \ell,$$
			and
			$${\rm Op}^{w_0}_{[\Theta,{+\infty})}(k)\subset {\rm Op}^{w_0}_{[\Theta,{+\infty})}(\ell)\quad \mbox{for}\ k\geq \ell.$$
		\end{proposition}
		\begin{proof}
			This follows immediately from the observation that $\sigma_\tau(w)\lesssim 1$ for $\tau\geq \Theta$, and $\sigma_\tau(w)\gtrsim 1$ for $\tau\leq \Theta$.
		\end{proof}	
			
			Our main result of this section is as follows.
			
			\begin{theorem}\label{thm:chr-der-szegoderivatives} 
				Let $\alpha$ and $\beta$ be multi-indices, and fix $0<\Theta<{+\infty}$.  If $M,K,K'\geq 0$, then the following hold.
				\begin{itemize}
					\item[(a)] $\tau^M\partial_\tau^M \big(W_\tau^\alpha \mathbb{R}_\tau^K\mathbb{S}_\tau(\mathbb{R}_\tau^\ast)^{K'}W_\tau^\beta\big) \in {\rm Op}_{(0,\Theta]}^0(K+K'-\min(|\alpha|,2)-\min(|\beta|,2))$, if $P=P^{\mb{\sigma},2}$,
					\item[(b)] $\tau^M\partial_\tau^M \big(W_\tau^\alpha \mathbb{R}_\tau^K\mathbb{S}_\tau(\mathbb{R}_\tau^\ast)^{K'}W_\tau^\beta\big)\in {\rm Op}_{[\Theta,{+\infty})}^0(K+K'-|\alpha|-|\beta|)$,\newline if $P=P^{\mb{\sigma},\max(m,|\alpha|,|\beta|)}$.
				\end{itemize}
				In each case, the constants in the estimates depend on $\Theta$, (H1), (H2), (H3), $M,\ K,\ K',\ m,\ |\alpha|$, and $|\beta|$, but are independent of $\tau$.
			\end{theorem}
			
			The proof of Theorem \ref{thm:chr-der-szegoderivatives} is accomplished in several stages. 
			We begin with a definition.
			Say that an operator $\mathbb{H}_\tau$ is $\mathbb{O}_\tau (L,M,N)$ (and write $\mathbb{H}_\tau=\mathbb{O}_\tau(L,M,N)$) if either
			\begin{itemize}
				\item[(a)] $\mathbb{H}_\tau$ is a composition of $A+B+C+D+M$ operators, with $A\geq 1$ and $B+2C-D=L$, where the factors consist of
				\subitem(i) $A$ copies of $\mathbb{S}_\tau$,
				\subitem(ii) $B$ operators from $\{\mathbb{R}_\tau,\mathbb{R}_\tau^\ast\}$,
				\subitem(iii) $C$ copies of $\mathbb{G}_\tau$,
				\subitem(iv) $D$ multiplication operators of the form $\tau \nabla P$
				\subitem(v) $M$ multiplication operators of the form $\tau\nabla^{k_i+2}P$, with $k_i\geq 0$ and $$\displaystyle \sum_{1\leq i\leq M} k_i=N,$$ or
				\item[(b)] $\mathbb{H}_\tau$ is a linear combination of operators described by (a).
			\end{itemize}
			
			In Section \ref{sec:kernels-alg} we rewrite expressions of the form $W_\tau^\alpha \mathbb{R}^K_\tau \mathbb{S}_\tau (\mathbb{R}_\tau^\ast)^{K'} W_\tau^\beta$ as a sum of operators in the classes $\mathbb{O}_\tau(L,M,N)$.
			We then show in Section \ref{sec:kernels-kernelests} that for $\mathbb{H}_\tau=\mathbb{O}_\tau(L,M,N)$ and for $\phi$ supported in $\{|w|<\sigma_\tau(0)\}$, the operator $\phi\mapsto \mathbb{H}_\tau[\phi](z)$ is given by integration against a Schwartz kernel which is in ${\rm Op}^0_{I}(k)$ for appropriate $I$ and $k$, depending on our choice of $\kappa$.
			The arguments in Section \ref{sec:kernels-der} show that the Schwartz kernels of $\mathbb{G}_\tau,\ \mathbb{R}_\tau,\ \mathbb{R}_\tau^\ast,$ and $\mathbb{S}_\tau$  are differentiable in $\tau$, and explicitly compute formulas for their derivatives.
			This gives us a natural definition of $\partial_\tau \mathbb{H}_\tau$ for any $\mathbb{H}_\tau$ in the class $\mathbb{O}_\tau(L,M,N)$; see Corollary \ref{cor:chr-der} for the details. 
			These results are combined in Section \ref{sec:kernels-comp} to prove Theorem \ref{thm:chr-der-szegoderivatives}.

			\subsection{Alternate Expression for $W_\tau^\alpha \mathbb{R}_\tau^K \mathbb{S}_\tau (\mathbb{R}_\tau^\ast)^{K'} W_\tau^\beta$}
			\label{sec:kernels-alg}

			Our goal in this section is to prove the following lemma, which allows us to write operators such as $W_\tau^\alpha \mathbb{R}_\tau^K \mathbb{S}_\tau (\mathbb{R}_\tau^\ast)^{K'} W_\tau^\beta$ in a form which does not involve any explicit differentiation.
			
			\begin{lemma}\label{lem:kernels-alg-structure} Let $K,K',|\alpha|,|\beta|\geq 0$. Then
				\begin{align}
				  W_\tau^\alpha & \mathbb{R}_\tau^K \mathbb{S}_\tau (\mathbb{R}_\tau^\ast)^{K'} W_\tau^\beta\nonumber\\
					 & =  \sum_{\ell} \sum_{\ell'} \sum_{i=0}^{|\alpha|-\ell} \sum_{i'=0}^{|\beta|-\ell'} \mathbb{O}_\tau(K-|\alpha|+2i+\ell,i,\ell)\mathbb{O}_\tau(K'-|\beta|+2i'+\ell',i',\ell')\label{eq:kernels-alg-structure1}\\
					& =  \sum_{\ell} \sum_{\ell'} \sum_{i=0}^{|\alpha|-\ell} \sum_{i'=0}^{|\beta|-\ell'} \mathbb{O}_\tau(K+K'-|\alpha|-|\beta|+2i+2i'+\ell+\ell',i+i',\ell+\ell'),\label{eq:kernels-alg-structure2} \end{align}
				where the outer two sums are over $0\leq \ell\leq \max(|\alpha|-2,0)$ and \newline $0\leq \ell' \leq \max(|\beta|-2,0)$.
			\end{lemma}
			
			\begin{remark}{\rm Note that (\ref{eq:kernels-alg-structure1}) implies that, when writing $W_\tau^\alpha \mathbb{R}_\tau^K \mathbb{S}_\tau (\mathbb{R}_\tau^\ast)^{K'} W_\tau^\beta$ as a linear combination of operators which are $\mathbb{O}_\tau(L,M,N)$, we may assume that each term $\tau\nabla^{2+k_i}P$ satisfies $k_i\leq \max(|\alpha|,|\beta|)$. We will see later how this leads us to use $P^{\mb{\sigma},\max(m,|\alpha|,|\beta|)}$ in part (b) of Theorem \ref{thm:chr-der-szegoderivatives}. }\end{remark}
			
			We begin with a few algebraic computations.
			
			\begin{proposition}\label{prop:kernels-alg-derivs} Let $\mathbb{I}$ denote the identity operator. Then
				\begin{itemize}
					\item[(a)] $D_\tau \mathbb{G}_\tau = \mathbb{R}_\tau$,\quad  $\bar{D}_\tau \mathbb{G}_\tau = -\mathbb{R}_\tau^\ast (4\pi \tau P_{z,\bar{z}})\mathbb{G}_\tau + \mathbb{R}_\tau^\ast$
					\item[(b)] $\mathbb{G}_\tau \bar{D}_\tau = \mathbb{R}_\tau^\ast,$\quad $\mathbb{G}_\tau D_\tau = -\mathbb{G}_\tau (4\pi\tau P_{z,\bar{z}})\mathbb{R}_\tau + \mathbb{R}_\tau.$
					\item[(c)] $\bar{D}_\tau\mathbb{R}_\tau = \mathbb{I}=\mathbb{R}_\tau^\ast D_\tau$,\quad $[D_\tau, \mathbb{R}_\tau] = \mathbb{R}_\tau (4\pi \tau P_{z,\bar{z}})\mathbb{R}_\tau$,
					\item[] \quad $[\bar{D}_\tau,\mathbb{R}_\tau^\ast] = -\mathbb{R}_\tau^\ast (4\pi \tau P_{z,\bar{z}})\mathbb{R}_\tau^\ast$,\quad  $D_\tau \mathbb{R}_\tau^\ast = \mathbb{I}-\mathbb{S}_\tau = \mathbb{R}_\tau \bar{D}_\tau$.
					\item[(d)] $\bar{D}_\tau \mathbb{S}_\tau = 0 = \mathbb{S}_\tau D_\tau$,\quad $D_\tau \mathbb{S}_\tau = \mathbb{R}_\tau (4\pi \tau P_{z,\bar{z}})\mathbb{S}_\tau$,\quad $\mathbb{S}_\tau \bar{D}_\tau = \mathbb{S}_\tau (4\pi \tau P_{z,\bar{z}})\mathbb{R}_\tau^\ast$.
				\end{itemize}
			\end{proposition}
			\begin{proof}
				We will only prove the formula for $\bar{D}_\tau \mathbb{G}_\tau$, as the other formulas either use the same techniques or follow immediately from the formulas in (\ref{eq:dbarops}).
				
				We compute as follows:
				$$\bar{D}_\tau \mathbb{G}_\tau = \underbrace{\mathbb{G}_\tau \bar{D}_\tau}_{=\mathbb{R}_\tau^\ast} D_\tau \bar{D}_\tau \mathbb{G}_\tau = \mathbb{R}_\tau^\ast [D_\tau,\bar{D}_\tau]\mathbb{G}_\tau + \mathbb{R}_\tau^\ast \underbrace{\bar{D}_\tau D_\tau \mathbb{G}_\tau}_{=\mathbb{I}} = - \mathbb{R}_\tau^\ast (4\pi \tau P_{z,\bar{z}})\mathbb{G}_\tau + \mathbb{R}_\tau^\ast.$$
			\end{proof}

			As an immediate application of the above formulas, we have
			\begin{corollary}\label{cor:kernels-alg-upshot} Let $W_\tau$ denote either $D_\tau$ or $\bar{D}_\tau$. Then
				$$W_\tau \mathbb{O}_\tau(L,M,N) = \mathbb{O}_\tau(L-1,M,N) + \mathbb{O}_\tau(L+1,M+1,N)+ \mathbb{O}_\tau(L,M,N+1)$$
				and
				$$\mathbb{O}_\tau (L,M,N)W_\tau = \mathbb{O}_\tau(L-1,M,N) + \mathbb{O}_\tau(L+1,M+1,N)+ \mathbb{O}_\tau(L,M,N+1).$$
			\end{corollary}
			
			We now prove Lemma \ref{lem:kernels-alg-structure}.
			
			\begin{proof}[Proof of Lemma \ref{lem:kernels-alg-structure}.]
				By writing $$W_\tau^\alpha \mathbb{R}_\tau^K \mathbb{S}_\tau (\mathbb{R}_\tau^\ast)^{K'} W_\tau^\beta = (W_\tau^\alpha \mathbb{R}_\tau^K\mathbb{S}_\tau) \mathbb{S}_\tau \big((W_\tau^\beta)^{\ast}\mathbb{R}_\tau^{K'}\mathbb{S}_\tau\big)^\ast$$ and observing that
				$$\mathbb{O}_\tau(L,M,N)\mathbb{O}_\tau(L',M',N')=\mathbb{O}_\tau(L+L',M+M',N+N'),$$
				we are done when we prove (\ref{eq:kernels-alg-structure1}) for $|\beta|=K'=0$.
				
				We prove the cases where $|\alpha|\leq 2$ directly, and then handle the higher-order cases by induction.
				When $|\alpha|=1$, Proposition \ref{prop:kernels-alg-derivs} implies that
				\begin{equation*}
					\bar{D}_\tau \mathbb{R}_\tau^K\mathbb{S}_\tau  =  \begin{cases} \mathbb{R}_\tau^{K-1}\mathbb{S}_\tau &\mbox{if}\ K\geq 1,\\ 0 &\mbox{if}\ K=0,\end{cases}\quad \mbox{and}\quad D_\tau \mathbb{R}_\tau^K \mathbb{S}_\tau  =  \sum_{j=0}^K \mathbb{R}_\tau^{j+1}(4\pi \tau P_{z,\bar{z}})\mathbb{R}_\tau^{K-j}\mathbb{S}_\tau.
				\end{equation*}
				
				There are four possibilities when $|\alpha|=2$:
				\begin{align*}
					\bar{D}_\tau\bar{D}_\tau \mathbb{R}_\tau^K\mathbb{S}_\tau  &=  \begin{cases} \mathbb{R}_\tau^{K-2}\mathbb{S}_\tau\quad &\mbox{if}\ K\geq 2,\\ 0\quad &\mbox{if}\ K\leq 1,\end{cases}\\
					D_\tau \bar{D}_\tau \mathbb{R}_\tau^K \mathbb{S}_\tau  &=  \begin{cases} 0\quad &\mbox{if}\ K=0,\\ \displaystyle \sum_{j=0}^{K-1} \mathbb{R}_\tau^{j+1}(4\pi \tau P_{z,\bar{z}})\mathbb{R}_\tau^{K-1-j}\mathbb{S}_\tau\quad &\mbox{if}\ K\geq 1,\end{cases}\\
					\bar{D}_\tau D_\tau \mathbb{R}_\tau^K \mathbb{S}_\tau  &=  \sum_{j=0}^K \mathbb{R}_\tau^j (4\pi \tau P_{z,\bar{z}})\mathbb{R}_\tau^{K-j}\mathbb{S}_\tau,\\
					D_\tau  D_\tau \mathbb{R}_\tau^K\mathbb{S}_\tau & =  \sum_{j=0}^K \sum_{\ell=0}^j \mathbb{R}_\tau^{\ell+1}(4\pi \tau P_{z,\bar{z}})\mathbb{R}_\tau^{j+1-\ell}(4\pi \tau P_{z,\bar{z}})\mathbb{R}_\tau^{K-j}\mathbb{S}_\tau \\
					& \qquad \qquad + \sum_{j=0}^K \mathbb{R}_\tau^j \mathbb{S}_\tau(4\pi \tau P_{z,z})\mathbb{R}_\tau^{K-j}\mathbb{S}_\tau - \sum_{j=0}^K \mathbb{R}_\tau^j (4\pi \tau P_{z,z})\mathbb{R}_\tau^{K-j}\mathbb{S}_\tau \\
					& \qquad\qquad + \sum_{j=0}^{K-1} \mathbb{R}_\tau^{j+1} (4\pi \tau P_{z,z})\mathbb{R}_\tau^{K-1-j}\mathbb{S}_\tau\\
					&  \qquad\qquad  + \sum_{j=0}^K \sum_{\ell=0}^{K-j} \mathbb{R}_\tau^{j+1}(4\pi \tau P_{z,\bar{z}})\mathbb{R}_\tau^{\ell+1}(4\pi \tau P_{z,\bar{z}})\mathbb{R}_\tau^{K-j-\ell}\mathbb{S}_\tau,
				\end{align*}
				where in the case $D_\tau D_\tau \mathbb{R}_\tau^K \mathbb{S}_\tau$ we needed to use, in addition to Proposition \ref{prop:kernels-alg-derivs}, the computation $$[D_\tau,4\pi \tau P_{z,\bar{z}}] = -4\pi \tau P_{z,\bar{z},z} = -[\bar{D}_\tau,4\pi \tau P_{z,z}].$$
				Because $\mathbb{R}_\tau^K\mathbb{S}_\tau=\mathbb{O}_\tau(K,0,0)$, we see that for $|\alpha|\leq 2$, $$\displaystyle W_\tau^\alpha \mathbb{R}_\tau^K\mathbb{S}_\tau = \sum_{i=0}^{|\alpha|} \mathbb{O}_\tau(K-|\alpha|+2i,i,0)$$ as desired.
				
				We turn now to the proof of equation (\ref{eq:kernels-alg-structure1}), which we have shown holds for $|\alpha|=0,1,2$. If we know that (\ref{eq:kernels-alg-structure1}) holds for some $\alpha$ with $|\alpha|\geq 2$, then by Corollary \ref{cor:kernels-alg-upshot} we have
				\begin{align*}
					W_\tau W^\alpha_\tau \mathbb{R}_\tau^{K} \mathbb{S}_\tau & =  W_\tau \sum_{\ell=0}^{|\alpha|-2} \sum_{i=0}^{|\alpha|-\ell} \mathbb{O}_\tau (K-|\alpha|+2i+\ell,i,\ell) \\
					& =  \sum_{\ell=0}^{|\alpha|-2}\sum_{i=0}^{|\alpha|-\ell} \Big[ \mathbb{O}_\tau(K-(|\alpha|+1)+2i+\ell,i,\ell)\\
					& \qquad\qquad\qquad + \mathbb{O}_\tau (K-(|\alpha|+1)+2i+(\ell+1),i,\ell+1) \\
					&  \qquad\qquad\qquad + \mathbb{O}_\tau(K-(|\alpha|+1)+2(i+1)+\ell,i+1,\ell)\Big] \\
					& =  \sum_{\ell=0}^{(|\alpha|+1)-2}\sum_{i=0}^{(|\alpha|+1)-\ell} \mathbb{O}_\tau (K-(|\alpha|+1)+2i+\ell,i,\ell).
				\end{align*}
				This completes the proof of Lemma \ref{lem:kernels-alg-structure}.
			\end{proof}

			\begin{remark}{\rm The inclusion of the multiplication operators $\tau\nabla P$ in the definition of $\mathbb{O}_\tau(L,M,N)$ might currently seem superfluous because these terms have not yet appeared in the above proof (indeed, only higher-order derivatives of $P$ appeared).
					These operators will show up when we differentiate operators in $\mathbb{O}_\tau(L,M,N)$ with respect to $\tau$.
				}\end{remark}
			
			\begin{remark}\label{rem:kernels-heisenberg}{\rm If we are working with the Heisenberg group (i.e. if $P(z)=|z|^2$), then the above results simplify substantially. This is due to the fact that $P_{z,z}\equiv 0$ and $P_{z,\bar{z}}\equiv 1$, and therefore 
			$$W_\tau^\alpha \mathbb{R}_\tau^K \mathbb{S}_\tau (\mathbb{R}_\tau^\ast)^{K'} W_\tau^\beta=\sum_{j=0}^{|\alpha|+|\beta|} \mathbb{O}_\tau (K+K'-|\alpha|-|\beta|+2j,j,0).$$
			The arguments used to prove Theorem \ref{thm:chr-der-szegoderivatives} in Section \ref{sec:kernels-comp} then yield $$\tau^M \partial_\tau^M (W_\tau^\alpha \mathbb{R}_\tau^K \mathbb{S}_\tau (\mathbb{R}_\tau^\ast)^{K'} W_\tau^\beta)\in {\rm Op}_{(0,{+\infty})}^0 (K+K'-|\alpha|-|\beta|),$$ which is a much stronger result for $\tau\lesssim 1$ than that of Theorem \ref{thm:chr-der-szegoderivatives} for general $P$.  }\end{remark}
				
				\subsection{Estimates for $\mathbb{O}_\tau (L,M,N)$}
				\label{sec:kernels-kernelests}
				
				Our next goal is the following lemma.
				
				\begin{lemma}\label{lem:kernels-kernelests} Let $\mathbb{H}_\tau\in \mathbb{O}_\tau(L,M,N)$, and restrict $\mathbb{H}_\tau$ to test functions supported in $\{|w|<\sigma_\tau(0)\}$. As before, fix $0<\Theta<{+\infty}$.
					
					If $P=P^{\mb{\sigma},2}$, then
					\begin{equation}\label{eq:kernels-kernelest2}\mathbb{H}_{(0,\Theta]} \in {\rm Op}^0_{(0,\Theta]}(L-2M).\end{equation}
					If $P=P^{\mb{\sigma},\kappa}$ for $\kappa\geq \max(m,|\alpha|,|\beta|)$, then
					\begin{equation}\label{eq:kernels-kernelestkappa}\mathbb{H}_{[\Theta,{+\infty})} \in {\rm Op}^0_{[\Theta,{+\infty})}(L-2M-N).\end{equation}
				\end{lemma}
				
				As a preliminary step, we give additional pointwise bounds for the derivatives of $P^{\mb{\sigma},2}$ and $P^{\mb{\sigma},\kappa}$.
				
				\begin{lemma}\label{lem:kernels-potderivs}
					Let $|w|\leq \sigma_\tau(0)$ and $\eta\in \mathbb{C}$, and fix $\epsilon>\epsilon'>0$, $\kappa\geq m$, and $1\leq k\leq \kappa$. Then for $\tau\leq \Theta$ we have
					\begin{equation}\label{eq:kernels-potderivs2} |\tau\nabla^k P^{\mb{\sigma},2}(\eta)|e^{-\epsilon\t\rho_\tau(\eta,w)}\lesssim  \sigma_\tau(w)^{-\min(k,2)}e^{-\epsilon'\t\rho_\tau(\eta,w)},\end{equation}
					while for $\tau\geq \Theta$ we have
					\begin{equation}\label{eq:kernels-potderivsk} |\tau\nabla^k P^{\mb{\sigma},\kappa}(\eta)|e^{-\epsilon\t\rho_\tau(\eta,w)}\lesssim \sigma_\tau(w)^{-k}e^{-\epsilon'\t\rho_\tau(\eta,w)},\end{equation}
					where the constants involved depend only on $\epsilon$, $\epsilon'$, $\Theta$, and (H1)-(H3).
				\end{lemma}
				\begin{proof}
					
					When $\tau\leq \Theta$, note that $\sigma_\tau(\eta)^2\tau \approx 1$ uniformly in $\eta$. Thus if $|\eta|\leq \sigma_\tau(\eta)\approx \sigma_\tau(w)$, then Proposition \ref{prop:potentials} gives
					$$|\tau \nabla^k P^{\mb{\sigma},2}(\eta)|\lesssim \tau |\eta|^{2-\min(k,2)} \lesssim \tau \sigma_\tau(\eta)^{2-\min(k,2)}\approx \sigma_\tau(w)^{-\min(k,2)}.$$
					On the other hand, if $|\eta|\geq \sigma_\tau(\eta)$ then we apply part (b) of Proposition \ref{prop:chr-rec-specforms} to obtain
					\begin{align*}
					 |\tau\nabla^k P^{\mb{\sigma},2}(\eta)|&\lesssim \tau |\eta|^{2-\min(k,2)}\lesssim \tau\sigma_\tau(\eta)^{2-\min(k,2)}\t\rho_\tau(\eta,0)^N \\
					 &\lesssim \sigma_\tau(w)^{-\min(k,2)}(\t\rho_\tau(\eta,w)+\t\rho_\tau(w,0))^N \\
					 & \lesssim \sigma_\tau(w)^{-\min(k,2)}(\t\rho_\tau(\eta,w)+1)^N,\end{align*}
					from which (\ref{eq:kernels-potderivs2}) follows.
					
					When $\tau\geq \Theta$, we write
					\begin{equation}\label{eq:kernels-potder}
						P^{\mb{\sigma},\kappa}(\eta)=\sum_{\on{2\leq \alpha+\beta\leq \kappa,}{\alpha,\beta\geq 1}} \frac{\partial^{\alpha+\beta-2}h}{\partial z^\alpha \partial \bar{z}^\beta}(0) \eta^\alpha \bar{\eta}^\beta + \mathcal{O}(\|h\|_{C^{\kappa-1}}|\eta|^{\kappa+1}). \end{equation}
					Computing $\nabla^k P^{\mb{\sigma},\kappa}(\eta)$ and estimating yields
					$$|\nabla^k P^{\mb{\sigma},\kappa}(\eta)|\lesssim \sum_{\on{\max(2,k)\leq\alpha+\beta\leq \kappa}{\alpha,\beta\geq 1}} \Big| \frac{\partial^{\alpha+\beta-2}h}{\partial z^\alpha \partial \bar{z}^\beta}(0)\Big| |\eta|^{\alpha+\beta-k}+\|h\|_{C^{\kappa-1}}|\eta|^{\kappa-k+1}.$$
					
					If $|\eta|\lesssim \sigma_\tau(0)$, then because $\sigma_\tau(w)\approx\sigma_\tau(0)\lesssim 1$ we apply Remark \ref{rem:metric-approxhighorder} to get
					$$|\tau\nabla^k P^{\mb{\sigma},\kappa}(\eta)|\lesssim \tau\sigma_\tau(0)^{-k}\Lambda(\mb{0},\sigma_\tau(0))=\sigma_\tau(0)^{-k}\approx \sigma_\tau(w)^{-k}.$$
					Similarly, for $|\eta|\gg \sigma_\tau(0)$ we apply part (b) of Proposition \ref{prop:chr-rec-specforms} to see that
					$$|\tau\nabla^k P^{\mb{\sigma},\kappa}(\eta)|\lesssim \tau\sigma_\tau(0)^{-k}\Big(\frac{|\eta|}{\sigma_\tau(0)}\Big)^{\kappa+1-k} \lesssim \sigma_\tau (w)^{-k}\t\rho_\tau(\eta,w)^N,$$
					which yields (\ref{eq:kernels-potderivsk}).
				\end{proof}

				Our first lemma shows how operators behave under composition.
				
				\begin{lemma}\label{lem:kernels-basicest} For $\epsilon>0$ there exists $\epsilon'>0$ such that for $0\leq i,j\leq 2$,
					\begin{align*}
					\int_\mathbb{C} e^{-\epsilon(\t\rho_\tau(z,\eta)+\t\rho_\tau(\eta,w))}& K_{\tau,i}(z,\eta)K_{\tau,j}(\eta,w)dm(\eta)\\
					&\lesssim \begin{cases} \sigma_\tau(w)^{2}K_{\tau,2}(z,w)e^{-\epsilon'\t\rho_\tau(z,w)}\ & \mbox{if}\ i=j=1,\\ \sigma_\tau(w)^{2}e^{-\epsilon'\t\rho_\tau(z,w)} &\mbox{otherwise}.\end{cases}\end{align*}
				\end{lemma}
				\begin{proof}
					We will prove the case $i=j=1$, as it is almost identical to the other cases and exhibits all of the relevant techniques. Throughout the proof, $\epsilon$ denotes an arbitrary small positive number that may shrink from line to line.
					
					We want to estimate the integral $$\displaystyle \mathcal{I}:=\int_\mathbb{C} e^{-\epsilon(\t\rho_\tau(z,\eta)+\t\rho_\tau(\eta,w))}|K_{\tau,1}(z,\eta)||K_{\tau,1}(\eta,w)|dm(\eta).$$
					To do this, we consider two cases:
					\medskip
					
					\noindent\textbf{Case 1:} $|z-w|\leq 2{\rm max}(\sigma_\tau(z),\sigma_\tau(w))$,
					\smallskip
					
					\noindent\textbf{Case 2:} $|z-w|\geq 2{\rm max}(\sigma_\tau(z),\sigma_\tau(w))$.
					\medskip
					
					For Case 1, break the integral $\mathcal{I}$ into 
					$$\mathcal{I}=\int_{|z-\eta|\leq M\sigma_\tau(w)}+\int_{|z-\eta|\geq M\sigma_\tau(w)}=:I_1+I_2.$$
					Here, $M$ is chosen large but depends only on the doubling constant of $h(\eta)dm(\eta)$.
					Throughout this case we will use the fact that $\t\rho_\tau(z,w)\lesssim 1$ and $\sigma_\tau(z)\approx \sigma_\tau(w)$.
					
					For $I_1$, we may assume that $|z-w|>0$. Choosing $M$ so large that $|z-w|\leq \displaystyle \frac{M\sigma_\tau (w)}{20}$ and setting $\displaystyle a=\frac{|z-w|}{M\sigma_\tau(w)}$, we have
					\begin{align*}
						|I_1|&\lesssim  \int_{|z-\eta|\leq M\sigma_\tau(w)} |z-\eta|^{-1}\sigma_\tau(\eta)|\eta-w|^{-1}\sigma_\tau(w) dm(\eta) \\
						&\lesssim  \sigma_\tau(w)^2 \int_{|\hat{\eta}|\leq a^{-1}} |\hat{\eta}|^{-1}|\hat{\eta}+1|^{-1}dm(\eta) \\
						& \leq  \sigma_\tau(w)^2\Bigg(\int_{|\hat{\eta}|\leq 3}  |\hat{\eta}|^{-1}\big|\hat{\eta}+1\big|^{-1} dm(\hat{\eta}) +\int_{3\leq|\hat{\eta}|\leq a^{-1}} |\hat{\eta}|^{-1}\big|\hat{\eta}+1\big|^{-1} dm(\hat{\eta})\Bigg) \\
						& \lesssim  \sigma_\tau (w)^2\Bigg(1 +\int_{3\leq|\hat{\eta}|\leq a^{-1}} |\hat{\eta}|^{-2} dm(\hat{\eta})\Bigg)\\
						& \approx  \sigma_\tau(w)^2 \ln(a^{-1}) \\
						& =  \sigma_\tau (w)^2 \ln\Big( \frac{M\sigma_\tau(w)}{|z-w|}\Big),
					\end{align*}
					where in the second step we made the change of variable $(z-w)\hat{\eta}=\eta-z$.
					
					For $I_2$, we make the change of variable $z-\eta=\sigma_\tau(w)\hat{\eta}$ and note that $|\eta-z|\approx |\eta-w|$ to get
					\begin{align*}
						|I_2| &\lesssim  \int_{|z-\eta|\geq M\sigma_\tau(w)} e^{-\epsilon\t\rho_\tau(z,\eta)}e^{-\epsilon\t\rho_\tau(\eta,w)}dm(\eta) \\
						& \lesssim  \sigma_\tau(w)^2\int_{|\hat\eta|\geq M} e^{-\epsilon(\tau\Lambda(z,|\hat\eta|\sigma_\tau(w)))^\nu} dm(\hat\eta) \\
						& \lesssim  \sigma_\tau(w)^{2}\int_{|\hat\eta|\geq M} e^{-\epsilon|\hat\eta|^\nu}dm(\hat\eta)\\
						&\lesssim  \sigma_\tau(w)^{2}, \\
					\end{align*}
					where we used Proposition \ref{prop:kernels-rhotilde} in the first two lines.
					This completes the proof of Case 1.
					
					For Case 2, we need to break $\mathcal{I}$ into five pieces:
					\begin{align*}
						\mathcal{I} &= \int_{|z-\eta|\leq \sigma_\tau(z)} + \int_{|w-\eta|\leq \sigma_\tau(w)} + \int_{\sigma_\tau(z)\leq |z-\eta|\leq \frac{|z-w|}{2}}\\
						&  \qquad\qquad +\int_{\sigma_\tau(w)\leq |w-\eta|\leq \frac{|z-w|}{2}}+\int_{\min(|z-\eta|,|w-\eta|)\geq \frac{|z-w|}{2}}\\
						&=: I_1+I_2+I_3+I_4+I_5.
					\end{align*}
					 
					For $I_1$ we note that because $|w-\eta|\approx |w-z|$, part (v) of Proposition \ref{prop:kernels-rhotilde} implies that
					\begin{align*}
						|I_1| & \lesssim  \int_{|z-\eta|\leq \sigma_\tau(z)} |z-\eta|^{-1}\sigma_\tau(\eta)e^{-\epsilon\t\rho_\tau(\eta,w)}dm(\eta) \\
						& \lesssim  \sigma_\tau(w)\int_{|z-\eta|\leq \sigma_\tau(z)} |z-\eta|^{-1}e^{-\epsilon\t\rho_\tau(\eta,w)}dm(\eta) \\
						& \lesssim   \sigma_\tau(w)^2 e^{-\epsilon\t\rho_\tau(z,w)}.
					\end{align*}
					In the second line we used (\ref{eq:chr-rec-sigmaratios}) and (\ref{eq:chr-rec-rholarge}).
					$I_2$ is estimated in the same way.
					
					For $I_3$, again applying Proposition \ref{prop:kernels-rhotilde} and Proposition \ref{prop:chr-rec-specforms},
					\begin{align*}
						|I_3| & \lesssim  \int_{\sigma_\tau(z)\leq |z-\eta|\leq \frac{|z-w|}{2}} e^{-\epsilon\t\rho_\tau(z,\eta)}e^{-\epsilon\t\rho_\tau(\eta,w)}dm(\eta) \\
						& \lesssim  e^{-\epsilon\t\rho_\tau(z,w)}\sigma_\tau(z)^2\int_{1\leq |\hat{\eta}|} e^{-\epsilon |\hat{\eta}|^\nu}dm(\hat{\eta}) \\
						& \lesssim  \sigma_\tau(w)^2\t\rho_\tau(z,w)^M e^{-\epsilon\t\rho_\tau(z,w)} \\
						& \lesssim  \sigma_\tau(w)^2e^{-\epsilon\t\rho_\tau(z,w)}.
					\end{align*}
					
					Of course, $I_4$ is almost identical.
					Because $I_5$ involves nothing more than changing variables and using the fact that $|\eta-z|\approx |\eta-w|\geq |w-z|$, the proof of the case $i=j=1$ is complete.
				\end{proof}

				\begin{proof}[Proof of Lemma \ref{lem:kernels-kernelests}.]
					
					It is enough to prove the theorem when $\mathbb{H}_\tau$ falls under part (a) of the definition of $\mathbb{O}_\tau(L,M,N)$.
					
					We prove (\ref{eq:kernels-kernelestkappa}); the proof of (\ref{eq:kernels-kernelest2}) is exactly the same, but with all of the $k_i$ replaced with $\min(k_i,2)$.
					
					Write $\mathbb{H}_\tau =\displaystyle \prod_{i=1}^I \mathbb{H}_{\tau,i}$, where each $\mathbb{H}_{\tau,i}$ is either $\mathbb{S}_\tau,\ \mathbb{R}_\tau,\ \mathbb{R}_\tau^\ast,\ \mathbb{G}_\tau$, or multiplication by $\tau \nabla^{k_i}P$ (with $k_i\geq 1$), and define
					$$d(i)=\begin{cases} -k & \mbox{if}\ \mathbb{H}_{\tau,i}=\tau \nabla^{k}P,\\ 0 & \mbox{if}\ \mathbb{H}_{\tau,i}=\mathbb{S}_\tau,\\
					1 & \mbox{if}\ \mathbb{H}_{\tau,i}=\mathbb{R}_\tau,\mathbb{R}^\ast_\tau,\\
					2 & \mbox{if}\ \mathbb{H}_{\tau,i}=\mathbb{G}_\tau.\end{cases}$$
					Let $1\leq i_1<i_2<\cdots<i_\theta\leq I$ be the indices for which $\mathbb{H}_{\tau,i_j}\in \{\mathbb{S}_\tau,\mathbb{R}_\tau,\mathbb{R}_\tau^\ast,\mathbb{G}_\tau\}$, and set $i_0=0$.
					Define $$\mathcal{I}_1(z_1,w),\ \mathcal{I}_2(z_2,w),\ldots,\ \mathcal{I}_{\theta-1}(z_{\theta-1},w),\ \mathcal{I}_{\theta}(z,w)$$
					by setting
					\begin{align*}
						\mathcal{I}_1(z_1,w) & =  |[\mathbb{H}_{\tau,i_1}](z_1,w)|\prod_{0<i<i_1} |\tau \nabla^{k_i}P(w)|,\\
						\mathcal{I}_2(z_2,w) & =  \int_\mathbb{C} |[\mathbb{H}_{\tau,i_2}](z_2,z_1)|\prod_{i_1<i<i_2} |\tau \nabla^{k_i}P(w)| \mathcal{I}_1(z_1,w)dm(z_1)\\
						 & \vdots  \\
						 \mathcal{I}_{\theta}(z,w) & =  \int_\mathbb{C} |[\mathbb{H}_{\tau,i_\theta}](z,z_{\theta-1})|\prod_{i_{\theta-1}<i<i_\theta} |\tau \nabla^{k_i}P(w)| \mathcal{I}_1(z_{\theta-1},w)dm(z_{\theta-1}).
					\end{align*}
					By the Fubini-Tonelli Theorem,
					\begin{equation}\label{eq:kernels-induction}
						|[\mathbb{H}_\tau](z,w)| \leq \Big(\prod_{i_{\theta}<i\leq I} |\tau \nabla^{k_i}P(w)|\Big) \mathcal{I}_\theta (z,w),\end{equation}
					and therefore our first task is to bound the right-hand-side of (\ref{eq:kernels-induction}) by induction.
					
					By the definition of $\mathbb{O}_\tau (L,M,N)$, There are four mutually exclusive cases to consider:
					\begin{center}
					\begin{tabular}{rl}
						\textbf{Case 1:} & $\mathbb{H}_{\tau,i_1}=\mathbb{S}_\tau,$ \\
						\textbf{Case 2:} & $\mathbb{H}_{\tau,i_2}=\mathbb{S}_\tau$ and $\mathbb{H}_{\tau,i_1}\neq \mathbb{S}_\tau,$\\
						\textbf{Case 3:} & $\theta\geq 3$ and $\mathbb{H}_{\tau,i_1},\mathbb{H}_{\tau,i_2}\in\{\mathbb{R}_\tau,\mathbb{R}^\ast_\tau\},$ \\
						\textbf{Case 4:} & $\theta\geq 3$ and either $\mathbb{H}_{\tau,i_1}=\mathbb{G}_\tau$ or $\mathbb{H}_{\tau,i_2}=\mathbb{G}_\tau$.\end{tabular}
						\end{center}
				The proofs of the various cases are almost identical, differing only in the details of applying Lemma \ref{lem:kernels-basicest} to establish the inductive base step. We provide the details for Case 3.
				Throughout the argument, $\epsilon>0$ is a small number which might shrink from line to line.
				
				Assume that $\Theta\leq \tau<{+\infty}$, $\theta\geq 3$, and that $d(i_1)=d(i_2)=1$. By Remark \ref{rem:kernels-estrestate} and Lemma \ref{lem:kernels-potderivs} we have
				$$\mathcal{I}_1(z_1,w)\lesssim e^{-\epsilon \t\rho_\tau (z_1,w)}\sigma_\tau(w)^{-2+\sum_{0<i\leq i_1} d(i)} K_{\tau,1}(z_1,w),$$
				where $K_{\tau,1}(z_1,w)$ is as in Remark \ref{rem:kernels-estrestate}. Applying Remark \ref{rem:kernels-estrestate}, Lemma \ref{lem:chr-rec-metest}, Proposition \ref{prop:chr-rec-specforms}, Lemma \ref{lem:kernels-potderivs}, and Lemma \ref{lem:kernels-basicest} gives
				\begin{align*}
					& \mathcal{I}_2(z_2,w)\\
					& \lesssim\ \int_{\mathbb{C}} \Big\{ e^{-\epsilon(\t\rho_\tau (z_2,z_1)+\t\rho_\tau(z_1,w))}K_{\tau,1}(z_2,z_1)K_{\tau,1}(z_1,w)\\
					& \quad\qquad\qquad\times \sigma_\tau(z_1)^{-2+\sum_{i_1<i\leq i_2} d(i)}\sigma_\tau(w)^{-2+\sum_{0<i\leq i_1} d(i)}\Big\}dm(z_1) \\
					& \lesssim\ \sigma_\tau(w)^{-4+\sum_{0<i\leq i_2} d(i)}\int_{\mathbb{C}} e^{-\epsilon(\t\rho_\tau (z_2,z_1)+\t\rho_\tau(z_1,w))}K_{\tau,1}(z_2,z_1)K_{\tau,1}(z_1,w)dm(z_1)\\
					& \lesssim\ \sigma_\tau(w)^{-2+\sum_{0<i\leq i_2} d(i)}K_{\tau,2}(z_2,w)e^{-\epsilon \t\rho_\tau (z_2,w)}.
				\end{align*}
					Repeating this argument for $\mathcal{I}_3(z_3,w)$ yields
					$$\mathcal{I}_3(z_3,w)\lesssim \sigma_\tau(w)^{-2+\sum_{0<i\leq i_3} d(i)}K_{\tau,0}(z_3,w)e^{-\epsilon \t\rho_\tau (z_3,w)}.$$
					We now apply the same argument inductively to see that
					$$\mathcal{I}_\theta(z,w)\lesssim \sigma_\tau(w)^{-2+\sum_{0<i\leq i_\theta} d(i)}K_{\tau,0}(z,w)e^{-\epsilon \t\rho_\tau (z,w)},$$
					so that by Lemma \ref{lem:kernels-potderivs} we have
					$$|[\mathbb{H}_\tau](z,w)|\lesssim \sigma_\tau(w)^{-2+\sum_{0<i\leq I} d(i)} e^{-\epsilon\t\rho_\tau(z,w)}.$$
					In other words, as long as $\kappa\geq \max(m,|\alpha|,|\beta|)$ we have
					\begin{align*}
						|[\mathbb{H}_\tau](z,w)|& \lesssim   \sigma_\tau(w)^{-2+B+2C-D-2M-N}e^{-\epsilon\t\rho_\tau(z,w)}\\
						& = \sigma_\tau(w)^{-2+L-2M-N}e^{-\epsilon\t\rho_\tau(z,w)}\quad\mbox{for}\ \tau\geq\Theta.\end{align*}
					This shows that $\mathbb{H}_{[\Theta,{+\infty})}\in {\rm Op}^0_{[\Theta,{+\infty})}(L-2M-N).$
					
					In the case where we work with $P^{\mb{\sigma},2}$ (and where $0<\tau\leq \Theta$), our application of Lemma  \ref{lem:kernels-potderivs} necessitates that when $d(i)=-k<0$ we replace $-k$ with $-\min(2,k)$, yielding
					$\mathbb{H}_{(0,\Theta]}\in {\rm Op}^0_{(0,\Theta]} (L-2M).$
					
					This completes the proof of Lemma \ref{lem:kernels-kernelests}.

					\end{proof}

				\subsection{Derivatives in $\tau$}
				\label{sec:kernels-der}
				
				For a one-parameter family of operators $\mathbb{F}_{(0,+\infty)}$ and fixed $\tau\in (0,+\infty)$, define
				$$\Delta_h (\mathbb{F}_\tau) = h^{-1}(\mathbb{F}_{\tau+h}-\mathbb{F}_{\tau})$$
				for all $0<|h|< \tau $.
				
				\begin{lemma}\label{lem:chr-der-derivatives} Let $P_{z}$ and $P_{\bar{z}}$ denote multiplication by $P_{z}$ and $P_{\bar{z}}$, respectively. As operators, we have the following formulas:
					\begin{itemize}
						\item[(a)] $\Delta_h(\mathbb{G}_\tau) = -\mathbb{G}_\tau (2\pi P_{\bar{z}}) \mathbb{R}_{\tau+h}-\mathbb{R}^\ast_\tau (2\pi P_{z}) \mathbb{G}_{\tau+h}$,
						\item[(b)] $\Delta_h(\mathbb{R}_\tau) = \mathbb{S}_\tau (2\pi P_{z}) \mathbb{G}_{\tau+h} -\mathbb{R}_\tau (2\pi P_{\bar{z}})\mathbb{R}_{\tau+h}$,
						\item[(c)] $\Delta_h(\mathbb{R}^\ast_\tau) = \mathbb{G}_\tau (2\pi P_{\bar{z}})\mathbb{S}_{\tau+h} -\mathbb{R}^\ast_\tau (2\pi P_{z})\mathbb{R}^\ast_{\tau+h}$,
						\item[(d)] $\Delta_h(\mathbb{S}_\tau) = -\mathbb{S}_\tau (2\pi P_{z})\mathbb{R}^\ast_{\tau+h} - \mathbb{R}_\tau (2\pi P_{\bar{z}})\mathbb{S}_{\tau+h}$.
					\end{itemize}
				\end{lemma}
				\begin{proof}
					We first compute (a).
					Choose a smooth cutoff function $\chi(t)$ with $\chi(t)=1$ for $t\leq 1$ and $\chi(t)=0$ for $t\geq 2$.
					Also, let $\eta\in C_c^\infty(\mathbb{C}^2)$ satisfy $\int \eta = 1$, and write $\eta_t(z)=t^{-4}\eta(t^{-1}z,t^{-1}w)$.
					
					For $\epsilon>0$ define the regularized kernel
					$$[\mathbb{G}_\tau]^\epsilon(z,w,\tau):= \eta_\epsilon\ast[\chi(\epsilon|\bullet_1-\bullet_2|)[\mathbb{G}_\tau](\bullet_1,\bullet_2)](z,w),$$
					and let $\mathbb{G}_\tau^\epsilon$ be the operator given by integration against $[\mathbb{G}_\tau]^\epsilon$.
					We also define
					$$\mathbb{R}_\tau^\epsilon:= D_\tau\mathbb{G}_\tau^\epsilon,\quad \mathbb{R}_\tau^{\ast,\epsilon}:=\mathbb{G}_\tau^\epsilon \bar{D}_\tau,\quad \mathbb{S}_\tau^\epsilon:= \mathbb{I}-D_\tau \mathbb{R}_\tau^{\ast,\epsilon}.$$
					One can easily show that, as $\epsilon\rightarrow 0$, these distributions converge to their respective non-regularized operators.
					
					Notice that
					\begin{align*}
						\Delta_h & (\mathbb{G}_\tau^\epsilon \bar{D}_\tau D_\tau)\mathbb{G}_{\tau+h}^\delta \\
						& =  \Delta_h(\mathbb{G}_\tau^\epsilon)\bar{D}_{\tau+h} D_{\tau+h}\mathbb{G}_{\tau+h}^\delta + \mathbb{G}_\tau^\epsilon (2\pi P_{\bar{z}}) D_{\tau+h}\mathbb{G}_{\tau+h}^\delta + \mathbb{G}_\tau^\epsilon \bar{D}_\tau (2\pi P_{z})\mathbb{G}_{\tau+h}^\delta.
					\end{align*}
					Sending first $\epsilon\rightarrow 0$, and then $\delta\rightarrow 0$, we obtain (a).
					
					Now that (a) is established, we can use it to prove the other formulas.
					To this end, note that
					\begin{align*}
						\Delta_h(\mathbb{R}_\tau)& = (2\pi P_z)\mathbb{G}_{\tau + h} + D_\tau \Delta_h (\mathbb{G}_\tau) \\
						& = (2\pi P_z)\mathbb{G}_{\tau+h} -D_\tau\mathbb{G}_\tau (2\pi P_{\bar{z}})\mathbb{R}_{\tau+h}-D_\tau \mathbb{R}^\ast_\tau (2\pi P_z)\mathbb{G}_{\tau+h} \\
						& =  \mathbb{S}_\tau (2\pi P_z)\mathbb{G}_{\tau+h}-\mathbb{R}_\tau (2\pi P_{\bar{z}})\mathbb{R}_{\tau+h},
					\end{align*}
					proving (b).
					The proofs of (c) and (d) are similar to that of (b).
				\end{proof}
				
				As an immediate corollary of Lemma \ref{lem:chr-der-derivatives}, Proposition \ref{prop:continuity}, the definition of $\mathbb{O}_\tau (L,M,N)$, and the product rule, we have
				
				\begin{corollary}\label{cor:chr-der} For $\mathbb{H}_\tau\in \{\mathbb{G}_\tau,\mathbb{R}_\tau,\mathbb{R}^\ast_\tau,\mathbb{S}_\tau\}$ we have $$\lim_{h\to 0}\Delta_h (\mathbb{H}_\tau)=\begin{cases} -\mathbb{G}_\tau (2\pi P_{\bar{z}}) \mathbb{R}_{\tau}-\mathbb{R}^\ast_\tau (2\pi P_{z}) \mathbb{G}_{\tau} & \mbox{if}\ \mathbb{H}_\tau = \mathbb{G}_\tau,\\
				\mathbb{S}_\tau (2\pi P_{z}) \mathbb{G}_{\tau} -\mathbb{R}_\tau (2\pi P_{\bar{z}})\mathbb{R}_{\tau} & \mbox{if}\ \mathbb{H}_\tau = \mathbb{R}_\tau,\\
				\mathbb{G}_\tau (2\pi P_{\bar{z}})\mathbb{S}_{\tau} -\mathbb{R}^\ast_\tau (2\pi P_{z})\mathbb{R}^\ast_{\tau} & \mbox{if}\ \mathbb{H}_\tau =\mathbb{R}^\ast_\tau,\\
				-\mathbb{S}_\tau (2\pi P_{z})\mathbb{R}^\ast_{\tau} - \mathbb{R}_\tau (2\pi P_{\bar{z}})\mathbb{S}_{\tau}& \mbox{if}\ \mathbb{H}_\tau=\mathbb{S}_\tau. \end{cases}$$
				More generally, if\, $\mathbb{H}_\tau = \mathbb{O}_\tau (L,M,N)$, then $[\mathbb{H}_\tau](z,w)$ is  differentiable in $\tau>0$ and $\partial_\tau [\mathbb{H}_\tau](z,w)=[ \partial_\tau\mathbb{H}_\tau](z,w)$, where $\partial_\tau \mathbb{H}_\tau = \lim\limits_{h\to 0} \Delta_h(\mathbb{H}_\tau)$. Moreover, $\tau^n\partial^n_\tau\mathbb{H}_\tau= \mathbb{O}_\tau (L,M,N)$ for $n\geq 1$.
				\end{corollary}
				\medskip

				\subsection{Proof of Theorem \ref{thm:chr-der-szegoderivatives}}
				\label{sec:kernels-comp}

				By Lemma \ref{lem:kernels-alg-structure} and Corollary \ref{cor:chr-der}, we have
				\begin{align*}
				\tau^M& \partial_\tau^M (W_\tau^\alpha \mathbb{R}_\tau^K \mathbb{S}_\tau (\mathbb{R}_\tau^\ast)^{K'} W_\tau^\beta) \\
				& = \sum_{\ell} \sum_{\ell'} \sum_{i=0}^{|\alpha|-\ell} \sum_{i'=0}^{|\beta|-\ell'} \mathbb{O}_\tau(K+K'-|\alpha|-|\beta|+2i+2i'+\ell+\ell',i+i',\ell+\ell'),\end{align*}
				where the outer two sums are over $0\leq \ell\leq \max(|\alpha|-2,0)$ and $0\leq \ell'\leq \max(|\beta|-2,0)$.
				
				If $P=P^{\mb{\sigma},2}$, then Lemma \ref{lem:kernels-kernelests} implies that
				$$\tau^M\partial_\tau^M (W_\tau^\alpha \mathbb{R}_\tau^K \mathbb{S}_\tau (\mathbb{R}_\tau^\ast)^{K'} W_\tau^\beta) = \sum_{\ell} \sum_{\ell'} \sum_{i=0}^{|\alpha|-\ell} \sum_{i'=0}^{|\beta|-\ell'} {\rm Op}^0_{(0,\Theta]} (K+K'-|\alpha|-|\beta|+\ell+\ell'),$$
				which, via Proposition \ref{prop:kernels-inclusion}, can be simplified to
				\begin{align*}
					\tau^M\partial_\tau^M & (W_\tau^\alpha \mathbb{R}_\tau^K \mathbb{S}_\tau (\mathbb{R}_\tau^\ast)^{K'} W_\tau^\beta) \\
					& =  {\rm Op}^0_{(0,\Theta]} (K+K'-|\alpha|-|\beta|+\max(|\alpha|-2,0) + \max(|\beta|-2,0))\\
					& =  {\rm Op}^0_{(0,\Theta]}(K+K'-\min(|\alpha|,2)-\min(|\beta|,2)).
				\end{align*}
				On the other hand, if $P=P^{\mb{\sigma},\max(m,|\alpha|,|\beta|)}$ then Proposition \ref{prop:kernels-inclusion} yields
				\begin{align*}
					\tau^M\partial_\tau^M &(W_\tau^\alpha \mathbb{R}_\tau^K \mathbb{S}_\tau (\mathbb{R}_\tau^\ast)^{K'} W_\tau^\beta) \\
					& =  \sum_{\ell} \sum_{\ell'} \sum_{i=0}^{|\alpha|-\ell} \sum_{i'=0}^{|\beta|-\ell'} {\rm Op}^0_{[\Theta,{+\infty})} (K+K'-|\alpha|-|\beta|+\ell+\ell') \\
					& =  {\rm Op}^0_{[\Theta,{+\infty})}(K+K'-|\alpha|-|\beta|).
				\end{align*}
				This completes the proof of Theorem \ref{thm:chr-der-szegoderivatives}.

\section{Proof of Theorem \ref{thm:szego-decomp}}\label{sec:main}

Before proving our main theorems, we need several elementary estimation tools.

\begin{proposition}\label{prop:szego-functionswithgeometricdecay}
	Let $f:(0,{+\infty})\rightarrow [0,{+\infty})$ be a positive, decreasing function and fix $\epsilon>0$.
	\begin{itemize}
		\item[(i)] If $f(2\tau)\leq 2^{-1-\epsilon}f(\tau),$ then there exists $C=C(\epsilon)$ so that $$\int_a^{+\infty} f(\tau)d\tau\leq Caf(a),\quad 0<a<{+\infty}.$$
		\item[(ii)] If $f(2\tau)\geq 2^{-1+\epsilon} f(\tau)$, then there exists $C=C(\epsilon)$ so that 
		$$\int_0^a f(\tau)d\tau \leq Caf(a),\quad 0<a<{+\infty}.$$
		\end{itemize}
\end{proposition}
\begin{proof}
	In case (i) we have
	\begin{align*}
	\int_a^{+\infty} f(\tau)d\tau  & =   \displaystyle\sum_{k=0}^{+\infty} \int_{2^k a}^{2^{k+1}a} f(\tau)d\tau  \\
	& =  \displaystyle\sum_{k=0}^{+\infty} 2^{k}\int_a^{2a} f(2^k \tau)d\tau  \leq  \displaystyle\sum_{k=0}^{+\infty} 2^{-\epsilon k}\int_a^{2a} f(\tau)d\tau  \leq  Caf(a),\end{align*}
	where the last estimate uses the fact that $f$ is decreasing.
	
	For (ii), we have
	$$\int_0^a f(\tau)d\tau = \sum_{k=0}^{+\infty} 2^{-k-1} \int_{a}^{2a} f(2^{-k-1}\tau)d\tau \leq \sum_{k=0}^{+\infty} 2^{-\epsilon(k+1)} \int_a^{2a} f(\tau)d\tau \leq Caf(a).$$
\end{proof}

\begin{proposition}\label{prop:main-elem-concave}
	Fix $Q,\epsilon,\nu>0,$ $N\geq 0$, and $0\leq a<{+\infty}$. Suppose that $f:[0,{+\infty})\rightarrow [0,{+\infty})$ is either
	\begin{itemize}
		\item[(a)] increasing with $f(2\tau)\leq 2^{Q} f(\tau)$, or
		\item[(b)] decreasing with $f(2\tau)\geq 2^{-1+Q}f(\tau).$
	\end{itemize}
	Then $$\int_a^{+\infty} f(\tau)t^Ne^{-\epsilon \tau^\nu}d\tau \lesssim f(1),$$
	where the constants depend on $Q,\ \epsilon,\ \nu$, and $N$, but are independent of $f$ and $a$.
\end{proposition}
\begin{proof}
	Because $f(\tau)\tau^N e^{-\epsilon \tau^\nu}\geq 0$, we need only prove the case where $a=0$.
	
	Throughout the proof $C$ denotes an arbitrary positive constant depending only on $Q,\ \epsilon,\ \nu$, and $N$.
	
	If (a) holds, then
	\begin{align*}
		\int_0^{+\infty} f(\tau)\tau^N e^{-\epsilon \tau^\nu}d\tau & \leq \int_0^1 f(\tau)d\tau + \sum_{k=0}^{+\infty} 2^{k(N+1)} \int_1^2 f(2^k \tau)\tau^N e^{-\epsilon 2^{k\nu}\tau^\nu}d\tau \\
		& \leq f(1) + \sum_{k=0}^{+\infty} 2^{k(N+1)} 2^N e^{-\epsilon 2^{k\nu}}f(2^{k+1}) \\
		& \leq f(1)\Big( 1 + \sum_{k=0}^{+\infty} 2^{k(N+1)+N+Q(k+1)}e^{-\epsilon 2^{k\nu}}\Big)\\
		& \leq C f(1).
	\end{align*}
	On the other hand, if (b) holds then
	$$\int_0^{+\infty} f(\tau)\tau^N e^{-\epsilon \tau^\nu}d\tau\leq \int_0^1 f(\tau)d\tau + f(1)\int_1^{+\infty} \tau^N e^{-\epsilon \tau^\nu}d\tau \leq Cf(1)$$
	by part (ii) of Proposition \ref{prop:szego-functionswithgeometricdecay}.
\end{proof}
\medskip

\begin{proposition}\label{prop:szego-oscillationsandboundedderivatives}
	Let $f,g:[0,+\infty)\to \mathbb{R}$ with $g$ non-negative and assume that there are constants $Q,M>0$ and $N\geq 0$ with
	\begin{itemize}
		\item[(a)] $\tau^k |\partial_\tau^k f(\tau)|\leq M\tau^N g(\tau)$ for $0\leq k\leq Q+2$,
	\end{itemize}
	 and either
	 \begin{itemize}
		\item[$\textit{(b)}_1$] $g$ increasing with $g(2\tau)\leq 2^Q g(\tau)$
	\end{itemize}
	or
	\begin{itemize}
		\item[$\textit{(b)}_2$] $g$ decreasing with $g(2\tau)\geq 2^{-1+Q}g(\tau)$.
	\end{itemize}	
	If $\lambda=x+iy$ with $x>0$ and $y\in\mathbb{R}$, then there exists $C>0$ (depending only on $Q$ and $M$) such that 
	\begin{equation}\label{eq:oscillationsandboundedderivativesest}
		\Big| \displaystyle\int_0^{+\infty} e^{-\lambda \tau} f(\tau)d\tau\Big| \leq C |\lambda|^{-1-N}g(|\lambda|^{-1}).
	\end{equation}
\end{proposition}
\begin{proof}
	Let $c=|\lambda|^{-1}\lambda$ and $a=|\lambda|^{-1}$. Then 
	\begin{align*}
		\displaystyle\int_0^{+\infty} e^{-\lambda \tau} f(\tau)d\tau & = a\displaystyle\int_0^{+\infty} e^{-c\tau}f(a\tau)d\tau \\
		& = a\displaystyle\int_0^1 e^{-c\tau}f(a\tau)d\tau + a\displaystyle\int_1^{+\infty} e^{-c\tau}f(a\tau)d\tau\\
		& =: I_1 + I_2.
	\end{align*}
	Note that by either a simple size estimate or part (ii) or Proposition \ref{prop:szego-functionswithgeometricdecay}, $$|I_1| \leq Ma^{1+N}\int_0^1 \tau^Ng(a\tau)d\tau \leq Ma^{1+N}\int_0^1 g(a\tau)d\tau \leq Ma^{1+N}g(a)$$ if either $(b)_1$ or $(b)_2$ holds. 
	
	For $I_2$, we integrate by parts $J>Q+N$ times, using the fact that $\Re(c)>0$ to compute the boundary terms at infinity, to obtain
	\begin{equation*}
		I_2 = \displaystyle\sum_{k=0}^{J} \frac{a^{k+1}}{c^k}e^{-c}f^{(k)}(a) + \displaystyle\frac{a^{J+2}}{c^{J+1}}\displaystyle\int_1^{+\infty} e^{-c\tau}f^{(J+1)}(a\tau)d\tau,
	\end{equation*}
	so we obtain
	\begin{align*}
		|I_2| & \lesssim \displaystyle\sum_{k=0}^{J} a^{k+1}a^{N-k} g(a) + a^{J+2}\displaystyle\int_1^{+\infty} |f^{(J+1)}(a\tau)|d\tau \\
		& \lesssim a^{1+N}g(a) + a^{J+2}\displaystyle\int_1^{+\infty} (a\tau)^{N-J-1}g(a\tau)d\tau \\
		& \lesssim a^{1+N}g(a) + a^{1+N}\int_1^{+\infty} \tau^{N-J-1}g(a\tau)d\tau.
		\end{align*}
If $(b)_1$ holds, then we can apply part (i) of Proposition \ref{prop:szego-functionswithgeometricdecay} (with $\epsilon=J-N-Q$) to see that
	\begin{equation*}
		|I_2|  \lesssim a^{1+N}g(a).
	\end{equation*}
	On the other hand, if $(b)_2$ holds then 
	$$|I_2|\lesssim a^{1+N} g(z) + a^{1+N}g(a)\int_1^{+\infty} \tau^{N-J-1}d\tau\lesssim a^{1+N}g(a)$$
	because $g$ is decreasing. This completes the proof.
\end{proof}
\medskip

We are now  ready to begin the proof of Theorem \ref{thm:szego-decomp}.

\begin{proof}[Proof of Theorem \ref{thm:szego-decomp}]
	
	By Remark \ref{rem:metric-invar} we may assume that $P=P^{\mb{0},2}$.
	Throughout the proof we fix $\mb{z},\mb{w}\in\b\Omega$, multi-indices $|\alpha|,|\beta|$, and $N,K,K'\geq 0$, and we let $\chi$ be as in Section \ref{sec:fn}. We also write $t=\Re(z_2)$ and $s=\Re(w_2)$. By Corollary \ref{cor:metric-mutilde}, we may choose $\sigma^\ast_\tau(w)=\mu^\ast(\mb{w},\tau^{-1})$, decreasing in $\tau$, so that $\sigma^\ast_\tau(w)\approx \sigma_\tau(w)$ (uniformly in $w$ and $\tau$), and $\sigma^\ast_{2\tau}(w)\geq 2^{-\frac{1}{2}}\sigma^\ast_\tau(w).$ 
	
	To simplify notation we write $\mathbb{H}_\tau=W_\tau^\alpha \mathbb{R}^K_\tau \mathbb{S}_\tau (\mathbb{R}^{\ast}_\tau)^{K'}W_\tau^\beta$. We also use the estimate $|B_d(\mb{z},\delta)|\approx \delta^2 \Lambda(\mb{z},\delta)$ without further mention.
	
	By (\ref{eq:fn-fkkkernel}) and (\ref{eq:fn-nkkkernel}), we have
	\begin{equation}
		T^N Z_{\mb{z}}^\alpha (Z^\beta_{\mb{w}})^\ast[\mathbb{F}_{K,K'}](\mb{z},\mb{w})=C\int_0^{{+\infty}} e^{2\pi i\tau(t-s)}\chi(\tau)\tau^N[\mathbb{H}_\tau](z,w)d\tau,
	\end{equation}
	and
	\begin{equation}
		T^N Z_{\mb{z}}^{\alpha} (Z^\beta_{\mb{w}})^{\ast}[\mathbb{N}_{K,K'}](\mb{z},\mb{w})=C\int_0^{{+\infty}} e^{2\pi i\tau(t-s)}(1-\chi(\tau))\tau^N[\mathbb{H}_\tau](z,w)d\tau,
	\end{equation}
	where $C=(2\pi i)^N$.

	If $\Phi^2,\Phi^\kappa:\mathbb{C}^2\to\mathbb{C}^2$ are the biholomorphisms constructed in Section \ref{sec:normalization-biholomorphic} associated (respectively) to $P=P^{\mb{w},2}$ and $P=P^{\mb{w},\kappa}$ for $\kappa=\max(m,|\alpha|,|\beta|)$, then let
	$$\tilde{\mathbb{H}}_{\tau,2}=\tilde{W}_{\tau,2}^\alpha \tilde{\mathbb{R}}^K_{\tau,2} \tilde{\mathbb{S}}_{\tau,2} (\tilde{\mathbb{R}}^{\ast}_{\tau,2})^{K'}\tilde{W}_{\tau,2}^\beta$$
	and 
	$$\tilde{\mathbb{H}}_{\tau,\kappa}=\tilde{W}_{\tau,\kappa}^\alpha \tilde{\mathbb{R}}^K_{\tau,\kappa} \tilde{\mathbb{S}}_{\tau,\kappa} (\tilde{\mathbb{R}}^{\ast}_{\tau,\kappa})^{K'}\tilde{W}_{\tau,\kappa}^\beta$$
	denote the operators $\mathbb{H}_\tau$ corresponding to $P^{\mb{w},2}$ and $P^{\mb{w},\kappa}$, respectively.

	By Lemma \ref{lem:fn-substitution}, we have
	\begin{align}
		T^N Z_{\mb{z}}^\alpha (Z^\beta_{\mb{w}})^\ast[\mathbb{F}_{K,K'}]&(\mb{z},\mb{w})\nonumber \\
		&= C\int_0^{{+\infty}} e^{2\pi i\tau(t-s-T_2(z,w))}\chi(\tau)\tau^N[\tilde{\mathbb{H}}_{\tau,2}](z-w,0)d\tau\label{eq:main-f}
	\end{align}
	and
	\begin{align}
		T^N Z_{\mb{z}}^\alpha (Z^\beta_{\mb{w}})^\ast[\mathbb{N}_{K,K'}]&(\mb{z},\mb{w})\nonumber \\
		&= C\int_0^{{+\infty}} e^{2\pi i\tau(t-s-T_\kappa(z,w))}(1-\chi(\tau))\tau^N[\tilde{\mathbb{H}}_{\tau,\kappa}](z-w,0)d\tau\label{eq:main-n}
	\end{align}
	
    Theorem \ref{thm:chr-der-szegoderivatives} implies that, writing $\tilde{\mb{0}}=\Phi(\mb{w})$,
	$$\tau^M \partial_\tau^M \tilde{\mathbb{H}}_{\tau,\kappa}\in {\rm Op}^{\t 0}_{[1,{+\infty})}(K+K'-|\alpha|-|\beta|)$$
	and 
	$$\tau^M \partial_\tau^M \tilde{\mathbb{H}}_{\tau,2}\in {\rm Op}^{\t 0}_{(0,2]}(K+K'-\min(|\alpha|,2)-\min(|\beta|,2)),$$
	so that, because $\t\rho_\tau(z-w,0)\approx \t\rho_\tau(z,w)$ and $\tilde{\sigma}_\tau(\t 0)=\sigma_\tau(w)$, for a small constant $\epsilon>0$ we have
	\begin{equation}\label{eq:main-fdecay}
	\begin{aligned}
	|\partial_\tau^M & [\tilde{\mathbb{H}}_{\tau,2}](z-w,0)|\\
	& \lesssim \tau^{-M}\sigma_\tau(w)^{-2+K+K'-\min(|\alpha|,2)-\min(|\beta|,2)} e^{-\epsilon\t\rho_\tau(z,w)},\quad 0<\tau\leq 2\end{aligned}\end{equation}
	and
	\begin{equation}\label{eq:main-ndecay}
	\begin{aligned}|\partial_\tau^M & [\tilde{\mathbb{H}}_{\tau,\kappa}](z-w,0)|\\
	& \lesssim \tau^{-M}\sigma_\tau(w)^{-2+K+K'-|\alpha|-|\beta|} e^{-\epsilon\t\rho_\tau(z,w)},\quad 1\leq \tau<{+\infty}.\end{aligned}\end{equation}
	
	\textbf{Estimates for $\mathbb{N}_{K,K'}$}
	
	We first focus on the estimates for $\mathbb{N}_{K,K'}$. To begin, we fix $0<c_1\ll c_2\ll 1$ (to be chosen later).  There are four cases.
	\bigskip
	
	\noindent\textbf{Case $(\mathbb{N})_1$:} $d(\mb{z},\mb{w})\geq c_2$ and $|z-w|\geq c_1$.
	\medskip
	
	Note that equation (\ref{eq:geom-metric}) implies that 
	\begin{align}
		d(\mb{z},\mb{w}) & \approx  |z-w|+\mu(\mb{w},|t-s-T_2(z,w)|)\nonumber \\
		& =  |z-w| + \mu\Big(\mb{w},\Big|t-s-T_\kappa (z,w) -2\Im\Big(\sum_{k=3}^\kappa \frac{1}{k!} \frac{\partial^k P}{\partial z^k}(w)(z-w)^k\Big)\Big|\Big)\nonumber \\
		& \lesssim  |z-w|+\mu(\mb{w},|t-s-T_\kappa (z,w)|)\nonumber \\
		&  \qquad +\mu\Big(\mb{w},\Big|2\Im\Big(\sum_{k=3}^\kappa \frac{1}{k!} \frac{\partial^k P}{\partial z^k}(w)(z-w)^k\Big)\Big|\Big)\nonumber \\
		& \lesssim |z-w|+\mu(\mb{w},|t-s-T_\kappa (z,w)|)+ |z-w|^{\kappa/2}\nonumber \\
		& \approx  |z-w|^{\kappa/2}+\mu(\mb{w},|t-s-T_\kappa (z,w)|),\label{eq:main-case1n}
	\end{align}
	where in the fourth line we used the fact that $|\nabla^k P|\lesssim 1$ for $k\geq 2$. The estimate $\mu(\mb{w},\delta)\lesssim \delta^{\frac{1}{2}}$ for $\delta\gtrsim 1$ follows from equation (\ref{eq:metrics-lambdamu}) and Lemma \ref{lem:geom-uftisaq}.
		
	The proof of (\ref{eq:szego-near-growth}) in this case requires us to show that the right-hand-side of (\ref{eq:main-n}) is controlled by $\Lambda(\mb{w},d(\mb{z},\mb{w}))^{-M}$ for large $M$, and by (\ref{eq:main-case1n}) it suffices to show that we can bound (\ref{eq:main-n}) by large negative powers of $|t-s-T_\kappa(z,w)|$ and $\Lambda(\mb{w},|z-w|)\approx |z-w|^2$.
	
	For small enough $\epsilon,\nu>0$, equations (\ref{eq:main-n}) and (\ref{eq:main-ndecay}) and Proposition \ref{prop:chr-rec-specforms} give
	$$|[T^N Z^\alpha \mathbb{N}_{K,K'}Z^\beta](\mb{z},\mb{w})|\lesssim \int_1^{+\infty} \tau^N \sigma_\tau(w)^{-2+K+K'-|\alpha|-|\beta|} e^{-\epsilon (\tau\Lambda(w,|z-w|))^\nu}d\tau,$$
	which we estimate by noting that, because $\Lambda(\mb{w},c_1)\gtrsim 1$ uniformly in $\mb{w}$,
	\begin{align*}
		 \int_1^{+\infty} \tau^N \sigma_\tau &(w)^{-2+K+K'-|\alpha|-|\beta|} e^{-\epsilon (\tau\Lambda(w,|z-w|))^\nu}d\tau \\
		& \leq  e^{-\frac{\epsilon}{2}\Lambda(\mb{w},|z-w|)^\nu} \int_1^{+\infty} \tau^N \sigma_\tau(w)^{-2+K+K'-|\alpha|-|\beta|}e^{-\frac{\epsilon}{2}\Lambda(\mb{w},c_1)^\nu \tau^\nu}d\tau\\
		& \approx  e^{-\frac{\epsilon}{2}\Lambda(\mb{w},|z-w|)^\nu} \int_1^{+\infty} \tau^N \sigma^\ast_\tau(w)^{-2+K+K'-|\alpha|-|\beta|}e^{-\frac{\epsilon}{2}\Lambda(\mb{w},c_1)^\nu \tau^\nu}d\tau\\
		& \lesssim  e^{-\frac{\epsilon}{2}\Lambda(\mb{w},|z-w|)^\nu},
		\end{align*}
	where in the last line we were able to apply Proposition \ref{prop:main-elem-concave} because $-\frac{1}{2}(-2+K+K'-|\alpha|-|\beta|)>-1$. This gives the fast decay in $\Lambda(\mb{w},|z-w|)$.
	
	On the other hand, by integrating by parts $M$ times\footnote{Recall that we have suppressed the term $e^{-2\pi \tau\epsilon}$ in the integral, and that our estimates are independent of $\epsilon>0$.} and using the support of $(1-\chi(\tau))$ to compute the boundary terms at $\tau=1$, we have
	\begin{align*}
		 \Big|\int_0^{{+\infty}} & e^{2\pi i\tau(t-s-T_\kappa(z,w))} (1-\chi(\tau))\tau^N[\tilde{\mathbb{H}}_{\tau,\kappa}](z-w,0)d\tau\Big|\\
		& =  |2\pi i (t-s-T_\kappa (z,w))|^{-M}\\
		& \qquad\qquad\times\Big|\int_1^{+\infty} e^{2\pi i \tau (t-s-T_\kappa (z,w))}\partial^M_\tau ((1-\chi(\tau))\tau^N [\tilde{\mathbb{H}}_{\tau,\kappa}](z-w,0))d\tau\Big|\\
		& \lesssim  |t-s-T_\kappa (z,w)|^{-M}\\
		& \qquad\qquad\times\int_1^{+\infty} \tau^{N-M}\sigma_\tau(w)^{-2+K+K'-|\alpha|-|\beta|}e^{-\epsilon(\tau \Lambda(\mb{w},|z-w|))^\nu}d\tau\\
		& \leq  |t-s-T_\kappa (z,w)|^{-M}\int_1^{+\infty} \tau^{N-M}\sigma_\tau(w)^{-2+K+K'-|\alpha|-|\beta|}e^{-\epsilon\Lambda(\mb{w},c_1)^\nu\tau ^\nu}d\tau\\
		& \lesssim  |t-s-T_\kappa (z,w)|^{-M}.
		\end{align*}
		This completes the proof of Case $(\mathbb{N})_1$.
		\bigskip
		
		\noindent\textbf{Case $(\mathbb{N})_2$:} $d(\mb{z},\mb{w})\geq c_2$ and $|z-w|\leq c_1$.
		\medskip
		
		In this case we have $|z-w|^{\kappa/2}\lesssim |z-w|\ll d(\mb{z},\mb{w})$, and therefore we may compute as in (\ref{eq:main-case1n}) to show that
		$$d(\mb{z},\mb{w})\approx \mu(\mb{w},|t-s-T_\kappa(z,w)|)$$
		provided that $c_1$ is chosen sufficiently small relative to $c_2$.
		
		Integrating by parts as in Case $(\mathbb{N})_1$, we have
		\begin{align*}
			 \Big|\int_0^{{+\infty}} & e^{2\pi i\tau(t-s-T_\kappa(z,w))}(1-\chi(\tau))\tau^N[\tilde{\mathbb{H}}_{\tau,\kappa}](z-w,0)d\tau\Big|\\
			& =  |2\pi i (t-s-T_\kappa (z,w))|^{-M}\\
			& \qquad\qquad\times\Big|\int_1^{+\infty} e^{2\pi i \tau (t-s-T_\kappa (z,w))}\partial^M_\tau ((1-\chi(\tau))\tau^N [\tilde{\mathbb{H}}_{\tau,\kappa}](z-w,0))d\tau\Big|\\
			& \lesssim  |t-s-T_\kappa (z,w)|^{-M}\\
			& \qquad\qquad\times\int_1^{+\infty} \tau^{N-M}\sigma_\tau(w)^{-2+K+K'-|\alpha|-|\beta|}e^{-\epsilon(\tau \Lambda(\mb{w},|z-w|))^\nu}d\tau\\
			& \leq  |t-s-T_\kappa (z,w)|^{-M}\int_1^{+\infty} \tau^{N-M}\sigma_\tau(w)^{-2+K+K'-|\alpha|-|\beta|}d\tau\\
			& \lesssim  t-s-T_\kappa (z,w)|^{-M} \int_1^{+\infty} \tau^{N-M+\frac{1}{2}|-2+K+K'-|\alpha|-|\beta||}d\tau\\
			& \lesssim  |t-s-T_\kappa (z,w)|^{-M}\approx \Lambda(\mb{w},d(\mb{w},\mb{z}))^{-M}
		\end{align*}
		provided that $M$ is sufficiently large. This completes the proof of Case $(\mathbb{N})_2$.
		\bigskip
		
		\noindent\textbf{Case $(\mathbb{N})_3$:} $d(\mb{z},\mb{w})\leq c_2$ and $|z-w|\geq c_1 d(\mb{z},\mb{w})$.
		\medskip
		
		Note that $d(\mb{z},\mb{w})\approx |z-w|$ in this case.
		By applying (\ref{eq:main-ndecay}) and making the substitution $\tau\mapsto \tau\Lambda(\mb{w},|z-w|)^{-1}$ we obtain
		\begin{align*}
			& \Big|\int_0^{{+\infty}} e^{2\pi i\tau(t-s-T_\kappa(z,w))}(1-\chi(\tau))\tau^N[\tilde{\mathbb{H}}_{\tau,\kappa}](z-w,0)d\tau\Big|\\
			& \lesssim  \int_1^{+\infty} \tau^N\sigma_\tau(w)^{-2+K+K'-|\alpha|-|\beta|}e^{-\epsilon(\tau\Lambda(\mb{w},|z-w|))^\nu}d\tau \\
			& =  \Lambda(\mb{w},|z-w|)^{-1-N}\int_{\Lambda(\mb{w},|z-w|)}^{+\infty} \tau^N \sigma_{\tau\Lambda(\mb{w},|z-w|)^{-1}}(w)^{-2+K+K'-|\alpha|-|\beta|}e^{-\epsilon \tau^\nu}d\tau\\
			& \approx  \Lambda(\mb{w},|z-w|)^{-1-N}\int_{\Lambda(\mb{w},|z-w|)}^{+\infty} \tau^N \sigma^\ast_{\tau\Lambda(\mb{w},|z-w|)^{-1}}(w)^{-2+K+K'-|\alpha|-|\beta|}e^{-\epsilon \tau^\nu}d\tau\\
			& \lesssim  \Lambda(\mb{w},|z-w|)^{-1-N}\sigma^\ast_{\Lambda(\mb{w},|z-w|)^{-1}}(w)^{-2+K+K'-|\alpha|-|\beta|}\\
			& \approx  \Lambda(\mb{w},|z-w|)^{-1-N}|z-w|^{-2+K+K'-|\alpha|-|\beta|}\\
			& \approx  \Lambda(\mb{w},d(\mb{z},\mb{w}))^{-1-N}d(\mb{z},\mb{w})^{-2+K+K'-|\alpha|-|\beta|},
			\end{align*}
			where in the fourth line we were able to apply Proposition \ref{prop:main-elem-concave} because $-\frac{1}{2}(-2+K+K'-|\alpha|-|\beta|)>-1$. This completes the proof of Case $(\mathbb{N})_3$.
			\bigskip
			
	\noindent\textbf{Case $(\mathbb{N})_4$:} $d(\mb{z},\mb{w})\leq c_2$ and $|z-w|\leq c_1 d(\mb{z},\mb{w})$.
	\medskip
	
	By part (b) of Lemma \ref{lem:metric-approx} and the assumption that $|z-w|\ll d(\mb{z},\mb{w})$, we have
	$$d(\mb{z},\mb{w})\approx \mu(\mb{w},|t-s-T_\kappa(z,w)|).$$		
	Write $g(\tau)=\sigma^\ast_\tau(w)^{-2+K+K'-|\alpha|-|\beta|}$ and $f(\tau)=(1-\chi(\tau))\tau^N[\tilde{\mathbb{H}}_{\tau,\kappa}](z-w,0)$, and note that $f(\tau)$ and $g(\tau)$ satisfy the hypotheses of Proposition \ref{prop:szego-oscillationsandboundedderivatives}, where in the case $-2+K+K'-|\alpha|-|\beta|>0$ we again use Corollary \ref{cor:metric-mutilde} and the fact that $-\frac{1}{2}(-2+K+K'-|\alpha|-|\beta|)>-1$. We therefore have
	\begin{align*}
		 \Big|\int_0^{{+\infty}} & e^{2\pi i\tau(t-s-T_\kappa(z,w))}(1-\chi(\tau))\tau^N[\tilde{\mathbb{H}}_{\tau,\kappa}](z-w,0)d\tau\Big|\\
		& \lesssim  |t-s-T_\kappa(z,w)|^{-1-N}g(|t-s-T_\kappa(z,w)|^{-1})\\
		& \approx  |t-s-T_\kappa(z,w)|^{-1-N}\mu(\mb{w},|t-s-T_\kappa(z,w)|)^{-2+K+K'-|\alpha|-|\beta|}\\
		& \approx  \Lambda(\mb{w},d(\mb{z},\mb{w}))^{-1-N} d(\mb{w},\mb{w})^{-2+K+K'-|\alpha|-|\beta|},
	\end{align*}
	which finishes the proof of (\ref{eq:szego-near-growth}).
\bigskip	

\textbf{Estimates for $\mathbb{F}_{K,K'}$}
\medskip

We now establish the estimates (\ref{eq:szego-far-growth}) for $\mathbb{F}_{K,K'}$. To begin, fix $0<c_1\ll c_2\ll 1$ (to be chosen later).  Throughout, we use the fact that $\sigma_\tau(w)\approx \tau^{-\frac{1}{2}}$ if $0<\tau\leq 2$. Then there are three cases.
\bigskip

	\noindent\textbf{Case $(\mathbb{F})_1$:} $d(\mb{z},\mb{w})\leq c_2$.
	\medskip
	
	We apply the estimate (\ref{eq:main-fdecay}) to the right-hand-side of (\ref{eq:main-f}), thereby obtaining
	\begin{align*}
		 \Big|\int_0^{{+\infty}} & e^{2\pi i\tau(t-s-T_2(z,w))}\chi(\tau)\tau^N[\tilde{\mathbb{H}}_{\tau,2}](z-w,0)d\tau\Big| \\
		& \lesssim  \int_0^2 \tau^N \sigma_\tau(w)^{-2+K+K'-\min(|\alpha|,2)-\min(|\beta|,2)}e^{-\epsilon \t\rho_\tau(z,w)}d\tau \\
		& \lesssim  \int_0^2 \tau^{N+1-\frac{1}{2}(K+K'-\min(|\alpha|,2)-\min(|\beta|,2))}d\tau\lesssim  1
		\end{align*}
	provided that $K+K'<4+2N+\min(|\alpha|,2)+\min(|\beta|,2).$
	\bigskip
	
	\noindent\textbf{Case $(\mathbb{F})_2$:} $d(\mb{z},\mb{w})\geq c_2$ and $|z-w|\geq c_1 d(\mb{z},\mb{w})$.
	\medskip
	
	We estimate as in Case $(\mathbb{F})_1$, but now use Proposition \ref{prop:main-elem-concave} and $d(\mb{z},\mb{w})\approx |z-w|$:
	\begin{align*}
		 \Big|\int_0^{{+\infty}} & e^{2\pi i\tau(t-s-T_2(z,w))}\chi(\tau)\tau^N[\tilde{\mathbb{H}}_{\tau,2}](z-w,0)d\tau\Big| \\
		& \lesssim  \int_0^2 \tau^{N+1-\frac{1}{2}(K+K'-\min(|\alpha|,2)-\min(|\beta|,2))}e^{-\epsilon (\tau\Lambda(\mb{w},|z-w|))^\nu}d\tau \\
		& \lesssim  \Lambda(\mb{w},|z-w|)^{-N-2+\frac{1}{2}(K+K'-\min(|\alpha|,2)-\min(|\beta|,2))}\\
		&  \qquad\times\int_0^{\Lambda(\mb{w},|z-w|)} \tau^{N+1-\frac{1}{2}(K+K'-\min(|\alpha|,2)-\min(|\beta|,2))}e^{-\epsilon\tau^\nu}d\tau \\
		& \lesssim  \Lambda(\mb{w},|z-w|)^{-N-2+\frac{1}{2}(K+K'-\min(|\alpha|,2)-\min(|\beta|,2))}\\
		& \approx  \Lambda(\mb{w},|z-w|)^{-1-N} |z-w|^{-2+K+K'-\min(|\alpha|,2)-\min(|\beta|,2)} \\
		& \approx  \Lambda(\mb{w},d(\mb{z},\mb{w}))^{-1-N} d(\mb{z},\mb{w})^{-2+K+K'-\min(|\alpha|,2)-\min(|\beta|,2)},
	\end{align*}
	 where in the third line we needed $K+K'<\min(|\alpha|,2)+\min(|\beta|,2)+4$, and in penultimate line we used that $\Lambda(\mb{w},\delta)\approx \delta^2$ for $\delta\gtrsim 1$.
	\bigskip
	
	\noindent\textbf{Case $(\mathbb{F})_3$:} $d(\mb{z},\mb{w})\geq c_2$ and $|z-w|\leq c_1d(\mb{z},\mb{w})$.
	\medskip

	Note that if $f(\tau)=\chi(\tau)\tau^N [\tilde{\mathbb{H}}_{\tau,2}](z-w,0)$ and $$g(\tau)=\sigma^\ast_\tau(w)^{-2+K+K'-\min(|\alpha|,2)-\min(|\beta|,2)},$$ then (as in Case $(\mathbb{N})_4$) $f$ and $g$ satisfy the hypotheses of Proposition \ref{prop:szego-oscillationsandboundedderivatives} (here we again use the fact that $K+K'<\min(|\alpha|,2)+\min(|\beta|,2)+4$). 
	Moreover, by taking $c_1$ sufficiently small relative to $c_2$, applying Lemma \ref{lem:metric-approx}, and arguing as in Case $(\mathbb{N})_1$ we can guarantee that $$d(\mb{z},\mb{w})\approx \mu(\mb{w},|t-s-T_\kappa(z,w)|)\approx \mu(\mb{w},|t-s-T_2(z,w)|).$$ 
	Together these observations yield
	\begin{align*}
		 \Big|&\int_0^{{+\infty}}  e^{2\pi i\tau(t-s-T_2(z,w))}\chi(\tau)\tau^N[\tilde{\mathbb{H}}_{\tau,2}](z-w,0)d\tau\Big| \\
		& \lesssim  |t-s-T_2(z,w)|^{-1-N}g(|t-s-T_2(z,w)|^{-1}) \\
		& \approx  |t-s-T_2(z,w)|^{-1-N}\mu(\mb{w},|t-s-T_2(z,w)|)^{-2+K+K'-\min(|\alpha|,2)-\min(|\beta|,2)}\\
		& \approx  \Lambda(\mb{w},d(\mb{z},\mb{w}))^{-1-N} d(\mb{z},\mb{w})^{-2+K+K'-\min(|\alpha|,2)-\min(|\beta|,2)},
	\end{align*}
	which completes the proof of (\ref{eq:szego-far-growth}).	
\end{proof}

\section{Proof of Theorem \ref{thm:szego-mapping}}
\label{sec:mapping}

We now turn to the proof of Theorem \ref{thm:szego-mapping}(a).
We first show that $W^\alpha \mathbb{F}$ is bounded on $L^p({\rm b}\Omega)$ for $|\alpha|\geq 1$.
Note first that, by Theorem \ref{thm:szego-decomp},
$$[W^\alpha\mathbb{F}](\mb{z},\bullet),[W^\alpha\mathbb{F}](\bullet,\mb{w})\in L^1({\rm b}\Omega)\cap L^{\infty}({\rm b}\Omega)$$
as long as $|\alpha|\geq 1$.
Therefore, Minkowski's inequality for integrals implies that
\begin{equation}\label{eq:NL-F-1ormore}
	\|W^\alpha \mathbb{F}[f]\|_{L^p({\rm b}\Omega)}\leq C_\alpha \|f\|_{L^p({\rm b}\Omega)},\quad |\alpha|\geq 1.
\end{equation}
The proof that $\mathbb{F}$ is bounded on $L^p({\rm b}\Omega)$ is the same as that which shows that $W^\alpha \mathbb{N}$ is bounded on $NL^p_k({\rm b}\Omega)$ for $|\alpha|\geq 0$, so we now focus only on this latter case.
Choosing $K=0$ and $K'=|\alpha|$, Theorem \ref{thm:szego-decomp} implies that there exists $\gamma'$ with $|\gamma'|=|\alpha|$ such that $W^\alpha \mathbb{N}=W^\alpha\mathbb{N}_{0,|\alpha|}W^{\gamma'}$, where $W^\alpha\mathbb{N}_{0,|\alpha|}$ satisfies all of the same estimates as does $\mathbb{N}$.
Thus, it suffices to show that the estimates (\ref{eq:szego-near-growth}) imply that $\mathbb{N}=\mathbb{N}_{0,0}$ is bounded on $L^p({\rm b}\Omega)$ for $1<p<+\infty$.

We will apply the general $T(1)$-theorem of \cite{DavidJourneSemmes1985}.
First we show that $\mathbb{N}$ is restrictedly bounded, in the sense that
\begin{equation}\label{eq:N-restrictedlybounded}
	\|\mathbb{N}[\phi^{\mb{\zeta}}_j]\|_{L^2({\rm b}\Omega)}\lesssim |B_d(\mb{\zeta},2^{-j})|^{\frac{1}{2}}
\end{equation}
for any function $\phi^{\mb{\zeta}}_j$ such that 
\begin{itemize}
	\item[(a)] ${\rm supp}\; \phi^{\mb{\zeta}}_j \subset B_d(\mb{\zeta},2^{-j}),$
	\item[(b)] $\|(2^{-j}W)^\alpha \phi^{\mb{\zeta}}_j\|_\infty \lesssim 1$ for any $0\leq|\alpha|\leq 2$.
\end{itemize}
Such a function $\phi^{\mb{\zeta}}_j$ is called a \emph{normalized bump function} adapted to $B_d(\mb{\zeta},2^{-j})$.
The existence of such functions for arbitrary $\mb{\zeta}$ and $2^{-j}$ is guaranteed by Lemma \ref{lem:metric-smooth}.

To do this, fix $C>1$ and writing $dm_z:=dm(z,\Re(z_2))$ throughout this section for brevity, we have
\begin{align*}
	\|\mathbb{N}[\phi^{\mb{\zeta}}_j]\|_{L^2({\rm b}\Omega)}^2 & = \int_{d(\mb{z},\mb{\zeta})\geq C2^{-j}} |\mathbb{N}[\phi^{\mb{\zeta}}_j](\mb{z})|^2dm_z \\
	&  \quad + \int_{d(\mb{z},\mb{\zeta})\leq C2^{-j}} |\mathbb{N}[\phi^{\mb{\zeta}}_j](\mb{z})|^2dm_z \\
	& =: I_1+I_2.
\end{align*}

For $I_1$, we note that $d(\mb{z},\mb{w})\approx d(\mb{z},\mb{\zeta})$ when $d(\mb{w},\mb{\zeta})\leq 2^{-j}< C2^{-j}\leq d(\mb{z},\mb{\zeta})$ to obtain
\begin{align*}
	|I_1| & \lesssim  \int_{d(\mb{z},\mb{\zeta})\geq C2^{-j}} \Big( \int_{B_d(\mb{\zeta},2^{-j})} \frac{1}{|B_d(\mb{z},d(\mb{z},\mb{w}))|}dm_w\Big)^2 dm_z \\
	& \lesssim  \int_{d(\mb{z},\mb{\zeta})\geq C2^{-j}} \frac{|B_d(\mb{\zeta},2^{-j})|^2}{|B_d(\mb{z},d(\mb{z},\mb{\zeta}))|^2} dm_z \\
	& \lesssim  |B_d(\mb{\zeta},2^{-j})|^2 \displaystyle\sum_{k=1}^{+\infty} \int_{d(\mb{z},\mb{\zeta})\leq C 2^{k-j}}\frac{1}{|B_d(\mb{z},d(\mb{z},\mb{\zeta}))|^2}dm_z \\
	& \approx  |B_d(\mb{\zeta},2^{-j})|^2 \displaystyle\sum_{k=1}^{+\infty} \frac{1}{|B_d(\mb{\zeta},2^{k-j})|} \\
	& \lesssim |B_d(\mb{\zeta},2^{-j})|^2 \displaystyle\sum_{k=1}^{+\infty} \frac{2^{-k4}}{|B_d(\mb{\zeta},2^{-j})|} \\
	& \lesssim  |B_d(\mb{\zeta},2^{-j})|,
\end{align*}
as desired.

For $I_2$, write $\mathbb{N}_{0,0}=\mathbb{N}_{0,1}Z$ and observe that
\begin{align*}
	 \\
	|I_2| & \lesssim  \int_{d(\mb{z},\mb{\zeta})\leq C2^{-j}} \Big( \int_{B_d(\mb{\zeta},2^{-j})}\frac{d(\mb{z},\mb{w})2^j}{|B_d(\mb{z},d(\mb{z},\mb{w}))|}dm_w\Big)^2dm_z \\
	& \approx  2^{2j}\int_{d(\mb{z},\mb{\zeta})\leq C2^{-j}}\Big(\sum_{k=-\infty}^0 \int_{d(\mb{z},\mb{w})\approx C 2^{k-j}}\frac{2^{k-j}}{|B_d(\mb{z},2^{k-j})|}dm_w\Big)^2 dm_z \\
	& \lesssim  2^{2j} \int_{d(\mb{z},\mb{\zeta})\leq C2^{-j}}\Big(\sum_{k=-\infty}^0 2^{k-j}\Big)^2 dm_z \\
	& \approx  |B_d(\mb{\zeta},2^{-j})|,
\end{align*} 
proving (\ref{eq:N-restrictedlybounded}).
As $\mathbb{N}^\ast$ is assumed to satisfy the same estimates as $\mathbb{N}$, $\mathbb{N}^\ast$ is also restrictedly bounded.

It remains to show that $\mathbb{N}[1],\mathbb{N}^\ast[1]\in BMO({\rm b}\Omega).$ As above, it suffices to show the result for $\mathbb{N}[1]$.
Fix $\mb{z}_0\in {\rm b}\Omega$, $\delta>0$, and let $a(\mb{z})$ be an $H^1({\rm b}\Omega)$ atom associated to $B_d(\mb{z}_0,\delta)$. 
That is,
\begin{enumerate}
	\item[(a)] $\int a=0$,
	\item[(b)] ${\rm supp}\; a\subset B_d(\mb{z}_0,\delta)$,
	\item[(c)] $|a|\leq |B_d(\mb{z}_0,\delta)|^{-1}$.
\end{enumerate}

Let $\chi$ be smooth with $\chi(t)\equiv 1$ for $t\leq 1$ and $\chi(t)\equiv 0$ for $\chi(t)\geq 2$, and write $\eta^{\mb{w}}_j(\mb{z})=\chi(d^\ast(\mb{z},\mb{w})2^{-j})$.

The arguments in Chapter 7 of \cite{Stein1993} imply that $\mathbb{N}(\eta_j^{\mb{0}})$ is uniformly in $BMO({\rm b}\Omega)$ for all $j$.
Therefore, to prove that $\mathbb{N}[1]\in BMO({\rm b}\Omega)$ it suffices to show that there exists $D<+\infty$ (independent of $a$) so that
$$\lim_{j\rightarrow +\infty}|\lan \mathbb{N}^\ast(a),\eta^{\mb{0}}_j\ran|\leq D.$$
To start, choose $k$ such that $C\delta\leq 2^{k}\leq 2C\delta$, (for some fixed large $C$) and write
\begin{align*}
	\lan \mathbb{N}^\ast(a),\eta_j^{\mb{0}}\ran & = \int_{{\rm b}\Omega}(1-\eta_{k}^{\mb{z}_0})\mathbb{N}^\ast[a]\eta_j^{\mb{0}}dm + \int_{{\rm b}\Omega}\eta_{k}^{\mb{z}_0}\mathbb{N}^\ast[a]\eta_j^{\mb{0}}dm \\
	& =: I_1 + I_2.
\end{align*}

For $I_1$, we use the fact that $\int a=0$ to write
\begin{align}
|I_1| =  \Big| \int_{\b\Omega} \Big\{(1-&\eta_k^{\mb{z}_0}(\mb{z}))\eta_j^{\mb{0}}(\mb{z}) \nonumber\\
	 & \times\int_{B_d(\mb{z}_0,\delta)}\Big[[\mathbb{N}^\ast](\mb{z},\mb{w})-[\mathbb{N}^\ast](\mb{z},\mb{z}_0)\Big]a(\mb{w})dm_w\Big\}dm_z\Big|.\label{eq:cancel-atomcancel}
\end{align}

Now let $\gamma:[0,1]\to {\rm b}\Omega$ be a piecewise smooth path with $\gamma(r)\in B_d(\mb{z}_0,\delta)$ for all $r$, with $\gamma(0)=\mb{z}_0$, $\gamma(1)=\mb{w}$, and $$\gamma'(r)=\alpha(r)\delta X(\gamma(r))+\beta(r)\delta Y(\gamma(r))\ a.e.,$$
where $\alpha,\beta:[0,1]\to[0,1]$ are piecewise constant and $\||\alpha|^2+|\beta|^2\|_\infty\leq 1.$
Then
\begin{align*}
	[\mathbb{N}^\ast](\mb{z},\mb{w})-[\mathbb{N}^\ast](\mb{z},\mb{z}_0) & = \int_0^1 \frac{d}{dr}[\mathbb{N}^\ast (\mb{z},\gamma(r))]dr \\
	& = \int_0^1 \Big\{\alpha(r)\delta X_{\mb{w}}[\mathbb{N}^\ast](\mb{z},\gamma(r))+\beta(r)\delta Y_{\mb{w}}[\mathbb{N}^\ast](\mb{z},\gamma(r))\Big\}dr,
\end{align*}
so that the derivative estimates in (\ref{eq:szego-near-growth}) imply that
\begin{align}
	|[\mathbb{N}^\ast](\mb{z},\mb{w})-[\mathbb{N}^\ast](\mb{z},\mb{z}_0)|& \lesssim \delta\int_0^1 \frac{d(\mb{z},\gamma(r))^{-1}}{|B_d(\mb{z},d(\mb{z},\gamma(r)))|}dr.\label{eq:cancel-cancel}
\end{align}
Because $(1-\eta_k^{\mb{z}_0}(\mb{z}))$ is supported where $d^\ast(\mb{z},\mb{z}_0)\geq 2^k\geq C\delta$, and $\gamma(r)\in B_d(\mb{z}_0,\delta)$, we have $$d(\mb{z},\gamma(r))\approx d(\mb{z},\mb{z}_0).$$
Combining this observation with (\ref{eq:cancel-atomcancel}) and (\ref{eq:cancel-cancel}) yields
\begin{align*}
|I_1|	& \lesssim  \delta\int_{\b\Omega} (1-\eta_k^{\mb{z}_0}(\mb{z}))\eta_j^{\mb{0}}(\mb{z})\frac{d(\mb{z},\mb{z}_0)^{-1}}{|B_d(\mb{z},d(\mb{z},\mb{z}_0))|}dm_z \\
	& \lesssim  \delta \displaystyle\sum_{M=1}^{+\infty} \int_{d(\mb{z},\mb{z}_0)\approx 2^M\delta}\frac{d(\mb{z},\mb{z}_0)^{-1}}{|B_d(\mb{z},d(\mb{z},\mb{z}_0))|} dm_z \\
	& \lesssim  \delta \displaystyle\sum_{M=1}^{+\infty} 2^{-M}\delta^{-1} \\
	& \lesssim  1.
\end{align*}

For $I_2$, write $\mathbb{N}=\mathbb{N}_{0,1}Z$ and take $j$ so large that $\eta_j^{\mb{0}}\equiv 1$ on the support of $\eta_k^{\mb{z}_0}$.
We then integrate by parts to obtain
\begin{align*}
	|I_2|& = \Big|\int_{\b\Omega} Z(\eta_k^{\mb{z}_0}(\mb{z})) \mathbb{N}_{1,0}[a](\mb{z})dm_z\Big|\\
	& \lesssim  \delta^{-1}\int_{B_d(\mb{z}_0,C\delta)}\int_{B_d(\mb{z}_0,C\delta)}\frac{d(\mb{z},\mb{w})}{|B_d(\mb{z},d(\mb{z},\mb{w}))|\cdot |B_d(\mb{z}_0,\delta)|}dm_w dm_z \\
	& \lesssim  1.
\end{align*}
Hence $\mathbb{N}[1]$ (and $\mathbb{N}^\ast[1]$) are in $BMO({\rm b}\Omega)$, so that $\mathbb{N}:L^p({\rm b}\Omega)\rightarrow L^p({\rm b}\Omega)$, as desired.
\bigskip

We now begin our proof of part (b) of Theorem \ref{thm:szego-mapping}.
By the interpolation arguments in \cite{NagelRosaySteinWainger1989}, to prove that $\mathbb{S}:\Gamma_\alpha(E)\rightarrow\Gamma_\alpha({\rm b}\Omega)$ for all $0<\alpha<{+\infty}$, it suffices to prove the result for non-integer $\alpha$. 
Also, we may assume that $\delta_0\geq 1$.
By the same observations as in part (a), it suffices to prove the result for operators which satisfy the same estimates as $\mathbb{N}$, and to restrict our attention to $0<\alpha<1$.

Choose a bump function $\eta$ supported in $B_d(\mb{z}_0,5\delta_0)$, with $\eta\equiv 1$ on $B_d(\mb{z}_0,4\delta_0)$ and $|W^\beta \eta|\lesssim \delta_0^{-|\beta|},$ $|\beta|\leq 2$.
Let $0<\alpha<1$ and $f\in \Gamma_\alpha(E)$, and let $\mathbb{T}$ denote an operator satisfying the same estimates as $\mathbb{N}$.
For $\mb{z}\in {\rm b}\Omega$, define 
$$F(\mb{z}):=\mathbb{T}[f-f(\mb{z})\eta](\mb{z})+f(\mb{z})\mathbb{T}[\eta](\mb{z}).$$
We will show that $\|F\|_{\Gamma_\alpha}\leq C(1+\delta_0^\alpha)\|f\|_{\Gamma_\alpha}.$

First suppose that $d(\mb{z},\mb{z}_0)\leq 2\delta_0$. 
Then
\begin{align*}
	|f(\mb{z})\mathbb{T}[\eta](\mb{z})| & = |f(\mb{z})||\mathbb{T}_{0,1}[Z\eta](\mb{z})| \\
	& \lesssim  \|f\|_\infty \delta_0^{-1} \int_{d(\mb{w},\mb{z}_0)\leq 5\delta_0} \frac{d(\mb{z},\mb{w})}{|B_d(\mb{w},d(\mb{z},\mb{w}))|}dm_w \\
	& \lesssim  \|f\|_\infty \delta_0^{-1} \int_{d(\mb{w},\mb{z})\leq 7\delta_0} \frac{d(\mb{z},\mb{w})}{|B_d(\mb{w},d(\mb{z},\mb{w}))|}dm_w \\
	& \lesssim  \|f\|_{\Gamma_\alpha}.
\end{align*}
Moreover,
\begin{align*}
	|\mathbb{T}[f-f(\mb{z})\eta](\mb{z})| & \lesssim  \int_{d(\mb{w},\mb{z})\leq 7\delta_0} \frac{|f(\mb{w})-f(\mb{z})|}{|B_d(\mb{z},d(\mb{z},\mb{w}))|}dm_w \\
	& \lesssim  \|f\|_{\Gamma_\alpha}\int_{d(\mb{w},\mb{z})\leq 7\delta_0} \frac{d(\mb{z},\mb{w})^\alpha}{|B_d(\mb{z},d(\mb{z},\mb{w}))|} dm_w \\
	& \lesssim  \delta_0^\alpha \|f\|_{\Gamma_\alpha}.
\end{align*}
Therefore $|F(\mb{z})|\lesssim (1+\delta_0^\alpha)\|f\|_{\Gamma_\alpha}$ when $d(\mb{z},\mb{z}_0)\leq 2\delta_0$. 
On the other hand, if $\mb{z}\notin B_d(\mb{z}_0,2\delta_0)$, then $f(\mb{z})\equiv 0$ and the estimate of $\|F\|_\infty$ reduces to 
\begin{align*}
	|F(\mb{z})| & =  |\mathbb{T}[f](\mb{z})|\\
	& \lesssim  \|f\|_{\infty} \int_{d(\mb{w},\mb{z}_0)\leq \delta_0} |B_d(\mb{z},d(\mb{z},\mb{w}))|^{-1} dm_w \\
	& \approx  \|f\|_{\infty} \frac{|B_d(\mb{z}_0,\delta_0)|}{|B_d(\mb{z}_0,d(\mb{z},\mb{z}_0))|} \\
	& \leq  \|f\|_{\infty},
\end{align*}
which proves that $\|F\|_{\infty}\lesssim (1+\delta_0^\alpha)\|f\|_{\Gamma_\alpha}.$

Now, fix $\mb{z},\mb{\zeta}$ and write $d(\mb{z},\mb{\zeta})=\delta$.
Let $\phi$ be a bump function supported in $B_d(\mb{z},3\delta)$ with $\phi\equiv 1$ on $B_d(\mb{z},2\delta)$ and $|W^\beta \phi|\lesssim \delta^{-|\beta|}$ for $|\beta|\leq 2$.

If $\delta\geq 1$, then 
$$|F(\mb{z})-F(\mb{\zeta})|\leq 2\|F\|_{\infty}\lesssim (1+\delta_0^\alpha)\|f\|_{\Gamma_\alpha}\leq (1+\delta_0^\alpha)\delta^\alpha\|f\|_{\Gamma_\alpha}.$$
If $\delta\leq 1$, then we have
\begin{align*}
	F(\mb{z})-F(\mb{\zeta})& =  (f(\mb{z})-f(\mb{\zeta}))\mathbb{T}[\eta](\mb{z}) + f(\mb{\zeta})[\mathbb{T}\mathbb{R}^\ast[Z\eta](\mb{z})-\mathbb{T}\mathbb{R}^\ast[Z\eta](\mb{\zeta})] \\
	&  \quad + (f(\mb{z})-f(\mb{\zeta}))\mathbb{T}[\eta(1-\phi)](\mb{z}) \\
	&  \quad + \int_{\b\Omega} [[\mathbb{T}](\mb{z},\mb{w})-[\mathbb{T}](\mb{\zeta},\mb{w})] \\
	&  \qquad\qquad\qquad\qquad\times(f(\mb{w})-f(\mb{\zeta}))\eta(\mb{w})(1-\phi(\mb{w}))dm_w \\
	&  \quad + \mathbb{T}[(f-f(\mb{z}))\phi](\mb{z})-\mathbb{T}[(f-f(\mb{\zeta}))\phi](\mb{\zeta}) \\
	& =: I_1+I_2+I_3+I_4+I_5+I_6.
\end{align*}

Estimating as above (and using $\mathbb{T}=\mathbb{T}_{0,1}\circ Z$ in $I_3$), we have
$$|I_1|+|I_3|+|I_4|+|I_5|+|I_6|\lesssim \delta^\alpha \|f\|_{\Gamma_\alpha},\qquad |I_2|\lesssim \delta\|f\|_{\infty}\leq \delta^\alpha \|f\|_{\Gamma_\alpha}.$$
This concludes the proof of Theorem \ref{thm:szego-mapping}.

\section{Proof of Corollary \ref{cor:szego-growth}}
\label{sec:szego-growth}

In this section, we prove Corollary \ref{cor:szego-growth}.
Because the estimates follow immediately from Theorem \ref{thm:szego-decomp}, it suffices to prove the sharpness claim.
We will do this by inspecting a single example.

Consider the tube domain $\Omega=\{\mb{z}\in \mathbb{C}^2\ :\ \Im(z_2)>b(\Re(z))\}$, where $b:\mathbb{R}\to [0,{+\infty})$ is a convex function with
\begin{itemize}
	\item $b(0)=b'(0)=0$, 
	\item $b''(x) = e^{x-n}$ in a neighborhood of $x=n$, for all $n\in \mathbb{Z}$, and
	\item $b''(x)\approx 1,$ uniformly in $x\in \mathbb{R}$.
\end{itemize}

Fix $\mb{z},\mb{w}\in \b\Omega$, and recall from Section \ref{sec:examples} that we have
\begin{equation}\label{szegoformula2}
	[\mathbb{S}^\epsilon](\mb{z},\mb{w}) = \frac{1}{4\pi^2} \displaystyle\int_0^{+\infty} \displaystyle\int_\mathbb{R} \displaystyle\frac{e^{i\tau(z_2+i\epsilon-\bar{w}_2)+\eta(z+\bar{w})}}{\displaystyle\int_\mathbb{R} e^{2[\eta \theta - \tau b(\theta)]}d\theta}d\eta d\tau.
\end{equation}
Using the formulas $\bar{Z}[\mathbb{S}^\epsilon]\equiv 0$, $$[\bar{Z},Z]=\frac{i}{2} b''(\Re(z))(\partial_{z_2}+\partial_{\bar{z}_2}),$$
and
$$[\bar{Z},b^{(j)}(\Re(z))(\partial_{z_2}+\partial_{\bar{z}_2})]= \frac{1}{2}b^{(j+1)}(\Re(z))(\partial_{z_2}+\partial_{\bar{z}_2}),$$
for $k\geq 1$ we have
\begin{align*}
	\bar{Z}^k Z[\mathbb{S}^\epsilon](\mb{z},\mb{w}) & = \frac{i}{2}\bar{Z}^{k-1}\Big[b''(\Re(z))(\partial_{z_2}+\partial_{\bar{z}_2})[\mathbb{S}^\epsilon](\mb{z},\mb{w})\Big] \\
	& =  \frac{i}{2^{k}} b^{(k+1)}(\Re(z))(\partial_{z_2}+\partial_{\bar{z}_2})[\mathbb{S}^\epsilon](\mb{z},\mb{w}) \\
	& =  \frac{i}{2^{k+2}\pi^2} b^{(k+1)}(\Re(z)) \displaystyle\int_0^{+\infty} \displaystyle\int_\mathbb{R} \displaystyle\frac{i\tau e^{i\tau(z_2+i\epsilon-\bar{w}_2)+\eta(z+\bar{w})}}{\displaystyle\int_\mathbb{R} e^{2[\eta \theta - \tau b(\theta)]}d\theta}d\eta d\tau,
\end{align*}
so that
\begin{equation}\label{eq:szego-derivs}
	-2^{k+2}\pi^2\bar{Z}^k Z[\mathbb{S}^\epsilon](\mb{z},\mb{w})  =  b^{(k+1)}(\Re(z)) \displaystyle\int_0^{+\infty} \displaystyle\int_\mathbb{R} \displaystyle\frac{\tau e^{i\tau(z_2+i\epsilon-\bar{w}_2)+\eta(z+\bar{w})}}{\displaystyle\int_\mathbb{R} e^{2[\eta \theta - \tau b(\theta)]}d\theta}d\eta d\tau.
	\end{equation}	
For $n\in \mathbb{Z}$ we let $\mb{z}_n=(n+i0,0+ib(n))$, so that
\begin{equation}\label{eq:szego-simplified}
	-4\pi^2\bar{Z}^k Z[\mathbb{S}^\epsilon](\mb{z}_n,\mb{z}_{-n}) = \int_0^{+\infty} \tau e^{-\tau(b(n)+b(-n)+\epsilon)}\int_\mathbb{R} \Big( \int_\mathbb{R} e^{2[\eta \theta - \tau b(\theta)]}d\theta\Big)^{-1}d\eta d\tau.
\end{equation}

We begin our analysis of (\ref{eq:szego-simplified}) with a basic fact about convex functions.
\begin{proposition} Let $\phi:\mathbb{R}\to \mathbb{R}$ be convex, and suppose that $\phi(\theta)$ achieves its minimum value at $\theta_0$. If $L=\{ \theta\ :\ \phi(\theta)\leq \phi(\theta_0)+1\}$, then
	\begin{equation}\label{eq:convexexp}
		\int_\mathbb{R} e^{-\phi(\theta)} d\theta \approx |L|e^{-\phi(\theta_0)}.
	\end{equation}
\end{proposition}
\begin{proof}
	To avoid trivialities, assume that $|L|<{+\infty}$.
	
	Making the change of variable $\theta\mapsto \theta+\theta_0$, the left-hand-side of equation (\ref{eq:convexexp}) becomes
	\begin{equation}\label{eq:convexexpinter}
		\displaystyle\int_\mathbb{R} e^{-\phi(\theta)}d\theta = e^{-\phi(\theta_0)}\int_\mathbb{R} e^{-(\phi(\theta+\theta_0)-\phi(\theta_0))}d\theta.
	\end{equation}
	
	Write $L=[\theta_0-\theta_-,\theta_0+\theta_+]$ and 
	split the integral on the right-hand-side of (\ref{eq:convexexpinter}) above as $$\int_\mathbb{R} = \int_{[-\theta_-,\theta_+]} +\int_{\theta\geq \theta_+}+ \int_{\theta\leq -\theta_-}=I_1+I_2+I_3.$$ 
	A simple size estimate on $I_1$ yields $I_1\approx |L|$.
	
	We turn now to $I_2$ and $I_3$. Note that because  $\phi(\theta+\theta_0)$ is convex, the function $\theta^{-1}(\phi(\theta+\theta_0)-\phi(\theta_0))$ is increasing on $\theta>0$, and therefore for $\theta\geq \theta_+$ we have
	$$\frac{\phi(\theta+\theta_0)-\phi(\theta_0)}{\theta} \geq \frac{\phi(\theta_+ +\theta_0)-\phi(\theta_0)}{\theta_+} = \frac{1}{\theta_+},$$
	or rather
	$$\phi(\theta+\theta_0)-\phi(\theta_0)\geq \frac{1}{\theta_+}\theta,\quad \theta\geq \theta_+.$$
	Similarly, $\phi(\theta+\theta_0)-\phi(\theta_0)\geq -\frac{1}{\theta_-}\theta$ for $\theta\leq -\theta_-$. Combining these yields
	\begin{align*}
		I_2+I_3 & \leq \int_{\theta\geq \theta_+} e^{-\theta_+^{-1}\theta }d\theta + \int_{\theta\leq -\theta_-} e^{\theta_-^{-1}\theta }d\theta \\
		& =  \theta_+ \int_{x\geq 1} e^{-x}dx + \theta_- \int_{x\leq -1} e^{x}dx \\
		& \leq  \theta_+ + \theta_- \\
		& =  |L|,
	\end{align*}
	which is the desired result.
\end{proof}

Now, let $\phi_{\tau,\eta}(\theta)=2[\tau b(\theta)-\eta \theta]$, and define $\theta_0(\tau,\eta)$ and $L(\tau,\eta)$ for $\phi_{\tau,\eta}$ as in the previous proposition.
We then have

\begin{lemma}\label{lem:convexinteest} $|L(\tau,\eta)|\approx \tau^{-\frac{1}{2}}$ and $-\phi_{\tau,\eta}(\theta_0(\tau,\eta))\approx \frac{\eta^2}{\tau}$.
\end{lemma}
\begin{proof}
	Writing $L=[\theta_0-\theta_-,\theta_0+\theta_+]$, we have
	$$1=\int_{\theta_0}^{\theta_0+\theta_+} \int_{\theta_0}^\theta \phi_{\tau,\eta}''(r)drd\theta \approx \tau \int_{\theta_0}^{\theta_0+\theta_+} \int_{\theta_0}^\theta drd\theta=\frac{1}{2}\tau \theta^2_{+},$$
	which shows that $\theta_+\approx \tau^{-\frac{1}{2}}$. Similarly, $\theta_-\approx \tau^{-\frac{1}{2}}$, showing that $|L|\approx \tau^{-\frac{1}{2}}$ as claimed.
	
	For the second inequality, note first that $\theta_0(\tau,\eta)=(b')^{-1} (\tau^{-1}\eta)$, so that
	\begin{align*}
		\phi_{\tau,\eta}(\theta_0(\tau,\eta)) & = -2\tau[\tau^{-1}\eta (b')^{-1}(\tau^{-1}\eta)-b((b')^{-1}(\tau^{-1}\eta))] \\
		& =  -2\tau \int_0^{(b')^{-1}(\tau^{-1}\eta)} sb''(s)ds \\
		& \approx  -\tau \int_0^{(b')^{-1}(\tau^{-1}\eta)} sds \\
		& = -\tau \Big((b')^{-1}(\tau^{-1}\eta)\Big)^2.
	\end{align*}
	Because $b''\approx 1$, $b'(\theta)\approx \theta$, and therefore $(b')^{-1}(\theta)\approx \theta$. 
	This proves the second claim.
\end{proof}

By the results in Section \ref{sec:geom}, we have
$$d(\mb{z}_n,\mb{z}_{-n})\approx n,\quad \Lambda(\mb{z}_n,\delta)\approx \delta^2.$$
This allows us to estimate (\ref{eq:szego-simplified}) as
\begin{align*}
	-2^{k+2}\pi^2 & \bar{Z}^k Z[\mathbb{S}^\epsilon](\mb{z}_{n},\mb{z}_{-n}) \\
	& =  \int_0^{+\infty} \tau e^{-\tau(b(n)+b(-n)+\epsilon)}\int_\mathbb{R} \Big( \int_\mathbb{R} e^{2[\eta \theta - \tau b(\theta)]}d\theta\Big)^{-1}d\eta d\tau \\
	& \approx  \int_0^{+\infty} \tau e^{-\tau(b(n)+b(-n)+\epsilon)}\int_\mathbb{R} |L(\tau,\eta)|^{-1}e^{\phi_{\tau,\eta}(\theta_0(\tau,\eta))}d\eta d\tau \\
	& \geq  \int_0^{+\infty} \tau^{\frac{3}{2}} e^{-\tau(b(n)+b(-n)+\epsilon)}\int_\mathbb{R} e^{-c\tau^{-1}\eta^2}d\eta d\tau \\
	& \approx  \int_0^{+\infty} \tau^{2} e^{-\tau(b(n)+b(-n)+\epsilon)}d\tau \\
	& \approx  |b(n)+b(-n)+\epsilon|^{-3} \\
	& \approx  n^{-6} \\
	& \approx  \frac{d(\mb{z}_{n},\mb{z}_{-n})^{-2}}{|B_d(\mb{z}_{n},d(\mb{z}_{n},\mb{z}_{-n}))|},
\end{align*}
uniformly for $k\geq 1$, $n\in\mathbb{Z}$, and $\epsilon>0$ small. 
This concludes the proof of the sharpness claim in Corollary \ref{cor:szego-growth}, and therefore the full proof of Corollary \ref{cor:szego-growth}.
\bigskip

\section{Proof of Corollary \ref{cor:szego-cancellation}}

\label{sec:szego-cancellation}

The proof of Corollary \ref{cor:szego-cancellation} requires us to integrate by parts in the integral 
$$\displaystyle\int_{{\rm b}\Omega} [T^N Z^\alpha\mathbb{S}](\mb{z},\mb{w}) \phi(\mb{w})dm(w,\Re(w_2)).$$
To do this, we first adapt a result of Nagel and Stein to decompose the Szeg\H{o} kernel into two parts: one which is a nice function and one which is a high derivative in $\Re(z_2)$ of a nice function. 
More precisely we have the following result, proved in exactly the same way as Lemma 7.21 of \cite{NagelStein2006}.

\begin{lemma}\label{lem:decompofszego}
	Let $\mb{0}=(0,0+i0),\ \mb{w}=(w,s+iP_1(w))$. Fix $\delta>0$ and $N\geq 0,$ and let $\alpha,\beta$ be multi-indices with $|\alpha|,|\beta|\leq 2$. For $M>2+N+\frac{1}{2}(|\alpha|+|\beta|)$, there is a constant $C=C(|\alpha|,|\beta|,M,N)$ and functions $F^{(M)}(\mb{0},\mb{w})$, $G^{(M)}(\mb{0},\mb{w})$ such that
	\[ [T^N Z_{\mb{z}}^\alpha (Z_{\mb{w}}^\beta)^\ast \mathbb{S}](\mb{0},\mb{w})=F^{(M)}(\mb{0},\mb{w})+(\delta T)^M G^{(M)}(\mb{0},\mb{w}),\]
	where
	\begin{align*}
		|F^{(M)}(\mb{0},\mb{w})|&\leq C [d(\mb{0},\mb{w})+\mu(\mb{0},\delta)]^{-2-|\alpha|-|\beta|}[\Lambda(\mb{0},d(\mb{0},\mb{w}))+\delta]^{-1-N},\\
		|G^{(M)}(\mb{0},\mb{w})|&\leq C [d(\mb{0},\mb{w})+\mu(\mb{0},\delta)]^{-2-|\alpha|-|\beta|}[\Lambda(\mb{0},d(\mb{0},\mb{w}))+\delta]^{-1-N}.
	\end{align*}
\end{lemma}
\medskip

\begin{remark}{\rm It is in the proof of this lemma that we need to use the formula $[\mathbb{B}](\mb{z},\mb{w})=2i\frac{\partial}{\partial \bar{w}_2}[\mathbb{S}](\mb{z},\mb{w})$, which is proved in Appendix \ref{sec:sz}.
		Here $\mathbb{B}:L^2(\Omega)\to L^2(\Omega)\cap \mathcal{O}(\Omega)$ is the Bergman projection, which is the orthogonal projection of $L^2(\Omega)$ on the (closed) subspace of square-integrable holomorphic functions on $\Omega$. 
	}\end{remark}

	\begin{proof}[Proof of Corollary \ref{cor:szego-cancellation}]
		By making a biholomorphic change of variables as in Section \ref{sec:normalization-biholomorphic}, we may assume that $\mb{z}=\mb{0}$ and that $P(w)=P^{\mb{z},2}(w)$.
		
		Apply Lemma \ref{lem:decompofszego} with $\delta=\Lambda(\mb{0},\delta_0)$ and large enough $M$ to get
		\begin{align*}
			\displaystyle \sup_{\mb{\zeta}\in B_d(\mb{0},\delta_0)}&  |(T^N Z^\alpha \mathbb{S})[\phi](\mb{\zeta})| \\
			=&  \displaystyle\sup_{\mb{\zeta}\in B_d(\mb{0},\delta_0)} \left| \displaystyle\int_{{\rm b}\Omega} [T^N Z^\alpha \mathbb{S}](\mb{\zeta},\mb{w})\phi(\mb{w})dm(w,\Re(w_2))\right| \\
			=&  \displaystyle\sup_{\mb{\zeta}\in B_d(\mb{0},\delta_0)} \left| \displaystyle\int_{{\rm b}\Omega} F^{(M)}(\mb{\zeta},\mb{w})\phi(\mb{w})\right. \\
			&  \quad\qquad\qquad\qquad\left. + (-1)^MG^{(M)}(\mb{\zeta},\mb{w})(\Lambda(\mb{0},\delta_0)T)^M\phi(\mb{w})dm(w,\Re(w_2))\right| \\
			\lesssim &  \delta_0^{-2-|\alpha|}\Lambda(\mb{0},\delta_0)^{-1-N} |B_d(\mb{0},\delta_0)|\cdot \|\phi\|_\infty \\
			&  \qquad\qquad  + \delta_0^{-2-|\alpha|}\Lambda(\mb{0},\delta_0)^{-1-N} |B_d(\mb{0},\delta_0)|\|(\Lambda(\mb{0},\delta_0)T)^M \phi\|_\infty \\
			\lesssim &  \delta_0^{-|\alpha|}\Lambda(\mb{0},\delta_0)^{-N}(\|\phi\|_\infty + \|(\Lambda(\mb{0},\delta_0)T)^M \phi\|_\infty),
		\end{align*}
		as desired.
	\end{proof}

\begin{appendix}
	
	\section{The Szeg\H{o} Kernel}
	\label{sec:sz}
	
	Let $\Omega = \lb \mb{z}\in\mathbb{C}^2\ :\ \Im(z_2)>P(z)\rb$ be a pseudoconvex model domain, where $P:\mathbb{C}\rightarrow \mathbb{R}$ is smooth, subharmonic, and non-harmonic.
	
	In this section we give a precise definition of the Bergman and Szeg\H{o} kernels $[\mathbb{B}]$ and $[\mathbb{S}]$ for $\Omega$, respectively, and supply a proof of the following proposition which has been widely used (going back at least to \cite{NagelRosaySteinWainger1988}).
	
	\begin{proposition}\label{prop:szegobergman}
		Let $\Omega=\lb \mb{z}\in\mathbb{C}^2\ :\ \Im(z_2)>P(z)\rb$, where $P:\mathbb{C}\rightarrow \mathbb{R}$ is smooth and subharmonic.
		If we equip $\Omega$ and ${\rm b}\Omega$ with Lebesgue measure (as in the introduction), then we have
		$$[\mathbb{B}](\mb{z},\mb{w})=2i\frac{\partial}{\partial\bar{w}_2}[\mathbb{S}](\mb{z},\mb{w}),\quad \mbox{for}\ \ \mb{z}\in \Omega,\ \mb{w}\in{\rm b}\Omega.$$
	\end{proposition}
	
	Although the proof of Proposition \ref{prop:szegobergman} is not difficult, the technical issues surrounding it have, to the author's knowledge, not been treated in the literature.
	The goal of this appendix is to give a complete account of these non-trivial issues.
	
	Let $\mathcal{O}(\Omega)$ denote the space of holomorphic functions in $\Omega$.  
	For a function $F$ on $\Omega$ and $\epsilon>0$, define $F_\epsilon : \mathbb{C}\times\mathbb{R}\rightarrow \mathbb{C}$ by
	$$F_\epsilon (z,t)=F(z,t+i(P(z)+\epsilon)).$$
	We then define the Hardy space
	$$\mathcal{H}^2(\Omega)=\lb F\in\mathcal{O}(\Omega)\ |\ \displaystyle\sup_{\epsilon>0} \displaystyle\int_{\mathbb{C}\times\mathbb{R}} |F_\epsilon (z,t)|^2dm = \|F\|^2_{\mathcal{H}^2(\Omega)}<{+\infty}\rb.$$
	Here and below, we use $dm$ to denote both Lebesgue measure on $\Omega$ and Lebesgue measure on $\mathbb{C}\times\mathbb{R}$.
	
	In an appropriate sense, the space $\mathcal{H}^2(\Omega)$ consists of holomorphic functions with boundary values in $L^2({\rm b}\Omega)$.
	Halfpap, Nagel, and Wainger give the elementary properties of such spaces in \cite{HalfpapNagelWainger2010}.
	In particular, they prove the following result, which is valid for the domains above.
	
	\begin{proposition}[\cite{HalfpapNagelWainger2010}, Proposition 2.4]\label{prop:kernels-hnw24}
		Let $F\in \mathcal{H}^2(\Omega)$. Then there exists $F^b\in L^2(\mathbb{C}\times\mathbb{R})$ such that
		\begin{itemize}
			\item[(a)] $F^b(z,t)=\displaystyle\lim_{\epsilon\rightarrow 0^+} F_\epsilon (z,t)$ for a.e. $(z,t)\in\mathbb{C}\times\mathbb{R}$;
			\item[(b)] $\displaystyle\lim_{\epsilon\rightarrow 0^+} \|F_\epsilon - F^b\|_{L^2(\mathbb{C}\times\mathbb{R})}=0$, and $\|F^b\|_{L^2(\mathbb{C}\times\mathbb{R})}=\|F\|_{\mathcal{H}^2(\Omega)}$;
			\item[(c)] $F^b$ satisfies $\bar{Z} F^b=(\partial_{\bar{z}}-i P_{\bar{z}}(z)\partial_t)F^b=0$ in the sense of distributions;
			\item[(d)] For any compact set $K\Subset \Omega$, there exists $C(K)$ such that
			$$\sup_{z\in K} |F(z)|\leq C(K)\|F\|_{\mathcal{H}^2(\Omega)}.$$
		\end{itemize}
	\end{proposition}
	
	To study functions on the boundary of $\Omega$, we use the fact that $\Omega$ is translation invariant in $\Re(z_2)$.
	To this end, define the partial Fourier transform
	$$\mathcal{F}[f](z,\tau):= \int_{\mathbb{R}} e^{-2\pi i \tau t}f(z,t)dt.$$
	For our purposes, we need the following alteration of Proposition 2.5 of \cite{HalfpapNagelWainger2010}.
	
	\begin{proposition}\label{prop:kernels-hnw25alt}
		Let $P:\mathbb{C}\rightarrow\mathbb{R}$ be smooth, subharmonic, and non-harmonic.
		Also assume that $f\in L^2(\mathbb{C}\times\mathbb{R})$.
		Then the following hold:
		\begin{itemize}
			\item[(a)] The function $f$ satisfies
			\begin{equation}\label{hnw25alt-pde} \partial_{\bar{z}}f(z,t)-iP_{\bar{z}}(z)\partial_{t}f(z,t)=0\end{equation}
			on $\mathbb{C}\times\mathbb{R}$ as a tempered distribution if and only if the partial Fourier transform (in $t$) $\mathcal{F}[f]=\hat{f}(z,\tau)$ satisfies
			\begin{equation}\label{hnw25alt-cond}\partial_{\bar{z}}\Big( e^{2\pi \tau P(z)}\hat{f}(z,\tau)\Big)=0\end{equation}
			on $\mathbb{C}\times\mathbb{R}$ as a tempered distribution.
			\item[(b)] If $f$ satisfies the PDE (\ref{hnw25alt-pde}), then $\hat{f}(z,\tau)=0$ almost everywhere when $\tau<0$. In particular, if we set $h_s(z,\tau)=e^{-2\pi \tau s}\hat{f}(z,\tau)$, then $h_s\in L^2(\mathbb{R}^3)$ for $s\geq 0$.
			\item[(c)] If $f$ satisfies the PDE (\ref{hnw25alt-pde}), and if
			$$F(\mb{z})=F(z,t+iP(z)+is)=\mathcal{F}^{-1}[h_s](z,t),$$
			then $F\in\mathcal{H}^2(\Omega)$ and $F^b=f$.
		\end{itemize}
	\end{proposition}
	
	\begin{proof}
		The proof of (a) and (c) is very similar to that of Proposition 2.5 of \cite{HalfpapNagelWainger2010}.
		For (b), we refer to Lemma 5.2.1 of \cite{NagelStein2006}.
	\end{proof}

	We see therefore that the space of boundary values $$B({\rm b}\Omega)=\lb f\in L^2({\rm b}\Omega)\ :\ f(z,t+iP(z))=F^b(z,t)\ \mbox{for some}\ F\in \mathcal{H}^2(\Omega)\rb$$ is exactly the space of functions $f\in L^2(\b\Omega)$ that satisfy $\bar{Z}f=0$ as distributions.
	From Proposition \ref{prop:kernels-hnw25alt}, $B(\b\Omega)$ is closed in $L^2(\b\Omega)$.
	The orthogonal projection $\mathbb{S}:L^2({\rm b}\Omega)\rightarrow B({\rm b}\Omega)$ is called the Szeg\H{o} projection, and is given by the formula
	$$\mathbb{S}[f](\mb{z})=\displaystyle\int_{{\rm b}\Omega} [\mathbb{S}](\mb{z},\mb{w})f(\mb{w})dm(w,\Re(w_2)),$$ 
	where $[\mathbb{S}](\mb{z},\mb{w})$ is the Szeg\H{o} kernel. 
	
	By writing $[\mathbb{S}](\mb{z},\mb{w})=\sum \phi_j(\mb{z})\overline{\phi_j(\mb{w})}$ for some orthonormal basis $\{\phi_j\}$ of $\mathcal{H}^2(\Omega)$, we see that $[\mathbb{S}](\mb{z},\mb{w})$ is actually defined on $(\bar{\Omega}\times\bar{\Omega})\backslash\Sigma$, where $\Sigma\subset {\rm b}\Omega\times{\rm b}\Omega$.
	Theorem \ref{thm:szego-decomp} shows that $\Sigma\subseteq \{ (\mb{z},\mb{z})\ |\ \mb{z}\in {\rm b}\Omega\}$ for UFT domains.  
	Indeed, for $f\in L^2({\rm b}\Omega)$, we may write
	$$\mathbb{S}[f](\mb{z})=\lim_{\epsilon\rightarrow 0}\displaystyle\int_{{\rm b}\Omega} [\mathbb{S}](\mb{z},(w,w_2+i\epsilon))f(\mb{w})dm(w,\Re(w_2)),$$
	where the convergence is in $L^2({\rm b}\Omega)$.
	
	For functions $F\in L^2(\Omega)\cap \mathcal{O}(\Omega)$, we can also write
	$$F(\mb{z})=\int_{\Omega} [\mathbb{B}](\mb{z},\mb{w})F(\mb{w})dm(\mb{w}),$$
	where $[\mathbb{B}](\mb{z},\mb{w})$ is the Bergman kernel.
	One fruitful approach to studying the Szeg\H{o} kernel is to express it in terms of the Bergman kernel via Proposition \ref{prop:szegobergman}.
	This is slightly complicated by part \textit{(iii)} of the following proposition.
	
	\begin{proposition}\label{prop:kernels-clarify} If $P$ is as above, then the following statements hold.
		\begin{itemize}
			\item[(i)] $L^2(\Omega)\cap \mathcal{H}^2(\Omega)$ is dense in $L^2(\Omega)\cap \mathcal{O}(\Omega)$, and for $F\in L^2(\Omega)\cap \mathcal{O}(\Omega)$ and $\epsilon_0>0$ we have \begin{equation}\label{eq:h2embed} \|F_{\epsilon_0}\|_{\mathcal{H}^2(\Omega)}\leq \epsilon_0^{-\frac{1}{2}}\|F\|_{L^2(\Omega)}.\end{equation}
			\item[(ii)] If $F\in\mathcal{H}^2(\Omega)$ and $\epsilon_0>0$, then $\displaystyle\|(\partial_{z_{2}}F)_{\epsilon_0}\|_{L^2(\Omega)}\leq \frac{\epsilon_0^{-\frac{1}{2}}}{2e^{\frac{1}{2}}}\|F\|_{\mathcal{H}^2(\Omega)}$.
			\item[(iii)] If $P(z)=b(\Re(z))$, then there are functions $F\in\mathcal{H}^2(\Omega)$ with $F\notin L^2(\Omega)$.
		\end{itemize}
	\end{proposition}
	\begin{proof}
		\pagestyle{plain}
		For (i), we first note that for $F\in L^2(\Omega)$, $F_\epsilon\rightarrow F$ in $L^2(\Omega)$ as $\epsilon\to 0$.
		To complete the proof, we will show that $F_{\epsilon_0}\in \mathcal{H}^2(\Omega)$ for any $\epsilon_0>0$, and establish equation (\ref{eq:h2embed}).
		
		Because $F\in L^2(\Omega)$, $$H(\epsilon):=\int_{\mathbb{C}\times\mathbb{R}} |F_\epsilon(z,t+iP(z))|^2dm\in L^1(0,{+\infty}),$$ and therefore $H(\epsilon)<+\infty$ for almost every $\epsilon>0$.
		For $\epsilon\in \mathbb{C}_+=\{ x+iy\in \mathbb{C}\ :\ x>0\}$ and $n\in \mathbb{N}$, define
		$$H_n(\epsilon)=\int_{|z|+|t|\leq n}|F_\epsilon(z,t+iP(z))|^2dm(z,t).$$
		We note that, because $F$ is holomorphic, 
		$$\Delta H_n(\epsilon)=4\frac{\partial^2 H_n}{\partial \epsilon \partial \bar{\epsilon}}(\epsilon) = 4\int_{|z|+|t|\leq n} \Big|\frac{\partial F}{\partial z_2}(z,t+iP(z)+i\epsilon)\Big|^2 dm(z,t),$$
		and therefore $H_n(\epsilon)$ is subharmonic.
		
		Moreover, we claim that $H_n(\epsilon)$ is locally pointwise bounded (uniformly in $n$).
		Indeed, for $\epsilon\in \mathbb{C}_+$, let $\lambda=\Re(\epsilon)$ and write, for $|\eta|\leq \frac{\lambda}{10}$ and $R\leq \frac{\lambda}{2}$,
		$$H_n(\epsilon+\eta) \leq \frac{1}{2\pi R}\int_{|\zeta|=R} H_n(\epsilon+\eta+\zeta)d\sigma(\zeta),$$
		so that
		\begin{align*}
			H_n(\epsilon+\eta) & \lesssim  \frac{1}{\lambda} \int_{\frac{\lambda}{10}}^{\frac{\lambda}{5}} \frac{1}{R}\int_{|\zeta|=R} H_n(\epsilon+\eta+\zeta)d\sigma(\zeta) dR \\
			& \approx  \frac{1}{\lambda^2} \int_{\frac{\lambda}{10}}^{\frac{\lambda}{5}} \int_{|\zeta|=R} H_n(\epsilon+\eta+\zeta)d\sigma(\zeta) dR \\
			& \lesssim  \frac{1}{\lambda^2} \int_{|\epsilon-\hat{\epsilon}|_\infty \leq \frac{4\lambda}{5}} H_n(\hat{\epsilon})dm(\hat{\epsilon}) \\
			& \leq  \frac{1}{\lambda^2} \int_{|\epsilon-\hat{\epsilon}|_\infty \leq \frac{4\lambda}{5}} H(\hat{\epsilon})dm(\hat{\epsilon}) \\
			& \lesssim  \frac{1}{\lambda} \int_{|\lambda-x| \leq \frac{4\lambda}{5}} H(x)dx \\
			& \leq  \frac{1}{\lambda} \|F\|^2_{L^2(\Omega)},
		\end{align*}
		proving the claim.
		
		Because $H_n(\epsilon)$ is increasing (in $n$) and locally bounded, Remark 4.4.43 of \cite{BerensteinGay1991} implies that the pointwise limit $H(\epsilon)$ is a.e. equal to the subharmonic function
		$$H^{\ast}(\epsilon)=\limsup_{\eta\rightarrow \epsilon} H(\eta).$$
		
		Because $H(x+iy)$ does not depend on $y$, neither does $H^{\ast}(x+iy)$, so that $\partial^2_x H^{\ast}(x)\geq 0$ (in the distributional sense) on $0<x<{+\infty}$.
		After one notes that $H^{\ast}(x)$ is locally bounded, a mollification argument (together with the fact that the pointwise limit of convex functions is convex) implies that $H^{\ast}(x)$ is actually a convex function.
		
		Therefore, the non-negative function $H(\epsilon)$ (for $\epsilon>0$) is dominated by, and a.e. equal to, the convex function $H^{\ast}(\epsilon)$.
		In particular, this shows that $H(\epsilon)<{+\infty}$ for all $\epsilon>0$, and that $H(\epsilon)$ is dominated by and equal a.e. to a non-increasing convex function.
		Thus, for $\epsilon_0>0$, $F_{\epsilon_0}\in \mathcal{H}^2(\Omega)$ as desired.
		
		Moreover, by Proposition \ref{prop:kernels-hnw24},
		\begin{align*}
			\|F_{\epsilon_0}\|^2_{\mathcal{H}^2(\Omega)}& = \int_{\mathbb{C}\times\mathbb{R}} |F(z,t+iP(z)+i\epsilon_0)|^2dm(z,t) \\
			& =  H(\epsilon_0) \\
			& \leq  \epsilon_0^{-1} \|F\|^2_{L^2(\Omega)},
		\end{align*}
		where in the last line we used the elementary fact that if $H:(0,+\infty)\rightarrow [0,+\infty)$ is a non-negative, non-increasing function with $\int_0^{+\infty} H(s)ds=A$, then $H(\epsilon_0)\leq A\epsilon_0^{-1}$ for any $\epsilon_0>0$.
		This completes the proof of (i).
		
		For (ii), using the characterization in Proposition \ref{prop:kernels-hnw25alt} we can write
		\begin{align*}
			\|(\partial_{z_2}F)_\epsilon\|^2_{L^2(\Omega)} & =  \int_0^{+\infty} \int_{\mathbb{R}} \int_{\mathbb{C}} |\mathcal{F}^{-1}[2\pi i \bullet h_{s+\epsilon}](z,t)|^2 dm(z)dtds \\
			& =  \int_0^{+\infty} \int_{\mathbb{C}} \int_0^{+\infty} 4\pi^2 \tau^2 e^{-4\pi \tau (s+\epsilon)}|\hat{f}(z,\tau)|^2dsdm(z)d\tau\\
			& =  \int_0^{+\infty} \int_{\mathbb{C}} \pi \tau e^{-4\pi\tau\epsilon} |\hat{f}(z,\tau)|^2 dm(z)d\tau \\
			& \leq \frac{1}{4 e\epsilon}\|\hat{f}\|^2_{L^2(\mathbb{C}\times\mathbb{R})} \\
			& =  \frac{1}{4 e \epsilon}\|f\|^2_{L^2(\mathbb{C}\times\mathbb{R})} = \frac{\epsilon^{-1}}{4e} \|F\|^2_{\mathcal{H}^2(\Omega)}.
		\end{align*}

		For (iii), the proof of Proposition 2.5 of \cite{HalfpapNagelWainger2010} yields
		\begin{equation}\label{eq:szego-plancherel}
		\int_{\mathbb{R}^3} |F(x+iy,t+ib(x)+is)|^2dm(x,y,t) = \int_{\mathbb{R}^2}\int_0^{+\infty} e^{-4\pi \tau s} |\tilde{f}(x,\eta,\tau)|^2d\tau dxd\eta,\end{equation}
		where $\displaystyle\tilde{f}(x,\eta,\tau)=\int_{\mathbb{R}^2} e^{-2\pi i (y\eta + t\tau)}f(x+iy,t)dydt.$
		Moreover, part (a) of Proposition 2.5 of \cite{HalfpapNagelWainger2010} gives a 1-1 correspondence between such functions $f$ and functions $g(\eta,\tau)$ such that
		$$\tilde{f}(x,\eta,\tau)=e^{2\pi (x\eta - \tau b(x))}g(\eta,\tau)\in L^2(\mathbb{C}\times (0,+\infty)).$$
		
		Integrating (\ref{eq:szego-plancherel}) in $s$ over $[0,{+\infty})$ and applying Fubini-Tonelli gives
		\begin{align*}
			\|F\|^2_{L^2(\Omega)}&= \int_{\mathbb{R}^2}\int_0^{+\infty} \frac{|\hat{f}(x,\eta,\tau)|^2}{4\pi \tau} d\tau dxd\eta \\
			& =   \int_{\mathbb{R}^2}\int_0^{+\infty} \frac{|g(\eta,\tau)|^2e^{4\pi(\eta x - \tau b(x))}}{4\pi \tau} d\tau dxd\eta \\
			& \geq  \int_{|x|,|\eta|\leq 1}\int_0^1 \frac{|g(\eta,\tau)|^2e^{4\pi(\eta x - \tau b(x))}}{4\pi \tau} d\tau dxd\eta \\
			& \approx  \int_{|\eta|\leq 1}\int_0^1 \frac{|g(\eta,\tau)|^2}{4\pi \tau} d\tau d\eta.
		\end{align*}
		We see that $\|F\|_{L^2(\Omega)}=+\infty$, for example, if $|g(\eta,\tau)|\gtrsim 1$ for $|\eta|\leq 1$, $0<\tau<1$. 
	\end{proof}
	
	\begin{remark}{\rm The phenomenon described in Proposition \ref{prop:kernels-clarify} seems to be a byproduct of the unboundedness of $\Omega$. For a simple example in $\mathbb{C}^1$, let $f(z)=(z+i)^{-1}$ and $\Omega=\{z\ :\ \Im(z)>0\}$. Then $\Omega$ is pseudoconvex and the Lebesgue measure on ${\rm b}\Omega=\{z\ :\ \Im(z)=0\}$ coincides with the standard arclength measure, so we have
			$$\iint_\Omega |f(z)|^2 dm(z) > \int_{\frac{\pi}{4}}^{\frac{3\pi}{4}} \int_{\sqrt{2}}^{+\infty} \frac{1}{r^2}rdrd\theta = +\infty,$$
			while
			$$\int_{\b\Omega} |f(x+iy)|^2dm(x) = \frac{\pi}{y+1}\leq \pi$$
			for all $y\geq 0$.
		}\end{remark}
	
	We are now ready to prove Proposition \ref{prop:szegobergman}.
	
	\begin{proof}[Proof of Proposition \ref{prop:szegobergman}]
		Let $F(\mb{z})\in L^2(\Omega)\cap \mathcal{H}^2(\Omega)$.
		Then
		\begin{align*}
			F(\mb{z}) & = \int_{{\rm b}\Omega} [\mathbb{S}](\mb{z},\mb{w})F^b(w,\Re(w_2)) dm(w,\Re(w_2)) \\
			& =  \lim_{\epsilon\rightarrow 0} \int_{{\rm b}\Omega} [\mathbb{S}](\mb{z},(w,w_2+i\epsilon))F^b(w,\Re(w_2)) dm(w,\Re(w_2))\\
			& =  \lim_{\epsilon\rightarrow 0} \int_{\mathbb{C}}\Bigg[\int_{\Im(w_2)=P(w)} [\mathbb{S}](\mb{z},(w,w_2+i\epsilon))F(\mb{w})dw_2\Bigg] dm(w) \\
			& =  \lim_{\epsilon\rightarrow 0} \int_{\mathbb{C}}\Bigg[\int_{\Im(w_2)>P(w)} \frac{\partial}{\partial\bar{w}_2}\big([\mathbb{S}](\mb{z},(w,w_2+i\epsilon))F(\mb{w})\big)d\bar{w}_2\wedge dw_2\Bigg] dm(w) \\
			& =  \lim_{\epsilon\rightarrow 0} \int_{\Omega} 2i\frac{\partial [\mathbb{S}]}{\partial\bar{w}_2}(\mb{z},(w,w_2+i\epsilon))F(\mb{w}) dm(\mb{w}).
		\end{align*} 
		Moreover, part (i) of Proposition \ref{prop:kernels-clarify} implies that
		\begin{equation}\label{eqend} \int_{\Omega} 2i\frac{\partial [\mathbb{S}]}{\partial\bar{w}_2}(\mb{z},(w,w_2+i\epsilon))G(\mb{w})dm(\mb{w})=0,\quad \forall \epsilon>0\end{equation}
		for any $G\perp L^2(\Omega)\cap \mathcal{O}(\Omega)$.
		
		Hence
		$$2i\frac{\partial [\mathbb{S}]}{\partial \bar{w}_2}(\mb{z},\mb{w}) = [\mathbb{B}](\mb{z},\mb{w}),$$ 
		as desired.
	\end{proof}

	\section{Proofs of Lemma \ref{lem:unifL2bound} and Propositions \ref{prop:continuity} and \ref{prop:link}}
	\label{sec:link}
	
	We first prove Lemma \ref{lem:unifL2bound}.

	\begin{proof}[Proof of \ref{lem:unifL2bound}] 
		Lemmas \ref{lem:metric-approxsmall} and \ref{lem:geom-uftisaq} imply that $$\sigma_\tau(w)=\mu(\mb{w},\tau^{-1}) \lesssim \begin{cases} \tau^{-\frac{1}{2}} & \mbox{if}\ \tau\leq 1,\\ C_1^{-1} \tau^{-\frac{1}{m}} & \mbox{if}\ \tau\geq 1,\end{cases}$$ where $C_1$ is the constant from (H1). This shows that $A(\tau)<+\infty$ for $0<\tau<+\infty$. The fact that $A(\tau)$ is non-increasing follows immediately from its definition as the supremum of a family of non-increasing functions.
		
		The operator bounds for $\mathbb{G}_\tau,\ \mathbb{R}_\tau$, and $\mathbb{R}_\tau^\ast$ follow immediately from the Schur test, coupled with the scaling arguments used in the proof of Lemma \ref{lem:kernels-basicest}. The bound for $\mathbb{S}_\tau$ is due to the fact that $\mathbb{S}_\tau$ is an orthogonal projection.
	\end{proof}
	
	The proof of Proposition \ref{prop:link} requires a preparatory lemma.
	
	\begin{lemma}\label{lem:smoothapprox} For $f\in L^2(\mathbb{C}\times\mathbb{R})$ with $\bar{Z}f=0$ in the sense of tempered distributions, there exist Schwartz functions $\{f_n\}\subset \mathscr{S}(\mathbb{C}\times\mathbb{R})$ such that
		$\hat{f}_n\in C_c^{\infty} (\mathbb{C}\times \mathbb{R})$, $f_n\to f$ in $L^2$, and $A(\tau)\bar{D}_\tau \hat{f}_n(z,\tau)\to 0$ in $L^2(\mathbb{C}\times\mathbb{R})$, where $A(\tau)=\sup\limits_{z\in \mathbb{C}} \sigma_\tau(z)$.
	\end{lemma}
	\begin{proof}
		By Proposition \ref{prop:kernels-hnw25alt}, $\hat{f}(z,\tau)=0$ a.e. on $\{(z,\tau)\ :\ \tau\leq 0\}$. Fix a non-negative $\psi\in C_c^\infty (\mathbb{C}\times \mathbb{R})$ with $\int \psi = 1$, ${\rm supp}\,\psi\subset\{|z|+|\tau|\leq 1\}$, and let $$\psi_\epsilon(z,\tau) = \epsilon^{-3}\psi\Big(\epsilon^{-2}z,\epsilon^{-1}\tau\Big).$$
		
		Choose a smooth, non-negative, non-increasing function $$\chi(t)= \begin{cases} 1\ &\mbox{if}\ t\leq 1,\\ 0\ &\mbox{if}\ t\geq 2,\end{cases}$$
		and with $|\chi'(t)|\leq 2$.
		For $M,N\geq 1$, define
		$$\chi_{M,N}(z,\tau)=\begin{cases} \chi(N^{-1}|z|)\chi(M^{-1}\tau)\chi(M^{-1}\tau^{-1}), & \mbox{if}\ \tau>0,\\ 0, & \mbox{otherwise}.\end{cases}$$
		and note that
		\begin{equation}\label{link:eq-suppchi}
			{\rm supp}\, \chi_{M,N}(z,\tau)\subset\{(z,\tau)\in\mathbb{C}\times\mathbb{R}\ :\ |z|\leq 2N,\ \frac{1}{2M}\leq \tau\leq 2M\}
			\end{equation}
		and
		\begin{equation}\label{link:eq-chi1}
			\chi_{M,N}(z,\tau)\equiv 1\quad \mbox{for}\ |z|\leq N,\ \frac{1}{M}\leq \tau\leq M.
			\end{equation}
		
		Define $\hat{f}_{\epsilon,M,N}(z,\tau)=\psi_\epsilon\ast (\chi_{M,N} \hat{f})(z,\tau).$
		Because $\chi_{M,N} \hat{f}\to \hat{f}$ in $L^2$ as $M,N\to {+\infty}$, and since $\psi_\epsilon\ast g\to  g$ for any $g\in L^2(\mathbb{C}\times \mathbb{R})$ by standard approximate identity results, we may make $\hat{f}_{\epsilon,M,N}\to \hat{f}$ (and therefore $f_{\epsilon,M,N}\to f$) in $L^2(\mathbb{C}\times\mathbb{R})$ by first making $M$ and $N$ large, and then taking $\epsilon$ small (depending on $M$ and $N$).
		
		Now, using the fact that $\partial_{\bar{z}}\hat{f}(z,\tau)=-2\pi \tau P_{\bar{z}}(z)\hat{f}(z,\tau)$ as distributions,
		\begin{align*}
			 \bar{D}_\tau \hat{f}_{\epsilon,M,N}& (z,\tau)\\
		    & =  \bar{D}_\tau (\psi_\epsilon \ast (\chi_{M,N}\hat{f})(z,\tau))\\
			& =  \int_{\mathbb{C}\times\mathbb{R}} [(\partial_{\bar{z}}\chi_{M,N})(z-\xi,\tau-s)\hat{f}(z-\xi,\tau-s)\\
			& \qquad\qquad +\chi_{M,N}(z-\xi,\tau-s)(\partial_{\bar{z}}\hat{f})(z-\xi,\tau-s)\\
			&  \qquad\qquad + \chi_{M,N}(z-\xi,\tau-s)\hat{f}(z-\xi,\tau-s)2\pi \tau P_{\bar{z}}(z)]\psi_\epsilon (\xi,s)dm(\xi,s)\\
			& =  \int_{\mathbb{C}\times\mathbb{R}} (\partial_{\bar{\xi}}\chi_{M,N})(\xi,s)\hat{f}(\xi,s)\psi_\epsilon(z-\xi,\tau-s)dm(\xi,s) \\
			& \quad +\int_{\mathbb{C}\times\mathbb{R}}[-\chi_{M,N}(z-\xi,\tau-s)2\pi(\tau-s)P_{\bar{z}}(z-\xi)\hat{f}(z-\xi,\tau-s)\\
			&  \qquad\qquad + \chi_{M,N}(z-\xi,\tau-s)\hat{f}(z-\xi,\tau-s)2\pi \tau P_{\bar{z}}(z)]\psi_\epsilon (\xi,s)dm(\xi,s)\\
			& =  \int_{\mathbb{C}\times\mathbb{R}} (\partial_{\bar{\xi}}\chi_{M,N})(\xi,s)\hat{f}(\xi,s)\psi_\epsilon(z-\xi,\tau-s)dm(\xi,s) \\
			&  \quad +\int_{\mathbb{C}\times \mathbb{R}} \hat{f}(\xi,s) \chi_{M,N}(\xi,s)2\pi (\tau P_{\bar{z}}(z)-sP_{\bar{z}}(\xi))\psi_\epsilon(z-\xi,\tau-s)dm(\xi,s)\\
			& =  I(z,\tau)+II(z,\tau).
		\end{align*}
		By our assumptions on $\chi_{M,N}$, Young's inequality, and the fact that $\sigma_\tau(z)\lesssim \tau^{-\frac{1}{2}}+\tau^{-\frac{1}{m}}$, for fixed $M$ we have 
		$$\|A(\bullet)I\|_{L^2(\mathbb{C}\times\mathbb{R})}\leq (M^{\frac{1}{m}}+M^{\frac{1}{2}})\|\hat{f}\|_{L^2(\{N\leq |z|\leq 2N\}\times\mathbb{R})} \to 0\ \mbox{as}\ N\to {+\infty}.$$
		Defining $$E_{\epsilon,M,N}= \sup_{|z|\leq 2N,\ (2M)^{-1}\leq \tau\leq 2M}\ \sup_{\max(\sqrt{|z-\xi|},|\tau-s|)\leq \epsilon} |2\pi \tau P_{\bar{z}}(z)-2\pi s P_{\bar{z}}(\xi)|,$$
		we see that for any large $M,N$ we may choose $\epsilon$ sufficiently small so that $E_{\epsilon,M,N}$ is as small as we'd like.
		In other words, for fixed $M,N$ large we have
		$$\|A(\bullet)II\|_2\leq (M^{\frac{1}{m}}+M^{\frac{1}{2}})\|\hat{f}\|_2 E_{\epsilon,M,N} \to 0\ \mbox{as}\ \epsilon\to 0.$$
		Therefore, choosing first $M=M(n)$ large, and then $N=N(n)$ large (depending on $M$), and then $\epsilon=\epsilon(n)$ small (depending on $M$ and $N$), we obtain a sequence of Schwartz functions $f_n=f_{\epsilon(n),M(n),N(n)}$ with $\hat{f}_n\in C_c^\infty(\mathbb{C}\times\mathbb{R})$ which, by Plancherel, satisfy $f_n\to f$ in $L^2$, $\|A(\bullet)\bar{D}_\bullet \hat{f}_n\|_2\to 0$ as $n\to {+\infty}.$
	\end{proof}
	
	We now prove Proposition \ref{prop:continuity}.
	
	\begin{proof}[Proof of Proposition \ref{prop:continuity}.]
		By Lemma \ref{lem:geom-uftisaq}, Lemma \ref{lem:metric-approxsmall}, and (H1), for fixed $\tau>0$ there exists constants $c,C$ with $$0<c\leq \sigma_{\tau+h}(z)\leq C<+\infty,\quad\mbox{uniformly in}\ w\in\mathbb{C},\ |h|<\frac{\tau}{2}.$$
		In other words we have
		\begin{equation}\label{eq:link-l2bound} \|\mathbb{H}_{\tau+h}f\|_{L^2(\mathbb{C})}\lesssim \|f\|_{L^2(\mathbb{C})}\quad\mbox{uniformly in}\ f\in L^2(\mathbb{C}),\ |h|<\frac{\tau}{2}.\end{equation}
		
		Because the operators $\mathbb{H}_{\tau+h}$ are uniformly bounded as operators on $L^2(\mathbb{C})$, the statement of the proposition follows once we show that $\mathbb{H}_{\tau+h}f\to \mathbb{H}_\tau f$ in $L^2(\mathbb{C})$ for all $f\in \mathscr{S}(\mathbb{C})$.
		
		To this end, we first claim that if $f$ is such that for each $N\geq 1$ there is a constant $C_N$ satisfying $|f(z)|\leq C_N (1+|z|)^{-N}$, then for every $M\geq 1$ there is a constant $A_M$ with \begin{equation}\label{eq:link-fastdecay}|\mathbb{H}_{\tau+h}[f](z)|\leq A_M (1+|z|)^{-M},\quad |h|<\frac{\tau}{2}.\end{equation}
		To see this, note that the estimates in Theorem \ref{thm:chr-rec-Christest} and Lemma \ref{lem:chr-rec-metest} provide small numbers $\epsilon,\delta>0$ with
		$$|[\mathbb{H}_{\tau+h}](z,w)|\lesssim \Big(1+\frac{1}{|z-w|}\Big) e^{-\epsilon |z-w|^\delta},\quad z,w\in\mathbb{C},\ |h|<\frac{\tau}{2}.$$
		In other words, for each $M\geq 1$ we have
		$$|[\mathbb{H}_{\tau+h}](z,w)|\lesssim \frac{1}{|z-w|(1+|z-w|)^M}.$$
		Because $|z|^{-1}\in L^1_{loc}(\mathbb{C})$ and is bounded for $|z|\geq 1$, the inequality $1+|z|\leq (1+|z-w|)(1+|w|)$ yields (\ref{eq:link-fastdecay}).
		
		Writing $\phi_h(z)=\mathbb{H}_{\tau+h}f(z)-\mathbb{H}_{\tau}f(z)$ for $f\in \mathscr{S}(\mathbb{C})$, we apply (\ref{eq:link-fastdecay}), Proposition \ref{prop:kernels-alg-derivs}, and the fact that $D_{\tau+h}=D_\tau + 2\pi h P_z,\ \bar{D}_{\tau+h}=\bar{D}_\tau + 2\pi h P_{\bar{z}}$ to see that for all $M\geq 1$ there exists a constant $A_M$ with
		\begin{equation}\label{eq:link-decaycond} |\phi_h(z)|+|\bar{D}_\tau \phi_h (z)|+|D_\tau \phi_h(z)|\leq A_M (1+|z|)^{-M},\qquad z\in\mathbb{C},\ |h|<\frac{\tau}{2}.\end{equation}
		
		We now prove the proposition for $\mathbb{G}_\tau$.
		Letting $\langle \bullet,\bullet\rangle$ denote the standard inner product on $L^2(\mathbb{C})$, and fixing $f\in \mathscr{S}(\mathbb{C})$ and $\phi_h(z)=\mathbb{G}_{\tau+h}f(z)-\mathbb{G}_{\tau}f(z)$, we have
		\begin{align}
		\lan \mathbb{G}_{\tau}f,\phi_h\ran-\lan \mathbb{G}_{\tau+h}f,\phi_h\ran
		& =  \lan f,\mathbb{G}_{\tau}\phi_h\ran - \lan \mathbb{G}_{\tau+h}f,\phi_h\ran \nonumber\\
		& =  \lan \bar{D}_{\tau+h} D_{\tau+h}\mathbb{G}_{\tau+h}f,\mathbb{G}_{\tau}\phi_h\ran - \lan \mathbb{G}_{\tau+h}f,\phi_h\ran \nonumber \\
		& =  \lan \mathbb{G}_{\tau+h}f, (\bar{D}_{\tau+h} D_{\tau+h}-\bar{D}_\tau D_\tau)\mathbb{G}_\tau \phi_h\ran.\label{eq:link-rewrite}
		\end{align}
		Because $$\bar{D}_{\tau+h} D_{\tau+h}-\bar{D}_\tau D_\tau=2\pi h P_{\bar{z}} D_\tau + 4\pi^2 h^2 |P_z|^2 + 2\pi h P_{z,\bar{z}}+2\pi h P_z \bar{D}_\tau,$$
		Proposition \ref{prop:kernels-alg-derivs}-(a) gives 
		\begin{align}
		(\bar{D}_{\tau+h} D_{\tau+h}-\bar{D}_\tau D_\tau)\mathbb{G}_\tau \phi_h & =  2\pi h P_{\bar{z}}\mathbb{R}_\tau \phi_h + (4 \pi^2 h^2 |P_z|^2+2\pi h P_{z,\bar{z}})\mathbb{G}_\tau \phi_h \nonumber\\
		& \quad \qquad + 2\pi h P_z \mathbb{R}^\ast_\tau \phi_h - 2\pi h P_z\mathbb{R}^\ast_\tau (4\pi \tau P_{z,\bar{z}})\mathbb{G}_\tau\phi_h.\label{eq:link-adjoint}
		\end{align}	
		
		Because $P=P^{\mb{\sigma},\kappa}$ for some $\kappa$, there exists $N\geq 1$ so that $$|\nabla P(z)|+|\Delta P(z)|\lesssim (1+|z|)^N,$$
		and therefore (\ref{eq:link-fastdecay}) (applied to $\phi_h$) implies that $(\bar{D}_{\tau+h} D_{\tau+h}-\bar{D}_\tau D_\tau)\mathbb{G}_\tau \phi_h\in L^2(\mathbb{C})$ for $|h|<\frac{\tau}{2}$. 
		Then (\ref{eq:link-rewrite}), the Schwarz inequality, (\ref{eq:link-l2bound}), (\ref{eq:link-adjoint}), and the Dominated Convergence Theorem yield, for $|h|<\frac{\tau}{2}$,
		\begin{align*}
			\|\mathbb{G}_\tau f - \mathbb{G}_{\tau+h}f\|_2 & =  |	\lan \mathbb{G}_{\tau}f,\phi_h\ran-\lan \mathbb{G}_{\tau+h}f,\phi_h\ran| \\
			& \leq   \|\mathbb{G}_{\tau+h}f\|_{2}\|(\bar{D}_{\tau+h} D_{\tau+h}-\bar{D}_\tau D_\tau)\mathbb{G}_\tau \phi_h\|_{2} \\
			& \lesssim  \|f\|_2 \|(\bar{D}_{\tau+h} D_{\tau+h}-\bar{D}_\tau D_\tau)\mathbb{G}_\tau \phi_h\|_{2} \\
			& \to 0
		\end{align*}
		as $h\to 0$.
		This shows that $\mathbb{G}_{\tau+h}f\to \mathbb{G}_\tau f$ in $L^2(\mathbb{C})$ as $h\to 0$.
		
		A similar argument shows that if $\phi_h=\mathbb{R}_{\tau+h}f - \mathbb{R}_\tau f$, then
		$$\|\mathbb{R}_{\tau+h}f-\mathbb{R}_\tau f\|_2 = |\lan \mathbb{R}_{\tau+h}f ,\phi_h\ran-\lan \mathbb{R}_\tau f,\phi_h\ran | = |\lan \mathbb{R}_{\tau + h} f,2\pi h P_z \mathbb{R}^\ast_\tau \phi_h\ran|\to 0$$
		as $h\to 0,$ and therefore $\mathbb{R}_{\tau+h}f\to \mathbb{R}_\tau$ as $h\to 0$ in $L^2(\mathbb{C})$.
		
		For $\mathbb{R}^\ast_\tau$ and $\phi_h=\mathbb{R}^\ast_{\tau+h}f-\mathbb{R}^\ast_\tau f$, the analogous statement follows by noting that 
		$$\lan \mathbb{R}^\ast_{\tau+h}f,\phi_h\ran - \lan \mathbb{R}^\ast_\tau f,\phi_h\ran = \lan f,\mathbb{R}_{\tau+h}\phi_h\ran - \lan f,\mathbb{R}_\tau \phi_h\ran \to 0\ \mbox{as}\ h\to 0,$$
		while that for $\mathbb{S}_\tau$ and $\phi_h=\mathbb{S}_\tau f - \mathbb{S}_{\tau+h}f$ follows from
		\begin{align*}
		\lan \mathbb{S}_{\tau+h} f, \phi_h\ran & - \lan \mathbb{S}_\tau f,\phi_h\ran	\\
		& =  \lan - D_{\tau + h} \mathbb{R}^\ast_{\tau+h}f, \phi_h\ran - \lan -D_\tau \mathbb{R}^\ast_\tau f, \phi_h\ran  \\
			& = \lan \mathbb{R}_\tau^\ast f, \bar{D}_\tau \phi_h\ran - \lan \mathbb{R}^\ast_{\tau+h} f,\bar{D}_{\tau+h}\phi_h\ran  \\
			& = [\lan \mathbb{R}_\tau^\ast f, \bar{D}_\tau \phi_h\ran - \lan \mathbb{R}^\ast_{\tau+h} f,\bar{D}_{\tau}\phi_h\ran] - \lan \mathbb{R}_{\tau+h}^\ast f , 2\pi h P_{\bar{z}}\phi_h\ran \\
			& \to  0
		\end{align*}
		as $h\to 0$. This completes the proof.
	\end{proof}
	
	We move on to the proof of Proposition \ref{prop:link}. 
	In order to even make sense out of the expression $$\int_{0}^{+\infty} e^{2\pi i \tau t}\int_{\mathbb{C}} [\mathbb{S}_\tau](z,w)\hat{f}(w,\tau)dm(w)d\tau$$
	appearing in the proof, we need to have the following elementary result.
	
	\begin{lemma}\label{lem:link-cont} Let $g(z,\tau)\in \mathscr{S}(\mathbb{C}\times \mathbb{R})$, and define $$S(z,\tau)=\mathbb{S}_\tau[g(\bullet,\tau)](z),\quad z\in\mathbb{C},\ \tau\in (0,+\infty).$$
		Then $S(z,\tau)\in C^0(\mathbb{C}\times (0,+\infty))\cap L^2(\mathbb{C}\times (0,+\infty))$.
		\end{lemma}
		\begin{proof}
			Note that $S(z,\tau)\in L^2(\mathbb{C})$ for each fixed $\tau$, and that because $\bar{D}_\tau S(z,\tau)=0$ as distributions, we have
			$$\partial_{\bar{z}} \Big(e^{2\pi \tau P(z)}S(z,\tau)\Big)\equiv 0,$$
			and therefore the function $z\mapsto \tilde{S}_\tau(z)=e^{2\pi \tau P(z)}S(z,\tau)$ is entire for each fixed $\tau$.
			In particular, for fixed $\tau$ the function $z\mapsto S(z,\tau)$ agrees almost everywhere with a continuous function, and we can therefore take $S(z,\tau)$ to be continuous in $z$.
			We claim that for $\tau$ fixed, $S(z,\tau+h)\to S(z,\tau)$ as $h\to 0$, uniformly on compact subsets of $\mathbb{C}$, which immediately implies that $S(z,\tau)$ is continuous on $\mathbb{C}\times (0,+\infty)$.
			
			We prove the claim by first noting that the arguments in the proof of Proposition \ref{prop:continuity} can be extended to show that, for $\tau$ fixed, $$S(z,\tau+h)\to S(z,\tau)\ \mbox{in}\ L^2(\mathbb{C})\ \mbox{as}\ h\to 0.$$
			The continuity of $P(z)$ then implies that for every compact set $K\Subset \mathbb{C}$ we have $$\tilde{S}_{\tau+h}(z)\to \tilde{S}_\tau (z)\ \mbox{in}\ L^2(K)\ \mbox{as}\ h\to 0.$$
			Define $K(N)=\{ z\in \mathbb{C}\ :\ |z|\leq N\}$. We apply the Cauchy integral formula and the Schwarz inequality to see that for $z\in K(N)$, 
			\begin{align*}
				|\tilde{S}_{\tau+h}(z)-\tilde{S}_{\tau}(z)| & = \frac{1}{2\pi} \Big| \iint_{1\leq |\eta|\leq 2} \frac{\tilde{S}_{\tau+h}(z+\eta)-\tilde{S}_{\tau}(z+\eta)}{|\eta|}dm(\eta)\Big| \\
				& \leq  \frac{1}{2\pi} \iint_{1\leq |\eta|\leq 2} |\tilde{S}_{\tau+h}(z+\eta)-\tilde{S}_{\tau}(z+\eta)|dm(\eta) \\
				& \leq  \frac{\sqrt{3\pi}}{2\pi}\|\tilde{S}_{\tau+h}-\tilde{S}_\tau\|_{L^2(K(N+2))} \\
				& \to   0 
			\end{align*}
			as $h\to 0$. This completes the proof that $S(z,\tau)\in C^0(\mathbb{C}\times (0,+\infty))$. The claim that $S(z,\tau)\in L^2(\mathbb{C}\times (0,+\infty))$ is then an immediate consequence of Lemma \ref{lem:unifL2bound} and Fubini-Tonelli.
		\end{proof}
		
	We finally prove Proposition \ref{prop:link}. Define $$B(\mathbb{C}\times\mathbb{R})=\{f\in L^2(\mathbb{C}\times\mathbb{R})\ :\ \bar{Z}f=0\ \mbox{as tempered distributions}\}.$$	
		
	\begin{proof}[Proof of Proposition \ref{prop:link}.]
		
		For $f\in \mathscr{S}(\mathbb{C}\times\mathbb{R})$, define $$\mathbb{T}[f](z,t)=\int_0^{+\infty} e^{2\pi i \tau t} \int_\mathbb{C} [\mathbb{S}_\tau](z,w)\hat{f}(w,\tau)dm(w)d\tau.$$ 
		This expression is well-defined by Lemma \ref{lem:link-cont}, and by Lemma \ref{lem:unifL2bound} we have \newline $\|\mathbb{T}[f]\|_2\leq \|f\|_2$ and $\mathbb{T}[f]\in B(\mathbb{C}\times\mathbb{R})$, and therefore $\mathbb{T}$ extends to a bounded operator from $L^2(\mathbb{C}\times\mathbb{R})$ into $B(\mathbb{C}\times\mathbb{R})$ with norm $\leq 1$.
		
		We claim that $\mathbb{T}=\mathbb{S}$.
		To prove this, it suffices to show that if we write $f=f_{\|}+f_{\perp}$, with $f_{\|}\in B(\mathbb{C}\times\mathbb{R})$ and $f_{\perp}\perp B(\mathbb{C}\times\mathbb{R})$, then $\mathbb{T}[f]=f_{\|}.$
		
		Choosing $f_n\to f_{\|}$ as in Lemma \ref{lem:smoothapprox}, and writing $\mathbb{S}_\tau = I-\mathbb{R}_\tau\bar{D}_\tau$, we have
		\begin{align*}
			\mathbb{T}[f_{\|}] & = \lim_{n\to {+\infty}} \mathbb{T}[f_n] \\
			& = \lim_{n\to {+\infty}} f_n - \Big[\mathbb{R}_\bullet[\bar{D}_\bullet \hat{f}_n]\Big]^{\vee}.
		\end{align*}
		Now, $f_n\to f_{\|}$ in $L^2(\mathbb{C}\times\mathbb{R})$, while Lemma \ref{lem:unifL2bound} and Plancherel give us
		$$\|\Big[\mathbb{R}_\bullet[\bar{D}_\bullet \hat{f}_n]\Big]^{\vee}\|_2= \|\mathbb{R}_\bullet[\bar{D}_\bullet \hat{f}_n]\|_2 \leq C\|A(\bullet)\bar{D}_\bullet \hat{f}_n\|_2 \to 0,$$
		which shows that $\mathbb{T}[f_{\|}]=f_{\|}$.
		
		On the other hand one sees that $\mathbb{T}$ is self adjoint, so that for every $h\in L^2(\mathbb{C}\times\mathbb{R})$ we have
		$$\lan \mathbb{T}[f_\perp],h\ran = \lan f_\perp,\mathbb{T}[h]\ran = 0,$$ and therefore $\mathbb{T}[f_\perp]=0$.
		
		This completes the proof.
	\end{proof}

\end{appendix}

\bibliographystyle{siam}


\noindent Aaron Peterson\\
Department of Mathematics\\
Northwestern University\\
2033 Sheridan Road\\
Evanston, IL  60208-2730, USA\\
e-mail: aaron@math.northwestern.edu

\end{document}